\documentclass[3p,times]{elsarticle}

\usepackage[utf8]{inputenc}
\usepackage[T1]{fontenc}
\usepackage[english]{babel}
\usepackage{amsmath,amssymb,amsfonts,amsthm,mathtools}
\usepackage{bm}
\usepackage{mathrsfs}
\usepackage{enumitem}
\usepackage{graphicx}
\usepackage{booktabs}
\usepackage{placeins}
\usepackage{microtype}
\usepackage{siunitx}
\usepackage{pgfplots}
\usetikzlibrary{calc,arrows.meta}
\pgfplotsset{compat=1.18}
\usepackage[colorlinks=true,linkcolor=blue,citecolor=blue,urlcolor=blue]{hyperref}
\hypersetup{
  pdftitle={Exact discrete-adjoint optimization of trap timing and placement in a Stieltjes-time reaction-diffusion model: A Galicia case study},
  pdfauthor={Francisco J. Fernández; Iván Area},
  pdfsubject={Fully discrete Stieltjes-time optimal control and Galicia case study},
  pdfkeywords={Stieltjes time, discrete adjoint, reaction-diffusion model, optimal control, invasive species, finite elements, trap placement}
}
\journal{Applied Mathematical Modelling}
\biboptions{sort&compress}
\numberwithin{equation}{section}

\newtheorem{theorem}{Theorem}[section]
\newtheorem{proposition}[theorem]{Proposition}
\newtheorem{lemma}[theorem]{Lemma}
\newtheorem{assumption}[theorem]{Assumption}
\newtheorem{remark}[theorem]{Remark}
\newtheorem{definition}[theorem]{Definition}

\newcommand{\R}{\mathbb{R}}
\newcommand{\N}{\mathbb{N}}
\newcommand{\Uad}{\mathcal U_{\rm ad}}

\newcommand{\calF}{\mathcal F}
\newcommand{\calJ}{\mathcal J}
\newcommand{\calA}{\mathcal A}
\newcommand{\dd}{\,\mathrm d}

\newcommand{\Dg}{\mathrm D_g}
\newcommand{\one}{\mathbf 1}

\begin{document}
\begin{frontmatter}

\title{Exact discrete-adjoint optimization of trap timing and placement in a Stieltjes-time reaction--diffusion model: A Galicia case study}

\author[usc]{Francisco J. Fern\'andez\corref{cor1}}
\ead{fjavier.fernandez@usc.es}
\cortext[cor1]{Corresponding author.}

\author[uvigo]{Iv\'an Area}
\ead{area@uvigo.gal}

\affiliation[usc]{organization={CITMAga (Centro de Investigaci\'on e Tecnolox\'ia Matem\'atica de Galicia), Departamento de Estad\'istica, An\'alisis Matem\'atico y Optimizaci\'on, Facultade de Matem\'aticas, Universidade de Santiago de Compostela},
 city={Santiago de Compostela}, postcode={15782}, country={Spain}}
\affiliation[uvigo]{organization={IFCAE (Instituto de Investigaci\'on en F\'isica, Computaci\'on e Ciencia Aeroespacial), Departamento de Matem\'atica Aplicada II, Universidade de Vigo},
 addressline={Campus As Lagoas}, city={Ourense}, postcode={E-32004}, country={Spain}}

\begin{abstract}
Seasonal population models must often combine diffusion, inactive periods and abrupt biological transfers while remaining differentiable with respect to localized interventions. We formulate a finite-dimensional timing-and-placement control problem for a Stieltjes-time reaction--diffusion model and derive an exact adjoint of the fully discrete residual. Continuous finite elements and an implicit Stieltjes--Euler scheme encode propagation, compartment replacement, resets and historical phase averages. For the resulting box-constrained penalized problem, we establish well-posedness and differentiability of the discrete control-to-state map, existence of a minimizer, first-order stationarity conditions, and equivalence of direct-sensitivity and adjoint gradients. Independent benchmark tests give relative discrepancies at the level of machine precision (below $3\times10^{-15}$) and measured adjoint speed-ups from about $7.5$ to $30.9$ as the number of traps increases from one to four. The method is then applied to trap campaigns for \textit{Vespa velutina} on a realistic Galicia mesh. Smoothed activation windows, domain-normalized moving kernels and the complete adjoint are verified independently before optimization. The final two-cluster control has a scaled gradient infinity norm of $2.97\times10^{-5}$ at 120 temporal subdivisions per month and remains admissible. Fixed-control evaluations at 60, 120 and 240 subdivisions yield a refinement-increment ratio of $0.4984$. At the finest level, the joint control outperforms time-only and reference controls and reduces the diagnostic objective by $0.49209\%$. This absolute reduction is conditional on the uncalibrated trap-to-mortality intensity constants and must not be interpreted as field capture efficacy. The study provides a reproducible discrete-adjoint framework for seasonal intervention models, while not claiming global optimality or a calibrated field-management prescription.
\end{abstract}

\begin{keyword}
Stieltjes time \sep discrete adjoint \sep reaction--diffusion model \sep optimal control \sep invasive species \sep finite elements \sep trap placement
\end{keyword}

\end{frontmatter}
\section{Introduction}

The yellow-legged hornet, \textit{Vespa velutina}, is an invasive species whose expansion has become a significant ecological and economic problem in several European regions. Mathematical models can be useful to describe its spatial spread, to incorporate seasonal biological mechanisms and to compare possible intervention strategies.

A Stieltjes-time reaction--diffusion model was previously proposed to incorporate continuous diffusion periods, inactive periods and impulsive biological events, such as emergence, hibernation, production of workers and production of future foundress queens~\cite{AreaVelutina2025}. The purpose of the present work is to build on that model and study a control problem associated with the timing and location of traps. The mathematical setting is related to the Stieltjes differential calculus developed for systems with continuous, discrete and impulsive dynamics~\cite{PousoRodriguez2015,FrigonLopezPouso2017}. For completeness, Section~\ref{sec:stieltjes-preliminaries} develops the required derivator, measure, pointwise and weak-derivative, and function-space notions from published sources, and explains how flat and atomic parts of the intrinsic clocks represent dormancy and abrupt biological transitions.

The key modelling assumption is that each trap produces a localized increase of the mortality coefficient. If the number of traps, their intensity and their spatial and temporal radii are fixed, the control variable consists only of the activation times and spatial positions of the traps. This leads to a finite-dimensional optimal control problem constrained by a Stieltjes-time parabolic model. Equivalently, the present control can be viewed as a finite-dimensional subfamily of bilinear removal controls, where the removal rate is not arbitrary but generated by localized trapping kernels. Related reaction--diffusion control models for invasive species motivate the intervention viewpoint~\cite{Bonneau2017}, while the residual, sensitivity, and adjoint constructions follow standard finite-dimensional and PDE-constrained optimization principles~\cite{BonnansShapiro2000,HinzeEtAl2009,NocedalWright2006}.

\subsection*{Main contributions of the paper}
The contributions of the paper can be summarized as follows. \emph{(C1) Stieltjes modelling and analysis:} we provide the derivator, Lebesgue--Stieltjes measure, absolute-continuity, Stieltjes--Bochner, weak-evolution, and biological-transfer framework required by the seasonal reaction--diffusion model. \emph{(C2) Fully discrete optimal control:} we formulate the finite-dimensional timing--placement problem, write the finite-element/Stieltjes--Euler state equation as an algebraic residual, prove existence for the fully discrete problem, and state its box-constrained first-order conditions. \emph{(C3) Gradient construction and academic verification:} we derive linearized-state and exact discrete-adjoint gradients, validate them independently, quantify the adjoint speed-up, and study an academic problem whose multi-start results are consistent with nonconvex behaviour through 32 IPOPT runs. \emph{(C4) Galicia state evaluation and numerical screening:} we define explicit observation operators and baselines for a held-out temporal comparison, and use positivity, mesh, time, radius, allocation, activation-window, and compartment-mask tests to freeze a defensible field-domain discretization without modifying the original biological transfers. \emph{(C5) Galicia adjoint optimization and refinement:} we validate the exact adjoint of the lumped saturating cluster residual, perform local timing--placement optimization on L1 and L2, verify small scaled projected residuals, and document a first-order temporal refinement pattern for the final fixed control through $60$, $120$, and $240$ subdivisions per month.

\subsection*{Scope of the mathematical claims}
To prevent the modelling and computational levels from being conflated, Table~\ref{tab:claim-scope} records the status of the principal statements.
\begin{table}[htbp]
\centering
\small
\caption{Scope of the analytical and numerical claims.}
\label{tab:claim-scope}
\begin{tabular}{p{0.22\linewidth}p{0.34\linewidth}p{0.35\linewidth}}
\toprule
Level & Established in this paper & Not claimed \\
\midrule
Continuous Stieltjes model & Consistent weak formulation and atomic resolvent identities & Well-posedness or differentiability of the continuous control-to-state map \\
Fully discrete state & Existence, uniqueness, and $C^1$ dependence on the finite-dimensional control & Convergence of the discrete state to a continuous solution \\
Discrete optimization & Existence for the box-constrained penalized problem; necessary box-stationarity conditions; exact discrete adjoint identity & Convexity, uniqueness, or global optimality \\
Numerical refinement & Independent derivative checks and observed stability across the tested meshes and time grids & A theorem on mesh/time convergence of optimizers or objective values \\
Ecological interpretation & Reproducible scenario comparison for the stated parameters & Calibrated trap efficacy or an operational deployment recommendation \\
\bottomrule
\end{tabular}
\end{table}

\subsection*{Notation}
The symbols $g_i$ denote the compartment-specific derivators and $\mu_{g_i}$ their Lebesgue--Stieltjes measures. Natural mortality rates are denoted by $\mu_i$ and must not be confused with $\mu_{g_i}$. The continuous state is $Y=(y_1,y_2,y_3)$, the fully discrete state vector is $C$, and the control is denoted by $u$ in physical coordinates and by $\zeta$ after nondimensionalization. The reduced objectives are $j$ in the academic problem and $J_{\rm bio}$ in the Galicia study. In the numerical sections, L1 and L2 denote the exploratory and refined Galicia mesh levels, respectively; they are mesh labels and not Lebesgue spaces.

\section{Stieltjes differential calculus and functional setting}\label{sec:stieltjes-preliminaries}

This section collects the Stieltjes notions used later in the continuous model and fixes the conventions needed to interpret the discrete scheme. The presentation follows the differential--integral framework developed in~\cite{PousoRodriguez2015,FrigonLopezPouso2017}, its vector-valued and parabolic extension in~\cite{FernandezTojo2020Bochner}, and the more recent analysis of Stieltjes function spaces and pointwise differentiability in~\cite{FernandezTojoVillanueva2024,FernandezMarquezTojoVillanueva2025}. The objective is self-containment rather than a new development of Stieltjes calculus.

\subsection{Derivators and Lebesgue--Stieltjes measures}

\begin{definition}[Derivator]\label{def:derivator}
Let $I=[T_0,T_F]$. A \emph{derivator} is a nondecreasing, left-continuous map
\[
 g:I\longrightarrow\R.
\]
Unless stated otherwise, we assume that $g(T_F)>g(T_0)$, so that the associated clock is nontrivial.
Its right jump at $t\in[T_0,T_F)$ is
\[
 \Delta^+g(t):=g(t^+)-g(t).
\]
\end{definition}
This convention, introduced for Stieltjes differential equations in~\cite{PousoRodriguez2015,FrigonLopezPouso2017}, makes chronological time only one possible clock: the classical case is $g(t)=t$, while a more general $g$ measures the effective time experienced by the process.

Every derivator induces a finite positive Lebesgue--Stieltjes measure $\mu_g$ on $I$. With the left-continuous convention used here,
\begin{equation}\label{eq:stieltjes-measure-convention}
 \mu_g([c,d))=g(d)-g(c),
 \qquad T_0\le c<d\le T_F,
\end{equation}
and
\begin{equation}\label{eq:stieltjes-atom}
 \mu_g(\{t\})=\Delta^+g(t).
\end{equation}
Lebesgue--Stieltjes integration will be written either as $\int f\,\dd\mu_g$ or, when no ambiguity is possible, as $\int f\,\dd g$. Standard measure-theoretic constructions and the interval convention in~\eqref{eq:stieltjes-measure-convention} can be found in~\cite{CarterVanBrunt2000,PousoRodriguez2015}.

The derivators used in the benchmark have an absolutely continuous part and finitely many atoms. They can therefore be written as
\begin{equation}\label{eq:benchmark-derivator-decomposition}
 g(t)=g(T_0)+\int_{T_0}^{t}a_g(s)\,\dd s
 +\sum_{\{k:\,t_k<t\}}\delta_k,
 \qquad a_g\ge0,\quad \delta_k>0,
\end{equation}
so that
\begin{equation}\label{eq:benchmark-measure-decomposition}
 \dd\mu_g(t)=a_g(t)\,\dd t+\sum_k\delta_k\,\delta_{t_k}(\dd t).
\end{equation}
The general Lebesgue--Stieltjes measure may also have a singular-continuous part, but that component is not used in the biological calendar considered below.

\subsection{Notable sets associated with a derivator}

The geometry of $g$ is described by two primary sets~\cite{PousoRodriguez2015,FrigonLopezPouso2017,FernandezMarquezTojoVillanueva2025}:
\begin{align}
 D_g&:=\{t\in[T_0,T_F):\Delta^+g(t)>0\},\label{eq:Dg-definition}\\
 C_g&:=\bigl\{t\in I:\text{$g$ is constant on $(t-\varepsilon,t+\varepsilon)\cap I$
 for some $\varepsilon>0$}\bigr\}.\label{eq:Cg-definition}
\end{align}
The set $D_g$ is at most countable and is exactly the atomic support of $\mu_g$. The set $C_g$ is relatively open. Writing its pairwise disjoint connected components as
\begin{equation}\label{eq:Cg-components}
 C_g=\bigcup_{n\in\Lambda_g}(a_n,b_n),
\end{equation}
one also introduces the endpoint sets
\begin{equation}\label{eq:Ng-definition}
 N_g^-:=\{a_n:n\in\Lambda_g\}\setminus D_g,
 \qquad
 N_g^+:=\{b_n:n\in\Lambda_g\}\setminus D_g,
 \qquad
 N_g:=N_g^-\cup N_g^+.
\end{equation}
These endpoints are relevant in pointwise differentiability and compactness arguments. Moreover,
\[
 C_g\cap D_g=\varnothing,
 \qquad
 \mu_g(C_g\cup N_g)=0.
\]
Thus, the differential equation formulated $\mu_g$-almost everywhere does not depend on how the derivative is represented on the interiors or non-atomic endpoints of flat components.

\subsection{Stieltjes derivative}

Let $X$ be a Banach space and $u:I\to X$. At a point $t\notin C_g$, the Stieltjes derivative, or $g$-derivative, is defined, whenever the relevant limit exists, by~\cite{PousoRodriguez2015,FrigonLopezPouso2017,FernandezTojo2020Bochner}
\begin{equation}\label{eq:pointwise-g-derivative}
 u'_g(t)=
 \begin{cases}
 \displaystyle\lim_{s\to t,\,g(s)\ne g(t)}
 \frac{u(s)-u(t)}{g(s)-g(t)},
 &t\notin D_g\cup C_g,\\[2.2ex]
 \displaystyle\lim_{s\to t^+}
 \frac{u(s)-u(t)}{g(s)-g(t)},
 &t\in D_g.
 \end{cases}
\end{equation}
At an atom, existence of the right limit gives the exact identity
\begin{equation}\label{eq:g-derivative-at-jump}
 u'_g(t)=\frac{u(t^+)-u(t)}{\Delta^+g(t)},
 \qquad t\in D_g.
\end{equation}
If the usual derivatives $u'(t)$ and $g'(t)>0$ exist, then
\begin{equation}\label{eq:g-derivative-ratio}
 u'_g(t)=\frac{u'(t)}{g'(t)}.
\end{equation}
Consequently, $u'_g$ is a rate of change relative to the intrinsic clock $g$, rather than necessarily relative to chronological time.

The quotient in~\eqref{eq:pointwise-g-derivative} is indeterminate on $C_g$. For completeness, the extended pointwise convention studied in~\cite{FernandezMarquezTojoVillanueva2025} associates with $t\in(a_n,b_n)\subset C_g$ the point $t^*=b_n$ and defines
\begin{equation}\label{eq:extended-g-derivative-flat}
 u'_g(t):=\lim_{s\to (t^*)^+}
 \frac{u(s)-u(t^*)}{g(s)-g(t^*)},
\end{equation}
whenever the limit is meaningful and exists. The present PDE model does not rely on this extension: its weak derivative is defined through an integral representation and only $\mu_g$-almost everywhere. This distinction is important because pointwise $g$-differentiability alone has a nontrivial kernel and does not, without an appropriate function-space condition, guarantee the classical uniqueness property ``zero derivative implies constant''~\cite{FernandezMarquezTojoVillanueva2025}.

\subsection{\texorpdfstring{$g$}{g}-continuity, absolute continuity and the fundamental theorem}

The derivator induces the pseudometric
\[
 d_g(s,t):=|g(s)-g(t)|
\]
and its associated topology $\tau_g$. A function $u:I\to X$ is \emph{$g$-continuous} if it is continuous from $(I,\tau_g)$ to $X$; equivalently, for every $t\in I$ and $\varepsilon>0$ there exists $\delta>0$ such that
\[
 |g(s)-g(t)|<\delta\quad\Longrightarrow\quad
 \|u(s)-u(t)\|_X<\varepsilon.
\]
We denote by $BC_g(I;X)$ the Banach space of bounded $g$-continuous maps equipped with the supremum norm. These notions and their compactness properties are developed in~\cite{FrigonLopezPouso2017,FernandezTojoVillanueva2024}.

Assume that $X$ has the Radon--Nikod\'ym property (in particular, $X$ may be reflexive). A map $u:I\to X$ is $g$-absolutely continuous, written $u\in\mathcal{AC}_g(I;X)$, when it admits the equivalent integral representation supplied by the Stieltjes fundamental theorem of calculus~\cite{PousoRodriguez2015,FernandezTojo2020Bochner}:
\begin{equation}\label{eq:g-ac-definition}
 u(t)=u(T_0)+\int_{[T_0,t)}v(s)\,\dd\mu_g(s),
 \qquad t\in I,
\end{equation}
for some $v\in L_g^1(I;X)$. The function $v$ is unique $\mu_g$-almost everywhere and is denoted by $\Dg u$. Conversely, $\Dg u$ exists $\mu_g$-almost everywhere, belongs to $L_g^1(I;X)$, and reconstructs $u$ through~\eqref{eq:g-ac-definition}. In particular, the canonical representative satisfies
\begin{equation}\label{eq:g-ac-jump-identity}
 u(t^+)-u(t)=\Dg u(t)\,\Delta^+g(t),
 \qquad t\in D_g.
\end{equation}
For a derivator of the form~\eqref{eq:benchmark-derivator-decomposition}, the integral splits as
\begin{equation}\label{eq:stieltjes-integral-split}
 \int_{[T_0,t)}v\,\dd\mu_g
 =\int_{T_0}^{t}v(s)a_g(s)\,\dd s
 +\sum_{\{k:t_k<t\}}v(t_k)\delta_k.
\end{equation}
Equation~\eqref{eq:stieltjes-integral-split} displays explicitly how continuous evolution and prescribed impulses coexist in one integral equation.

\subsection{Stieltjes--Bochner and energy spaces}

For a Banach space $X$ and $1\le p\le\infty$, let
\begin{equation}\label{eq:Lpg-definition}
 L_g^p(I;X):=L^p(I,\mu_g;X)
\end{equation}
be the Bochner space of strongly $\mu_g$-measurable maps, with the usual modification when $p=\infty$. Following~\cite{FernandezTojo2020Bochner}, for a Gelfand triple $V\hookrightarrow H\equiv H'\hookrightarrow V'$ define
\begin{equation}\label{eq:Wgpq-definition}
 W_g^{1,p,q}(I;V,V')
 :=\left\{u\in L_g^p(I;V):
 \begin{array}{l}
 \text{there exists }z\in L_g^q(I;V')\text{ such that}\\[-0.2em]
 u(t)=u(T_0)+\displaystyle\int_{[T_0,t)}z(s)\,\dd\mu_g(s)
 \text{ in }V',\ \forall t\in I
 \end{array}\right\}.
\end{equation}
We write $\Dg u=z$ and use the graph norm
\begin{equation}\label{eq:Wgpq-norm}
 \|u\|_{W_g^{1,p,q}}
 :=\|u\|_{L_g^p(I;V)}+\|\Dg u\|_{L_g^q(I;V')}.
\end{equation}
The natural parabolic energy space used in the modelling-level formulation is
\begin{equation}\label{eq:stieltjes-energy-space}
 \mathcal W_g
 :=W_g^{1,2,2}(I;V,V')\cap BC_g(I;H),
 \qquad
 \mathcal W:=\mathcal W_{g_1}\times\mathcal W_{g_2}\times\mathcal W_{g_3}.
\end{equation}
The $BC_g(I;H)$ representative supplies meaningful left values and post-atomic right limits, while the two Bochner terms encode spatial energy and the weak Stieltjes derivative. Existence theory and compactness results for such spaces are given in~\cite{FernandezTojo2020Bochner,FernandezTojoVillanueva2024}. In the classical case $g(t)=t$, this reduces to the familiar Lions space with $u\in L^2(I;V)$ and $u'\in L^2(I;V')$, together with its $H$-continuous representative.

\subsection{Why Stieltjes time is suited to seasonal biological processes}\label{subsec:why-stieltjes-biology}

A derivator can encode qualitatively different biological regimes without switching between unrelated equations. In~\eqref{eq:benchmark-derivator-decomposition}:
\begin{itemize}
 \item $a_g(t)>0$ represents ordinary continuous biological activity, with the rate measured relative to the effective clock;
 \item $a_g(t)=0$ on an interval gives $\mu_g$-mass zero there, so a $g$-absolutely continuous state is constant on that inactive component; this naturally represents diapause, hibernation or another period in which the compartment does not evolve;
 \item an atom $\delta_k\delta_{t_k}$ converts the differential equation into the jump relation~\eqref{eq:g-ac-jump-identity}, allowing emergence, maturation, production, mortality pulses or compartmental resets to be prescribed at a known biological time;
 \item different derivators $g_i$ allow different compartments to have asynchronous active, dormant and impulsive phases.
\end{itemize}
This continuous--inactive--impulsive unification is a central motivation of Stieltjes differential equations~\cite{PousoRodriguez2015,FrigonLopezPouso2017}. Its use in mathematical biology and in systems with several derivators is analysed in~\cite{LopezPousoMarquez2018,LopezPousoMarquez2019}; the corresponding parabolic framework, including periods with no Stieltjes-time evolution, is developed in~\cite{FernandezTojo2020Bochner}; and the specific calendar of \textit{Vespa velutina} is taken from~\cite{AreaVelutina2025}.

\begin{figure}[htbp]
\centering
\begin{tikzpicture}
\begin{axis}[
 width=0.78\linewidth,height=5.0cm,
 axis lines=left,
 xlabel={chronological time $t$},ylabel={intrinsic time $g(t)$},
 xmin=0,xmax=6.4,ymin=0,ymax=5.4,
 xtick={0,2,4,6},ytick={0,2,3,5},
 clip=false]
\addplot[thick,black] coordinates {(0,0) (2,2) (4,2)};
\addplot[thick,black] coordinates {(4.02,3) (6,5)};
\addplot[dashed,black] coordinates {(4,2) (4,3)};
\addplot[only marks,mark=*,black] coordinates {(4,2)};
\node[anchor=south] at (axis cs:1,1) {active phase};
\node[anchor=south] at (axis cs:3,2) {dormancy};
\node[anchor=west,align=left] at (axis cs:4.08,2.5) {atomic\\transition};
\end{axis}
\end{tikzpicture}
\caption{Schematic left-continuous derivator combining active evolution, a flat inactive period and a right jump. The associated measure has continuous mass on the increasing pieces, zero mass on the flat piece and an atom at the jump; see~\cite{PousoRodriguez2015,FernandezTojo2020Bochner}.}
\label{fig:schematic-derivator}
\end{figure}
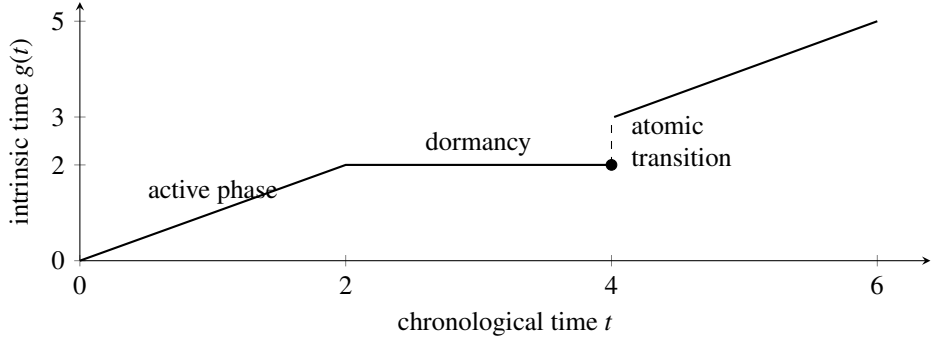

\begin{remark}[Chronological and biological time]
A flat part of $g$ does not remove the corresponding chronological dates from the model. It states that no effective time elapses for that compartment. Likewise, an atom has zero chronological duration but positive $\mu_g$-mass. Thus, dormancy and abrupt transitions are represented by the clock itself rather than by ad hoc case distinctions in every equation.
\end{remark}

\subsection{Spatial setting and state variables}

Let $\Omega\subset\R^2$ be a bounded polygonal domain and let
\[
 V:=H^1(\Omega)\hookrightarrow H:=L^2(\Omega)\equiv H'\hookrightarrow V':=H^1(\Omega)'.
\]
The duality between $V'$ and $V$ is denoted by $\langle\cdot,\cdot\rangle_{V',V}$, and the inner product in $H$ by $(\cdot,\cdot)_H$. The biological calendar is encoded by three derivators $g_i$, one for each state variable,
\begin{align*}
 y_1(t,x)&=\text{foundress queens},\\
 y_2(t,x)&=\text{future foundress queens},\\
 y_3(t,x)&=\text{workers and males}.
\end{align*}
Let $\nu_i\ge0$ be the diffusion coefficient of compartment $i$. The numerical benchmark takes $\nu_1=\nu_2=\nu_3=\nu$. Homogeneous Neumann conditions,
\[
 \partial_{\mathbf n}y_i=0\quad\text{on }(T_0,T_F)\times\partial\Omega,
\]
represent no flux through the computational boundary.

\subsection{Product rule, integration by parts and right limits}\label{subsec:stieltjes-product-ibp}

For sufficiently regular scalar functions \(u\) and \(v\), the Stieltjes product rule is
\begin{equation}\label{eq:stieltjes-product-rule}
 \Dg(uv)(t)
 =\Dg u(t)\,v(t)+u(t)\,\Dg v(t)
 +\Dg u(t)\,\Dg v(t)\,\Delta^+g(t).
\end{equation}
The correction term is supported on the atoms of $\mu_g$. Its origin can be seen directly. At a non-atomic point, $\Delta^+g(t)=0$ and the ordinary quotient argument gives the usual product rule relative to the clock $g$. At an atom,
\begin{align*}
 \Dg(uv)(t)
 &=\frac{u(t^+)v(t^+)-u(t)v(t)}{\Delta^+g(t)}\\
 &=\Dg u(t)\,v(t)+u(t)\,\Dg v(t)
   +\Delta^+g(t)\,\Dg u(t)\,\Dg v(t),
\end{align*}
because $u(t^+)=u(t)+\Delta^+g(t)\Dg u(t)$ and similarly for $v$. Thus the additional term is exactly the product of the two jump increments divided by the atomic mass. Formula~\eqref{eq:stieltjes-product-rule} and its full function-space justification are established in~\cite{PousoRodriguez2015,FernandezMarquezTojo2025Product}.

It is useful to rewrite~\eqref{eq:stieltjes-product-rule} as
\begin{equation}\label{eq:stieltjes-product-right-form}
 \Dg(uv)(t)=u(t^+)\Dg v(t)+v(t)\Dg u(t).
\end{equation}
Indeed, $u(t^+)=u(t)+\Delta^+g(t)\Dg u(t)$. If $p,z\in\mathcal{AC}_g([a,T];\R^m)$ have the required integrability, apply~\eqref{eq:stieltjes-product-right-form} componentwise to $\langle p,z\rangle$ and integrate. The Stieltjes fundamental theorem gives
\[
 \int_{[a,T)}\Dg\langle p,z\rangle\,\dd\mu_g
 =\langle p(T),z(T)\rangle-\langle p(a),z(a)\rangle.
\]
Consequently,
\begin{equation}\label{eq:stieltjes-ibp-right-value}
 \int_{[a,T)} \langle p^+(t),\Dg z(t)\rangle\,\dd\mu_g(t)
 =\langle p(T),z(T)\rangle-\langle p(a),z(a)\rangle
 -\int_{[a,T)}\langle \Dg p(t),z(t)\rangle\,\dd\mu_g(t),
\end{equation}
where
\begin{equation}\label{eq:right-value-p-plus}
 p^+(t):=p(t)+\Dg p(t)\,\Delta^+g(t)=p(t^+).
\end{equation}
The right value is therefore forced by the atomic product rule; it is not a discretionary time discretization. The continuous identity is used only as structural motivation here. The implemented adjoint is obtained by transposing the fully discrete residual, which automatically reproduces the correct propagation and jump transposes. A duality treatment of Stieltjes differential and integral equations is given in~\cite{MarquezSlavikTvrdy2022}.
\subsection{Weak evolution over one time step}

Assume that $y\in\mathcal W_g$ satisfies
\[
 \Dg y(t)+A(t)y(t)=f(t)
 \quad\text{in }V',\qquad \mu_g\text{-a.e. }t,
\]
where $A(t):V\to V'$ is an elliptic operator. For the left-continuous convention, the mass of $(s_{n-1},s_n]$ is
\begin{equation}\label{eq:stieltjes-step-measure}
 \mu_g((s_{n-1},s_n])
 =g(s_n^+)-g(s_{n-1}^+)=:\Delta g^n.
\end{equation}
This is the same measure convention as in~\eqref{eq:stieltjes-measure-convention}, expressed through right limits: the half-open interval is shifted so that an atom at $s_n$ belongs to step $n$, whereas an atom at $s_{n-1}$ does not.
The integral representation of $y$ gives
\[
 y(s_n^+)-y(s_{n-1}^+)
 =\int_{(s_{n-1},s_n]}\Dg y(t)\,\dd\mu_g(t).
\]
Substitution of the evolution equation therefore yields
\begin{equation}\label{eq:stieltjes-integrated-step}
 y(s_n^+)-y(s_{n-1}^+)
 +\int_{(s_{n-1},s_n]}A(t)y(t)\,\dd\mu_g(t)
 =\int_{(s_{n-1},s_n]}f(t)\,\dd\mu_g(t).
\end{equation}
The implicit Stieltjes--Euler rule replaces the two integrands by their post-step values and uses the exact slab mass $\Delta g^n$. Thus
\[
 \int_{(s_{n-1},s_n]}A(t)y(t)\,\dd\mu_g(t)
 \approx \Delta g^n A(s_n)y^n,
 \qquad
 \int_{(s_{n-1},s_n]}f(t)\,\dd\mu_g(t)
 \approx \Delta g^n f^n.
\]
In the absence of a source this gives
\begin{equation}\label{eq:stieltjes-euler-motivation}
 y^n-y^{n-1}+\Delta g^n A(s_n)y^n=0.
\end{equation}
If $\Delta g^n=0$, the canonical state is propagated unchanged. If the slab contains an atom, its mass is part of the same $\Delta g^n$ and hence of the same implicit solve. This derivation explains why continuous and atomic evolution must not be advanced by two unrelated update rules. Numerical methods for Stieltjes differential equations and the parabolic Stieltjes--Euler construction are discussed in~\cite{FernandezTojo2020Numerical,FernandezTojo2020Bochner}.

\subsection{Biological calendar and averaging notation}

Let $N_G$ be the number of generations. For $k=1,\ldots,N_G$, let
\[
 T_{k,0}<T_{k,1}<T_{k,2}<T_{k,3}<T_{k,4}<T_{k,5}=T_{k+1,0}.
\]
The five intervals represent foundress dispersal, primary-nest activity, secondary-nest activity, future-foundress activity/dispersal, and hibernation. The biological events occur at $T_{k,1},T_{k,2},T_{k,3},T_{k,4}$: first-generation workers are produced at $T_{k,1}$; second-generation workers are added at $T_{k,2}$; future foundresses and males are produced at $T_{k,3}$; and future foundresses generate the next foundress compartment while the other two compartments are reset at $T_{k,4}$. This derivator-based calendar is the one developed for the hornet population model in~\cite{AreaVelutina2025}.

Accordingly, $g_1$ has an atom at $T_{k,4}$ and is constant during $(T_{k,4},T_{k,5}]$; $g_2$ has atoms at $T_{k,3}$ and $T_{k,4}$; and $g_3$ has atoms at all four event times. For $a<b$ and $g(b)>g(a)$, define the Stieltjes phase average
\begin{equation}\label{eq:continuous-average}
 \langle z\rangle_{g,[a,b)}
 :=\frac{\displaystyle\int_{[a,b)}z(s)\,\dd\mu_g(s)}{g(b)-g(a)}.
\end{equation}
We denote the biological transfer coefficients by $\beta_1,\ldots,\beta_5\ge0$, reserving $\mu_i\ge0$ for the natural mortality of compartment $i$. The common-mortality model is recovered as the special case $\mu_1=\mu_2=\mu_3=\alpha_0$.

\begin{assumption}\label{ass:continuous-basic}
The diffusion coefficients satisfy $\nu_i\ge0$, the natural mortalities satisfy $\mu_i\ge0$, and the initial data satisfy $y_{1,0}\in H$, $y_{1,0}\ge0$, and $y_{2,0}=y_{3,0}=0$. Every phase average used in an event datum has a strictly positive clock increment. The fixed compartment--phase masks are measurable, satisfy $0\le\sigma_i(t)\le1$, and are independent of the control. For the academic benchmark, the trapping kernels are smooth, compactly supported functions of their temporal and spatial arguments, and their supports remain inside the prescribed time interval and computational domain for every admissible control. The Galicia computation instead uses domain-truncated kernels normalized by their integral over the finite-element mesh; their normalization denominators are required to stay strictly positive, and a separate coverage safeguard is reported.
\end{assumption}

\section{Continuous finite-dimensional trapping control}

\subsection{Admissible controls}

Fix the number of traps \(N_T\). To match the ordering used in the numerical code, write
\begin{equation}\label{eq:control-vector-order}
 u=(\boldsymbol\tau,\boldsymbol x,\boldsymbol y)
 = (\tau_1,\ldots,\tau_{N_T},x_1,\ldots,x_{N_T},y_1,\ldots,y_{N_T})
 \in\R^{3N_T},
\end{equation}
with trap locations \(z_k=(x_k,y_k)\). Let
\[
 \Omega_{\rm trap}=[x_{\min},x_{\max}]\times[y_{\min},y_{\max}]
 \Subset\Omega.
\]
The admissible set used in the analysis and in the projected first-order conditions is the rectangular box
\begin{equation}\label{eq:Uad-continuous}
 \Uad=[\tau_{\min},\tau_{\max}]^{N_T}
 \times[x_{\min},x_{\max}]^{N_T}
 \times[y_{\min},y_{\max}]^{N_T}.
\end{equation}
Nonrectangular accessibility restrictions can be imposed in applications, but then they require additional constraints and the componentwise box projection used below must be modified accordingly.

\subsection{Trap kernel and controlled mortality}

Let \(E_M>0\), \(T_M>0\), and \(R_M>0\) denote the trap intensity and its temporal and spatial support radii. Define
\begin{equation}\label{eq:trap-control-continuous}
 a_u(t,x)=E_M\sum_{k=1}^{N_T}
 \delta_{T_M}^{1D}(t-\tau_k)\,
 \delta_{R_M}^{2D}(x-z_k).
\end{equation}
Let $\sigma_i:I\to[0,1]$ be a prescribed compartment--phase mask, independent of $u$. The mortality acting on compartment \(i\) is
\begin{equation}\label{eq:controlled-mortality}
 m_{i,u}(t,x)=\mu_i+\sigma_i(t)a_u(t,x).
\end{equation}
The academic benchmark uses $\sigma_i\equiv1$ on every state equation that is present in a phase; the Galicia implementation uses the audited F--P--S--Q--H masks. At an event node, the mask is assigned according to the destination (post-event) phase, matching the implicit post-step convention. Since the masks are fixed, they do not create additional control variables, but they must appear in every state, sensitivity, and gradient formula.
One possible choice of smooth compactly supported bumps is
\begin{align}\label{eq:explicit-bumps}
\delta_{T_M}^{1D}(t-\tau_k)
&=\frac1{C_1T_M}
\exp\!\left(\frac1{((t-\tau_k)/T_M)^2-1}\right)
\one_{\{|t-\tau_k|<T_M\}},\\
\delta_{R_M}^{2D}(x-z_k)
&=\frac1{C_2R_M^2}
\exp\!\left(\frac1{|(x-z_k)/R_M|^2-1}\right)
\one_{\{|x-z_k|<R_M\}},
\end{align}
where the normalization constants used in the implementation are
\[
 C_1=0.4439938161680708,
 \qquad
 C_2=0.4665125410646768.
\]
These are the analytical counterparts of the routines used in the FreeFEM++ implementation.

\begin{remark}[Relation with bilinear removal control]\label{rem:bilinear-removal}
The model is a finite-dimensional subfamily of a bilinear removal problem. Instead of allowing an arbitrary nonnegative field \(m(t,x)\), the additional mortality is restricted to the interpretable family \(a_u\) generated by a prescribed number of localized traps.
\end{remark}

\section{Continuous controlled state system}

For fixed \(u\), define the elliptic operator \(\mathcal L_{i,u}(t):V\to V'\) by
\begin{equation}\label{eq:continuous-elliptic-operator}
 \langle\mathcal L_{i,u}(t)w,v\rangle
 =\nu_i\int_\Omega\nabla w\cdot\nabla v\dd x
 +\int_\Omega m_{i,u}(t,x)wv\dd x.
\end{equation}
To express the biological atoms without imposing a second, independent jump rule, introduce event data \(\mathcal B_i(t,Y)\). They are equal to the current value outside the event set and, at the four biological times, are defined by
\begin{align}\label{eq:biological-event-data}
\mathcal B_3(T_{k,1},Y)
 &=\beta_2\langle y_1\rangle_{g_1,[T_{k,0},T_{k,1})},\notag\\
\mathcal B_3(T_{k,2},Y)
 &=y_3(T_{k,2})+\beta_3\langle y_1\rangle_{g_1,[T_{k,1},T_{k,2})},\notag\\
\mathcal B_2(T_{k,3},Y)
 &=\beta_4\langle y_1\rangle_{g_1,[T_{k,2},T_{k,3})},\notag\\
\mathcal B_3(T_{k,3},Y)
 &=y_3(T_{k,3})+\beta_5\langle y_1\rangle_{g_1,[T_{k,2},T_{k,3})},\notag\\
\mathcal B_1(T_{k,4},Y)
 &=\beta_1\langle y_2\rangle_{g_2,[T_{k,3},T_{k,4})},\notag\\
\mathcal B_2(T_{k,4},Y)&=0,
 \qquad \mathcal B_3(T_{k,4},Y)=0.
\end{align}
For components not listed at a given event, \(\mathcal B_i(t,Y)=y_i(t)\). Let $\Delta^+g_i(t)=g_i(t^+)-g_i(t)$ denote the atomic mass. At an atom with $\Delta^+g_i(t)>0$, define
\begin{equation}\label{eq:continuous-event-source}
 F_i(t,Y)=\frac{\mathcal B_i(t,Y)-y_i(t)}{\Delta^+g_i(t)},
\end{equation}
and set $F_i(t,Y)=0$ outside the event set. The academic benchmark uses unit atomic masses, in which case \eqref{eq:continuous-event-source} reduces to $\mathcal B_i-y_i$. Replacement events therefore include removal of the pre-event value, whereas additive events retain that value in $\mathcal B_i$.

\begin{definition}[Weak solution of the controlled continuous problem]\label{def:continuous-solution}
Let \(u\in\Uad\). A triple \(Y=(y_1,y_2,y_3)\in\mathcal W\) is a weak solution if
\begin{equation}\label{eq:continuous-initial-data}
 y_1(T_0)=y_{1,0},\qquad y_2(T_0)=0,\qquad y_3(T_0)=0,
\end{equation}
and, for every \(v\in V\),
\begin{equation}\label{eq:weak-solution-general}
 \langle\mathrm D_{g_i}y_i(t),v\rangle
 +\langle\mathcal L_{i,u}(t)y_i(t^+),v\rangle
 =\langle F_i(t,Y),v\rangle,
 \qquad \mu_{g_i}\text{-a.e. }t,
\end{equation}
for \(i=1,2,3\).
\end{definition}

\begin{remark}[Continuous model versus implemented scheme]
Equation \eqref{eq:weak-solution-general} is a modelling-level Stieltjes system. At an atom, multiplication by $\Delta^+g_i(t)$ gives
\[
 y_i(t^+)-y_i(t)+\Delta^+g_i(t)\mathcal L_{i,u}(t)y_i(t^+)
 =\mathcal B_i(t,Y)-y_i(t).
\]
Cancelling the pre-event value on both sides yields the transparent resolvent form
\begin{equation}\label{eq:continuous-atomic-resolvent}
 \bigl(I+\Delta^+g_i(t)\mathcal L_{i,u}(t)\bigr)y_i(t^+)
 =\mathcal B_i(t,Y),
\end{equation}
which is the implicit post-event relation used by the discrete scheme. The event datum is assembled from previously stored states and phase averages. No theorem on existence, uniqueness, or differentiability of the continuous control-to-state map is claimed here; the proved well-posedness, differentiability, adjoint identity, and optimization results below concern the fully discrete residual. Continuous-to-discrete convergence is outside the scope of this paper.
\end{remark}

\section{Modelling-level continuous objective}

If the continuous controlled system admits a solution $Y(u)$, a natural modelling-level objective is
\begin{equation}\label{eq:continuous-cost}
 J(Y,u)=\frac12\sum_{i=1}^3\omega_i
 \int_{[T_0,T_F)}\|y_i(t^+)\|_H^2\dd\mu_{g_i}(t)+\Psi(u),
\end{equation}
where \(\omega_i\ge0\) and \(\Psi\) is an optional operational cost. The associated formal control problem is
\begin{equation}\label{eq:continuous-control-problem}
 \min_{u\in\Uad}J(Y(u),u).
\end{equation}
This display specifies a modelling target only: because a continuous control-to-state theorem is not proved here, it is not used to assert existence or differentiability of a continuous reduced objective. The analysis and optimization results below are stated directly for the finite-dimensional discrete residual.

\section{Discretization of the controlled problem}

\subsection{Time grid, biological pointers and Stieltjes increments}

Let
\[
 T_0=s_0<s_1<\cdots<s_N=T_F
\]
contain every \(T_{k,j}\), and let \(p_{k,j}\) satisfy \(s_{p_{k,j}}=T_{k,j}\). Define
\begin{equation}\label{eq:stieltjes-increments}
 \Delta g_i^n=g_i(s_n^+)-g_i(s_{n-1}^+)
 =\mu_{g_i}((s_{n-1},s_n]),
 \qquad n=1,\ldots,N.
\end{equation}
The vector \(c_i^n\) below approximates the post-step state at \(s_n^+\). Consequently, an event occurring at \(s_n=T_{k,j}\) is incorporated in step \(n=p_{k,j}\), not in step \(p_{k,j}+1\).

\subsection{Finite element space and matrices}

Let \(V_h\subset H^1(\Omega)\) have basis \(\{\varphi_1,\ldots,\varphi_q\}\), and write
\[
 y_{i,h}^n=\sum_{r=1}^q(c_i^n)_r\varphi_r.
\]
The mass, stiffness and weighted-mass matrices are defined, for
$r,s=1,\ldots,q$, by
\begin{alignat}{2}
 M_{rs}
 &:=\int_\Omega \varphi_r(x)\varphi_s(x)\,\dd x,
 &\qquad
 K_{rs}
 &:=\int_\Omega \nabla\varphi_r(x)\!\cdot\!\nabla\varphi_s(x)\,\dd x,
 \label{eq:mass-stiffness}\\[0.3em]
 M(b)_{rs}
 &:=\int_\Omega b(x)\varphi_r(x)\varphi_s(x)\,\dd x,
 &\qquad
 &b\in L^\infty(\Omega).
 \label{eq:weighted-mass}
\end{alignat}
Thus $M=M(1)$, while $M(b)$ represents multiplication by $b$ in the
finite-element basis. For the nonnegative coefficients used below, $M(b)$ is
symmetric positive semidefinite.

\subsection{Discrete control}

At time node \(s_n\), set
\begin{equation}\label{eq:discrete-control-field}
 a_u^n(x)=E_M\sum_{k=1}^{N_T}
 \delta_{T_M}^{1D}(s_n-\tau_k)\delta_{R_M}^{2D}(x-z_k).
\end{equation}
Let $\sigma_i^n\in[0,1]$ be the fixed mask used for compartment $i$ at node $s_n$. The mortality matrix for compartment \(i\) is
\begin{equation}\label{eq:controlled-mass-matrix}
 M_{i,u}^n=M(\mu_i+\sigma_i^n a_u^n)=\mu_iM+M(\sigma_i^n a_u^n).
\end{equation}
In FreeFEM++, \(a_u^n\) is first interpolated in \(V_h\); all formulas below are understood for that same finite-element coefficient.

For a perturbation \(h\in\R^{3N_T}\), the directional variation is
\begin{equation}\label{eq:discrete-control-variation}
 \delta a_h^n(x)=E_M\sum_{k=1}^{N_T}\Big[
 \partial_{\tau_k}\delta_{T_M}^{1D}(s_n-\tau_k)h_{\tau_k}\,
 \delta_{R_M}^{2D}(x-z_k)
 +\delta_{T_M}^{1D}(s_n-\tau_k)
 \nabla_{z_k}\delta_{R_M}^{2D}(x-z_k)\cdot(h_{x_k},h_{y_k})\Big],
\end{equation}
and therefore
\begin{equation}\label{eq:discrete-matrix-variation}
 D_uM_{i,u}^n[h]=M(\sigma_i^n\delta a_h^n).
\end{equation}

\subsection{Discrete biological averages}

The implementation uses the arithmetic phase average
\begin{equation}\label{eq:discrete-average}
 \calA_{i,k}^{a,b}(C)
 =\frac1{m_k^{a,b}}\sum_{r=p_{k,a}}^{p_{k,b}-1}c_i^r,
 \qquad m_k^{a,b}=p_{k,b}-p_{k,a}.
\end{equation}
This is the precise operator differentiated and transposed in the validation code. It is part of the definition of the discrete academic benchmark. On a uniform active interval it resembles an ordinary-time phase average, but it is not, in general, a quadrature formula for the Stieltjes average \eqref{eq:continuous-average} when endpoint atoms carry mass. Consequently, the adjoint and optimization results are exact for \eqref{eq:discrete-average}; convergence of this discrete averaging convention to the modelling-level Stieltjes average is not analysed.

\section{Fully discrete state equation}

For each step \(n\), let \(b_i^n(C)\) denote the datum on the right-hand side of the implicit elliptic solve. For \(k=1,\ldots,N_G\), its event values are
\begin{align}\label{eq:event-datum-b1}
 b_1^n(C)&=
 \begin{cases}
 \beta_1\calA_{2,k}^{3,4}(C),&n=p_{k,4},\\
 c_1^{n-1},&\text{otherwise},
 \end{cases}\\[0.4em]
 b_2^n(C)&=
 \begin{cases}
 \beta_4\calA_{1,k}^{2,3}(C),&n=p_{k,3},\\
 0,&n=p_{k,4},\\
 c_2^{n-1},&\text{otherwise},
 \end{cases}\label{eq:event-datum-b2}\\[0.4em]
 b_3^n(C)&=
 \begin{cases}
 \beta_2\calA_{1,k}^{0,1}(C),&n=p_{k,1},\\
 c_3^{n-1}+\beta_3\calA_{1,k}^{1,2}(C),&n=p_{k,2},\\
 c_3^{n-1}+\beta_5\calA_{1,k}^{2,3}(C),&n=p_{k,3},\\
 0,&n=p_{k,4},\\
 c_3^{n-1},&\text{otherwise}.
 \end{cases}\label{eq:event-datum-b3}
\end{align}
The ranges in \eqref{eq:event-datum-b1}--\eqref{eq:event-datum-b3} are causal: every average ends before the event node at which it is used. Define
\begin{equation}\label{eq:state-step-matrix}
 A_i^n(u)=M+\Delta g_i^n\bigl(\nu_iK+M_{i,u}^n\bigr).
\end{equation}
The fully discrete state equation is
\begin{equation}\label{eq:discrete-step}
 A_i^n(u)c_i^n=Mb_i^n(C),
 \qquad i=1,2,3,\quad n=1,\ldots,N,
\end{equation}
with prescribed \(c_1^0=c_{1,0}\) and \(c_2^0=c_3^0=0\). At a biological event, the continuous part and the atomic part of \(\Delta g_i^n\) both enter the same implicit solve, exactly as in the implementation.

\section{Algebraic formulation of the discrete state system}

Collect all coefficients in
\[
 C=(c_1^0,\ldots,c_1^N,c_2^0,\ldots,c_2^N,c_3^0,\ldots,c_3^N)\in\R^Q,
 \qquad Q=3(N+1)q.
\]
Define \(\calF(C,u)=0\) by
\begin{align}\label{eq:initial-residuals}
 \calF_i^0(C,u)&=c_i^0-c_{i,0},\\
 \calF_i^n(C,u)&=A_i^n(u)c_i^n-Mb_i^n(C),
 \qquad n=1,\ldots,N.\label{eq:residual-components}
\end{align}
This algebraic residual is the object differentiated in both the linearized and adjoint codes.

\begin{proposition}[Well-posedness and differentiability of the discrete state equation]\label{prop:discrete-state-wellposed}
Assume that the assembled mortality coefficients are nonnegative and depend $C^1$-smoothly on the control. In the normalized Galicia formulation, also assume that every kernel normalization denominator is bounded away from zero on an open set containing $\Uad$. Then, for every $u\in\Uad$, the discrete state equation $\calF(C,u)=0$ has a unique solution. Moreover, the control-to-state map $S:u\mapsto C$ is continuously differentiable on that open set.
\end{proposition}

\begin{proof}
Let $v\in\R^q\setminus\{0\}$ and let $v_h=\sum_{r=1}^qv_r\varphi_r$. For the consistent mass matrix, linear independence of the finite-element basis gives
\[
 v^\top Mv=\int_\Omega v_h^2\dd x>0,
 \qquad
 v^\top M_{i,u}^nv=\int_\Omega m_{i,u}^n v_h^2\dd x\ge0.
\]
The Galicia computations replace these reaction--mass matrices by their nodal lumped counterparts. If $m_r=\int_\Omega\varphi_r\dd x$ are the positive lumping weights and $m_{i,u,r}^n\ge0$ are the nodal mortality values, then
\[
 v^\top M_Lv=\sum_{r=1}^qm_rv_r^2>0,
 \qquad
 v^\top M_{i,u,L}^nv=\sum_{r=1}^qm_rm_{i,u,r}^nv_r^2\ge0.
\]
Thus the argument below applies verbatim with $(M,M_{i,u}^n)$ denoting either the consistent pair or the lumped pair used in the Galicia residual. Moreover, $v^\top Kv\ge0$. Consequently,
\begin{equation}\label{eq:discrete-step-coercivity}
 v^\top A_i^n(u)v
 =v^\top Mv+\Delta g_i^n\bigl(\nu_i v^\top Kv+v^\top M_{i,u}^nv\bigr)
 \ge v^\top Mv>0.
\end{equation}
Thus every $A_i^n(u)$ is symmetric positive definite, including a flat-clock step $\Delta g_i^n=0$, for which $A_i^n=M$.

The causality of the calendar is essential. Each average in~\eqref{eq:event-datum-b1}--\eqref{eq:event-datum-b3} ends at index $p_{k,b}-1$, whereas the corresponding event equation is solved at $p_{k,b}$. Therefore $b_i^n(C)$ depends only on coefficients already available when step $n$ is reached. Starting from the prescribed initial blocks, one obtains $c_i^n$ uniquely by solving
\[
 c_i^n=A_i^n(u)^{-1}Mb_i^n(C)
\]
in chronological order. This proves existence and uniqueness without invoking a global nonlinear theorem.

For differentiability, observe that the event maps $b_i^n$ are affine combinations of previously computed states. The matrices $A_i^n(u)$ are $C^1$ in $u$: this is immediate for translated smooth bumps; for normalized kernels it follows from the quotient rule once the denominator is positive. Matrix inversion is $C^1$ on the open set of nonsingular matrices. Induction over $n$ therefore shows directly that each $c_i^n=S_i^n(u)$ is $C^1$. Equivalently, $\calF_C$ is block lower triangular with identity initial blocks and diagonal blocks $A_i^n(u)$, so it is nonsingular and the finite-dimensional implicit function theorem applies.
\end{proof}

\section{The discrete optimal control problem}

For the abstract discrete analysis, let
\begin{equation}\label{eq:general-discrete-objective}
 \calJ(C,u)=\Phi(C)+\Psi(u),
\end{equation}
where $\Phi$ is continuously differentiable in the state and $\Psi$ is continuous on $\Uad$ and continuously differentiable on an open neighbourhood of $\Uad$. The differentiability assumption is used for first-order conditions; continuity alone is sufficient for existence. In the Galicia problem, $\Phi$ is the normalized diagnostic $J_{\rm bio}$ and $\Psi$ contains the $C^1$ coverage and separation penalties defined later.

The cost implemented in the academic benchmark is
\begin{equation}\label{eq:code-like-cost}
 \calJ(C,u)=\frac12\sum_{i=1}^3\sum_{n=1}^N
 w_i\Delta g_i^n(c_i^n)^\top Mc_i^n+\Psi(u),
\end{equation}
where \(w_i\ge0\). Equivalently,
\begin{equation}\label{eq:quadratic-cost-for-gradient}
 \calJ(C,u)=\frac12\sum_{i=1}^3\sum_{n=0}^N(c_i^n)^\top Q_i^nc_i^n+\Psi(u),
\end{equation}
with
\begin{equation}\label{eq:academic-Q-matrices}
 Q_i^0=0,\qquad Q_i^n=w_i\Delta g_i^nM\quad(n\ge1).
\end{equation}
The quadratic representation~\eqref{eq:quadratic-cost-for-gradient} is specific to the academic benchmark. In the Galicia experiment, $\nabla_C\calJ$ is the exact state gradient of the normalized diagnostic $J_{\rm bio}$; the coverage and separation penalties depend explicitly on the control and therefore enter $\nabla_u\calJ$. The algebraic linearized and adjoint identities below require only differentiability of $\calJ$ and are unchanged by this replacement.
\begin{samepage}
The full-space discrete control problem is
\begin{equation}\label{eq:opt-problem}
 \boxed{\begin{aligned}
  \underset{(C,u)\in\R^{Q}\times\Uad}{\operatorname{minimize}}
  \quad & \calJ(C,u)\\
  \text{subject to}\quad & \calF(C,u)=0.
\end{aligned}}
\end{equation}
Here $Q$ denotes the total number of discrete state coefficients. By
Proposition~\ref{prop:discrete-state-wellposed}, for every $u\in\Uad$ the
constraint has the unique solution $C=S(u)$. Hence~\eqref{eq:opt-problem} is
equivalent to the reduced problem
\begin{equation}\label{eq:reduced-problem}
 \boxed{\begin{aligned}
  \underset{u\in\Uad}{\operatorname{minimize}}
  \quad & j(u),\\[-0.1em]
  \text{where}\quad &j(u):=\calJ(S(u),u).
\end{aligned}}
\end{equation}
\end{samepage}

\begin{proposition}[Existence for the box-constrained discrete problem]\label{prop:existence}
The reduced problem \eqref{eq:reduced-problem} has at least one solution.
\end{proposition}

\begin{proof}
The admissible set $\Uad$ is a finite Cartesian product of closed bounded intervals; hence it is nonempty and compact in the finite-dimensional control space. Proposition~\ref{prop:discrete-state-wellposed} gives continuity of $S$, and the assumptions on $\Phi$ and $\Psi$ give continuity of $\calJ$. Therefore $j(u)=\calJ(S(u),u)$ is continuous on $\Uad$. Let $(u^m)$ be a minimizing sequence. Compactness yields a subsequence, not relabelled, such that $u^m\to u^*\in\Uad$. Continuity gives
\[
 j(u^*)=\lim_{m\to\infty}j(u^m)=\inf_{u\in\Uad}j(u),
\]
so $u^*$ is a discrete optimal control.
\end{proof}

\subsection{First-order optimality conditions for the discrete problem}\label{subsec:discrete-optimality-conditions}

\begin{lemma}[Variational inequality and box projection]\label{lem:vi-projection-equivalence}
Let $K\subset\R^m$ be a nonempty closed convex set, let $g\in\R^m$, and let $\rho>0$. For $u\in K$, the following are equivalent:
\[
 g^\top (v-u)\ge0\quad\forall v\in K,
 \qquad
 u=\Pi_K(u-\rho g).
\]
If $K=\prod_{r=1}^m[\ell_r,\bar u_r]$, these conditions are equivalent to the usual componentwise sign conditions at the lower bound, interior, and upper bound.
\end{lemma}

\begin{proof}
The Euclidean projection $p=\Pi_K(x)$ is characterized by
\[
 (x-p)^\top (v-p)\le0\qquad\forall v\in K.
\]
Set $x=u-\rho g$ and $p=u$. The projection inequality becomes
\[
 -\rho g^\top (v-u)\le0\qquad\forall v\in K,
\]
which is exactly the stated variational inequality. For a box, choose variations that change only component $r$. If $\ell_r<u_r<\bar u_r$, both positive and negative variations are admissible and force $g_r=0$. At $u_r=\ell_r$, only positive variations are locally admissible and give $g_r\ge0$; at $u_r=\bar u_r$, only negative variations are admissible and give $g_r\le0$.
\end{proof}

If \(u^\ast\) is a local minimizer, then
\begin{equation}\label{eq:reduced-vi-optimality}
 \nabla j(u^\ast)^\top (v-u^\ast)\ge0
 \qquad\forall v\in\Uad.
\end{equation}
Equivalently, for every \(\rho>0\),
\begin{equation}\label{eq:projection-optimality}
 u^\ast=\Pi_{\Uad}\bigl(u^\ast-\rho\nabla j(u^\ast)\bigr).
\end{equation}
For a component \(r\), this gives
\begin{equation}\label{eq:componentwise-box-optimality}
\begin{cases}
 \partial_rj(u^\ast)=0,&\ell_r<u_r^\ast<\bar u_r,\\
 \partial_rj(u^\ast)\ge0,&u_r^\ast=\ell_r,\\
 \partial_rj(u^\ast)\le0,&u_r^\ast=\bar u_r.
\end{cases}
\end{equation}
Since $\calF_C(C^\ast,u^\ast)$ is nonsingular, the equality constraint has full rank and the adjoint multiplier is unique. With the discrete Lagrangian
\begin{equation}\label{eq:discrete-lagrangian}
 \mathscr L(C,u,P)=\calJ(C,u)-P^\top \calF(C,u),
\end{equation}
the corresponding first-order system is
\begin{align}
 \calF(C^\ast,u^\ast)&=0,\label{eq:kkt-state}\\
 \calF_C(C^\ast,u^\ast)^\top P^\ast
 &=\nabla_C\calJ(C^\ast,u^\ast),\label{eq:kkt-adjoint}\\
 \bigl(\nabla_u\calJ(C^\ast,u^\ast)
 -\calF_u(C^\ast,u^\ast)^\top P^\ast\bigr)^\top(v-u^\ast)&\ge0
 \quad\forall v\in\Uad.\label{eq:kkt-control-vi}
\end{align}
Since convexity of the reduced problem is neither assumed nor established, these are necessary local conditions and not a certificate of global optimality.

\subsection{Projected-box diagnostics}\label{subsec:projected-gradient-residual}

In addition to IPOPT's internal termination quantities, we use two external diagnostics. In physical coordinates,
\begin{equation}\label{eq:projected-gradient-residual}
 R_{\rho}^{\rm phys}(u)=
 \left\|u-\Pi_{\Uad}\bigl(u-\rho\nabla j(u)\bigr)\right\|_2,
 \qquad
 \widehat R_{\rho}^{\rm phys}(u)=
 \frac{R_{\rho}^{\rm phys}(u)}{1+\|u\|_2}.
\end{equation}
This quantity is reproducible but unit-dependent when activation times and spatial coordinates are combined.

For cross-component stationarity we therefore introduce a dimensionless control. Let $u_c$ be a fixed reference control, let $D=\operatorname{diag}(s_1,\ldots,s_m)$ contain positive physical scales, and set
\[
 \zeta=D^{-1}(u-u_c),\qquad
 \widetilde j(\zeta)=\frac{j(u_c+D\zeta)}{J_s},
\]
where $J_s>0$ is an objective scale. Since $D$ is diagonal and positive,
\begin{equation}\label{eq:scaled-gradient-transform}
 \nabla_\zeta\widetilde j(\zeta)
 =\frac{1}{J_s}D^\top \nabla_u j(u_c+D\zeta)
 =\frac{1}{J_s}D\nabla_u j(u_c+D\zeta).
\end{equation}
If $\widetilde{\mathcal U}_{\rm ad}=D^{-1}(\Uad-u_c)$, the dimensionless projected residual is
\begin{equation}\label{eq:dimensionless-projected-residual}
 R^{\rm dim}(\zeta)=
 \left\|\zeta-\Pi_{\widetilde{\mathcal U}_{\rm ad}}
 \bigl(\zeta-\nabla_{\zeta}\widetilde j(\zeta)\bigr)\right\|_{\infty}.
\end{equation}
For the Galicia optimization, times are scaled by one month, spatial offsets by $10^5$ m, and $J_s=1$ because $J_{\rm bio}$ is already dimensionless. All Galicia stationarity residuals reported below refer to~\eqref{eq:dimensionless-projected-residual} applied to the penalized objective $\widehat J$ defined in Section~\ref{sec:galicia-ipopt-optimization}. Because the accepted safeguards are strictly inactive, the same residual is obtained from $J_{\rm bio}$ at the reported controls. By Lemma~\ref{lem:vi-projection-equivalence}, this residual vanishes exactly at a box-stationary point in the scaled variables. Its numerical value depends on the declared physical scales, which are therefore stated explicitly. It remains an external diagnostic and is not identical to IPOPT's internally scaled dual infeasibility.

\section{Linearized state equation and gradient from linearized states}\label{sec:direct-sensitivity-gradient}

Let \(C=S(u)\), and let
\[
 h=(h_{\tau_1},\ldots,h_{\tau_{N_T}},h_{x_1},\ldots,h_{x_{N_T}},h_{y_1},\ldots,h_{y_{N_T}}).
\]
The state sensitivity \(Z=S'(u)h=(z_i^n)_{i=1,2,3;\,n=0,\ldots,N}\) satisfies
\begin{equation}\label{eq:linearized-state}
 \calF_C(C,u)Z+\calF_u(C,u)h=0.
\end{equation}
For \(n\ge1\), its componentwise form is
\begin{equation}\label{eq:linearized-state-componentwise}
 A_i^n(u)z_i^n-MDb_i^n(C)[Z]
 =-\Delta g_i^nM(\sigma_i^n\delta a_h^n)c_i^n.
\end{equation}
The initial sensitivities are zero because the initial data do not depend on \(u\).

\begin{proposition}[Well-posedness of the linearized state equation]\label{prop:linearized-state-wellposed}
For every $u$ in the neighbourhood of Proposition~\ref{prop:discrete-state-wellposed} and every direction $h$, equation~\eqref{eq:linearized-state} has a unique solution $Z=S'(u)h$. It is obtained by the same chronological block solves as the state equation.
\end{proposition}

\begin{proof}
Differentiate $\calF(S(u),u)=0$ in direction $h$. The chain rule gives~\eqref{eq:linearized-state}. The matrix $\calF_C(C,u)$ is the same nonsingular block lower-triangular matrix identified in Proposition~\ref{prop:discrete-state-wellposed}; hence the linear system has a unique solution. Componentwise, the right-hand side at step $n$ contains the known control perturbation $-\Delta g_i^nM(\sigma_i^n\delta a_h^n)c_i^n$ and averages or propagated sensitivities from earlier indices only. Therefore the sensitivities can be computed by forward recursion using the matrices $A_i^n(u)$ already assembled for the state solve.
\end{proof}

For ordinary propagation, \(Db_i^n(C)[Z]=z_i^{n-1}\). If an event datum contains
\[
 \beta\calA_{i,k}^{a,b}(C),
\]
its derivative is
\begin{equation}\label{eq:average-linearization}
 D\calA_{i,k}^{a,b}(C)[Z]
 =\frac1{m_k^{a,b}}\sum_{r=p_{k,a}}^{p_{k,b}-1}z_i^r.
\end{equation}
Additive events retain the derivative of the propagated term, whereas replacement and reset events do not. This distinction is encoded exactly by \eqref{eq:event-datum-b1}--\eqref{eq:event-datum-b3}.

By the chain rule,
\begin{equation}\label{eq:direct-directional-gradient-general}
 j'(u)h=D_C\calJ(C,u)[Z]+D_u\calJ(C,u)[h].
\end{equation}
For \eqref{eq:quadratic-cost-for-gradient},
\begin{equation}\label{eq:direct-directional-gradient-quadratic}
 j'(u)h=\sum_{i=1}^3\sum_{n=0}^N(Q_i^nc_i^n)^\top z_i^n+\Psi'(u)h.
\end{equation}
Thus, for coordinate direction \(e_\ell\), the gradient component obtained from the corresponding linearized state \(Z_\ell=S'(u)e_\ell\) is
\begin{equation}\label{eq:direct-gradient-component}
 \frac{\partial j}{\partial u_\ell}(u)
 =\sum_{i=1}^3\sum_{n=0}^N(Q_i^nc_i^n)^\top z_{i,\ell}^n
 +\frac{\partial\Psi}{\partial u_\ell}(u).
\end{equation}
Computing the full gradient in this manner requires \(3N_T\) linearized Stieltjes--parabolic solves after the common state solve.

\section{Adjoint equation and explicit adjoint gradient}\label{sec:adjoint-gradient}

For the sign convention
\[
 \mathscr L(C,u,P)=\calJ(C,u)-P^\top \calF(C,u),
\]
the derivative with respect to the state is
\[
 D_C\mathscr L(C,u,P)[Z]
 =D_C\calJ(C,u)[Z]-P^\top \calF_C(C,u)Z.
\]
We eliminate the unknown state sensitivity from the reduced derivative by choosing $P$ so that this expression vanishes for every $Z$. In Euclidean coordinates this is exactly
\begin{equation}\label{eq:adjoint-algebraic}
 \calF_C(C,u)^\top P=\nabla_C\calJ(C,u).
\end{equation}
The adjoint exists uniquely because $\calF_C$ is nonsingular. Since $\calF_C$ is block lower triangular, its transpose is block upper triangular; hence~\eqref{eq:adjoint-algebraic} is solved backward in the biological calendar.

To derive the reduced gradient, let $Z=S'(u)h$. Then
\begin{align*}
 j'(u)h
 &=D_C\calJ(C,u)[Z]+D_u\calJ(C,u)[h]\\
 &=P^\top \calF_C(C,u)Z+D_u\calJ(C,u)[h]\\
 &=-P^\top \calF_u(C,u)h+D_u\calJ(C,u)[h],
\end{align*}
where the last equality uses the linearized state equation. Thus
\begin{equation}\label{eq:reduced-gradient-adjoint}
 j'(u)h=D_u\calJ(C,u)[h]-P^\top \calF_u(C,u)h,
\end{equation}
and, because this identity holds for every $h$,
\begin{equation}\label{eq:gradient-vector}
 \nabla j(u)=\nabla_u\calJ(C,u)-\calF_u(C,u)^\top P.
\end{equation}
\begin{remark}[Continuous Stieltjes adjoints and discrete transposes]\label{rem:p-plus-discrete-adjoint}
At the continuous level, Stieltjes integration by parts involves the right value \(p^+(t)=p(t^+)\); see \eqref{eq:stieltjes-ibp-right-value}. In the implemented scheme, the same structural issue appears as the exact transpose of propagation, replacement, reset and averaging operators. Deriving \eqref{eq:adjoint-algebraic} from the residual guarantees that these operations are transposed consistently.
\end{remark}

\subsection{Blockwise adjoint equation}\label{subsec:blockwise-adjoint}

For \(r=1,\ldots,N\), the block associated with \(c_i^r\) has the form
\begin{equation}\label{eq:block-adjoint-structure}
 (A_i^r)^\top P_i^r
 -\chi_{i,r}^{\rm self}M^\top P_i^{r+1}
 -\sum_{(j,n,\beta,k,a,b)\in\mathscr A(i,r)}
 \frac{\beta}{m_k^{a,b}}M^\top P_j^n
 =\nabla_{c_i^r}\Phi(C),
\end{equation}
with the convention \(P_i^{N+1}=0\). Here \(\chi_{i,r}^{\rm self}=1\) exactly when \(c_i^r\) occurs with coefficient one in the next datum \(b_i^{r+1}\). This includes ordinary propagation and the propagated part of an additive event, but excludes replacements and resets. The set \(\mathscr A(i,r)\) contains every later event datum whose phase average includes \(c_i^r\). For the initial block,
\begin{equation}\label{eq:block-adjoint-initial}
 P_i^0-\chi_{i,0}^{\rm self}M^\top P_i^1
 -\sum_{(j,n,\beta,k,a,b)\in\mathscr A(i,0)}
 \frac{\beta}{m_k^{a,b}}M^\top P_j^n
 =\nabla_{c_i^0}\Phi(C).
\end{equation}
For the academic quadratic objective, $\nabla_{c_i^r}\Phi=Q_i^rc_i^r$; for the Galicia diagnostic the exact assembled derivative is used instead. This separate initial equation avoids the undefined expression \(A_i^0\). Formula~\eqref{eq:block-adjoint-structure} follows by collecting every residual in which \(c_i^r\) occurs: its own residual contributes \((A_i^r)^\top P_i^r\); the next propagated residual contributes \(-M^\top P_i^{r+1}\) when propagation is present; and every later phase average containing \(c_i^r\) contributes the corresponding weighted term. Since all such later adjoints are already known in a backward sweep, the equation determines \(P_i^r\) uniquely. The spatial matrices are symmetric, so \((A_i^r)^\top=A_i^r\), but the full time operator is not symmetric because propagation, replacement, reset and averaging have different transposes.

\subsection{Transpose of the phase averages}\label{subsec:average-transpose-correction}

If a residual contains
\[
 -M\left(\beta\frac1m\sum_{r=a}^{b-1}c_i^r\right),
\]
then each participating adjoint block receives
\begin{equation}\label{eq:averaging-transpose}
 -\frac{\beta}{m}M^\top P_j^n.
\end{equation}
The denominator must therefore be identical in the state, linearized and adjoint implementations.

\subsection{Explicit adjoint formula for the gradient}\label{subsec:explicit-adjoint-gradient}

Only the matrices \(A_i^n(u)\) depend explicitly on the control. Hence
\begin{equation}\label{eq:adjoint-directional-gradient}
 j'(u)h=\Psi'(u)h
 -\sum_{i=1}^3\sum_{n=1}^N
 \Delta g_i^n(P_i^n)^\top M(\sigma_i^n\delta a_h^n)c_i^n.
\end{equation}
For a scalar component \(u_\ell\),
\begin{equation}\label{eq:component-gradient}
 \frac{\partial j}{\partial u_\ell}(u)
 =\frac{\partial\Psi}{\partial u_\ell}(u)
 -\sum_{i=1}^3\sum_{n=1}^N
 \Delta g_i^n(P_i^n)^\top
 M\left(\sigma_i^n\frac{\partial a_u^n}{\partial u_\ell}\right)c_i^n.
\end{equation}
The kernel derivatives are
\begin{align}
 \frac{\partial a_u^n}{\partial\tau_k}(x)
 &=E_M\,\partial_{\tau_k}\delta_{T_M}^{1D}(s_n-\tau_k)
 \delta_{R_M}^{2D}(x-z_k),\label{eq:kernel-derivative-tau}\\
 \frac{\partial a_u^n}{\partial x_k}(x)
 &=E_M\,\delta_{T_M}^{1D}(s_n-\tau_k)
 \partial_{x_k}\delta_{R_M}^{2D}(x-z_k),\label{eq:kernel-derivative-x}\\
 \frac{\partial a_u^n}{\partial y_k}(x)
 &=E_M\,\delta_{T_M}^{1D}(s_n-\tau_k)
 \partial_{y_k}\delta_{R_M}^{2D}(x-z_k).\label{eq:kernel-derivative-y}
\end{align}
The derivatives are taken with respect to the centres, including the corresponding sign convention used by the implementation.

\begin{proposition}[Equivalence of the two gradient evaluations]\label{prop:direct-adjoint-equivalence}
For every direction \(h\), the derivative computed with the linearized state equation coincides algebraically with \eqref{eq:adjoint-directional-gradient}.
\end{proposition}

\begin{proof}
Let $Z=S'(u)h$ be the unique solution of~\eqref{eq:linearized-state}. The adjoint equation implies
\[
 D_C\calJ(C,u)[Z]
 =\nabla_C\calJ(C,u)^\top Z
 =P^\top \calF_C(C,u)Z.
\]
Substituting $\calF_CZ=-\calF_uh$ gives
\[
 D_C\calJ(C,u)[Z]=-P^\top \calF_u(C,u)h.
\]
Adding the explicit control derivative $D_u\calJ(C,u)[h]=\Psi'(u)h$ proves~\eqref{eq:reduced-gradient-adjoint}. In the present residual, control dependence occurs only in
\[
 A_i^n(u)c_i^n
 =\left[M+\Delta g_i^n\bigl(\nu_iK+M(\mu_i+\sigma_i^na_u^n)\bigr)\right]c_i^n.
\]
Therefore
\[
 \calF_u(C,u)h
 =\left(\Delta g_i^nM(\sigma_i^n\delta a_h^n)c_i^n\right)_{i,n},
\]
and insertion into~\eqref{eq:reduced-gradient-adjoint} yields~\eqref{eq:adjoint-directional-gradient}. Hence the direct and adjoint evaluations are algebraically identical; their numerical comparison tests the implementations rather than two different mathematical derivatives.
\end{proof}

\subsection{Computational implication of the adjoint identity}\label{subsec:adjoint-computational-implication}

After the common state solve, the gradient obtained from the linearized state equations requires one solve per scalar control component, namely \(3N_T\) solves. The adjoint approach requires one backward transposed solve and then \(3N_T\) inexpensive kernel contractions:
\[
\begin{array}{c|c}
\text{gradients from linearized state equations} & 3N_T\ \text{linearized solves}\\
\text{adjoint gradient} & 1\ \text{adjoint solve}+3N_T\ \text{kernel contractions}.
\end{array}
\]
This operation count does not replace a timing experiment. The measured comparison is reported in Section~\ref{sec:academic-validation}, and all optimization runs use the adjoint state to provide the reduced gradient to IPOPT.

\section{Academic numerical validation framework}\label{sec:academic-validation}

This section introduces the academic benchmark used to validate the numerical differentiation and optimization procedures before passing to the realistic Galicia domain. It is not a repetition of the algebraic derivation above. The role of the previous sections is to derive the state equation, the linearized state equation and the adjoint gradient. The role of the present section is experimental: we verify that the implemented gradient satisfies those identities and we measure the computational cost of the two gradient-evaluation strategies in the actual FreeFEM++ code~\cite{Hecht2012FreeFEM}.

In contrast with the preliminary synthetic benchmark, the present one is explicitly aligned with the biological calendar of the published Stieltjes-time \textit{Vespa velutina} model. The aim is still academic: the parameters are not calibrated to field data, but the clocks, impulses and resets reproduce the structure of one biological generation. The validation is organized in three steps: description of the benchmark, correctness of the gradient, and timing of direct versus adjoint gradient evaluations.

\subsection{Computational domain, time grid and finite element space}

The spatial domain is the rectangle
\[
 \Omega=(0,40)\times(0,24)\subset\mathbb R^2,
\]
measured in kilometres. It is discretized by a structured triangular mesh obtained from a \(48\times 30\) rectangular grid and continuous piecewise linear finite elements. The time interval is
\[
 I=[0,180]\quad\text{days},
\]
with
\[
 N_t=60,\qquad \Delta t=3\quad\text{days}.
\]
The endpoints of the rescaled one-generation biological calendar are
\begin{equation}\label{eq:academic-calendar-times}
 s_0=0,\qquad s_1=30,\qquad s_2=45,\qquad s_3=105,
 \qquad s_4=135,
 \qquad s_5=180.
\end{equation}
The phase lengths follow the proportions
\[
 2:1:4:2:3,
\]
corresponding respectively to foundress-queen dispersal, primary nest activity, secondary nest activity, future-foundress activity/dispersal, and hibernation.

The three compartments are foundress queens, future foundress queens, and workers/males. The benchmark uses homogeneous Neumann boundary conditions and a localized Gaussian initial condition,
\[
 y_1(0,x)=80\exp\left(-\frac{|x-(8,12)|^2}{2\,2^2}\right),
 \qquad y_2(0,x)=y_3(0,x)=0.
\]
The diffusion, mortality, trapping and cost parameters used in the FreeFEM++ implementation are reported in Table~\ref{tab:academic-parameters}. They should be regarded as benchmark parameters rather than biological calibration values.

\begin{table}[htbp]
\centering
\begin{tabular}{lll}
\toprule
Quantity & Value & Meaning \\
\midrule
\(\nu_1,\nu_2,\nu_3\) & \(0.030,0.020,0.055\) km$^2$/day & diffusion coefficients \\
\(\mu_1,\mu_2,\mu_3\) & \(0.006,0.004,0.010\) day$^{-1}$ & natural mortality rates \\
\(w_1,w_2,w_3\) & \(0.2,1.0,0.6\) & dimensionless quadratic observation weights \\
\(E_M\) & \(15\) km$^2$ & integrated trap-intensity scale \\
\(T_M\) & \(10\) days & temporal support radius of a trap \\
\(R_M\) & \(4\) km & spatial support radius of the academic trapping kernel \\
\bottomrule
\end{tabular}
\caption{Main parameters of the academic benchmark.}
\label{tab:academic-parameters}
\end{table}

\subsection{Academic derivators aligned with the biological cycle}

The continuous parts of the derivators follow the original construction: only the foundress-queen clock is frozen during hibernation. Thus,
\[
 g_1^C(t)=\min\{t,s_4\},
 \qquad
 g_2^C(t)=t,
 \qquad
 g_3^C(t)=t.
\]
The jump parts are chosen according to the biological events of one generation:
\begin{align*}
 g_1^B(t)&=\chi_{(s_4,\infty)}(t),\\
 g_2^B(t)&=\chi_{(s_3,\infty)}(t)+\chi_{(s_4,\infty)}(t),\\
 g_3^B(t)&=\chi_{(s_1,\infty)}(t)+\chi_{(s_2,\infty)}(t)
 +\chi_{(s_3,\infty)}(t)+\chi_{(s_4,\infty)}(t).
\end{align*}
Hence
\[
 g_i=g_i^C+g_i^B,
 \qquad i=1,2,3.
\]
The atom of \(g_1\) at \(s_4\) represents the renewal of foundress queens from the surviving future foundress queens. The atoms of \(g_2\) represent production and reset of future foundress queens. The atoms of \(g_3\) represent production of workers/males and colony death. Figure~\ref{fig:academic-derivators} makes this decomposition explicit in the same visual format later used for the Galicia multi-generation clocks, so that flat periods, linear growth intervals and atomic transfers can be compared directly.

\begin{figure}[htbp]
\centering
\includegraphics[width=0.96\linewidth]{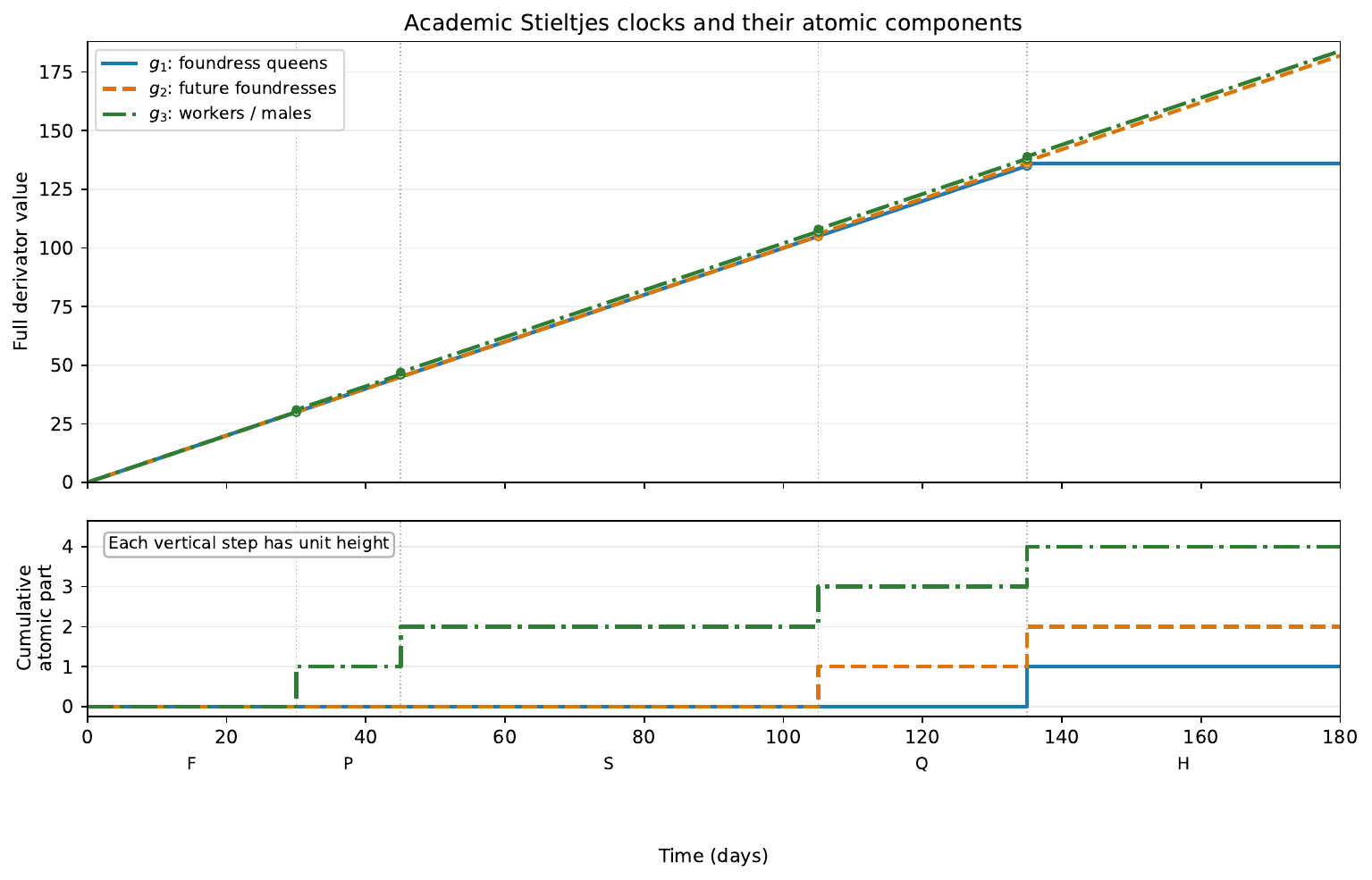}
\caption{Academic derivators in the same two-panel format used later for the Galicia clocks. The upper panel shows the full one-generation Stieltjes clocks $g_1$, $g_2$ and $g_3$, including their continuous parts, flat portions and unit jumps. The lower panel isolates the cumulative atomic components, so that each biological transfer is displayed at its true unit height and the zones of constancy between jumps are clearly visible. The letters F, P, S, Q and H denote foundress dispersal, primary-nest activity, secondary-nest activity, future-foundress activity/dispersal and hibernation, respectively. In the upper panel, open and filled markers indicate the left and right values at each jump.}
\label{fig:academic-derivators}
\end{figure}

On the grid, the Stieltjes increments used in the scheme are
\[
 \Delta g_i^n=g_i(t_n^+)-g_i(t_{n-1}^+),
 \qquad t_n=n\Delta t.
\]
Thus zero increments freeze the foundress-queen equation during hibernation, while atomic biological events are incorporated in the corresponding Stieltjes steps.

\subsection{Biological transfer terms in the academic benchmark}

The benchmark implements the following averaged transitions:
\begin{align*}
 s_1:
 &\qquad y_3(s_1^+)=\beta_2\,\langle y_1\rangle_{g_1,[s_0,s_1)},\\
 s_2:
 &\qquad y_3(s_2^+)=y_3(s_2)+\beta_3\,\langle y_1\rangle_{g_1,[s_1,s_2)},\\
 s_3:
 &\qquad y_2(s_3^+)=\beta_4\,\langle y_1\rangle_{g_1,[s_2,s_3)},\\
 &\qquad y_3(s_3^+)=y_3(s_3)+\beta_5\,\langle y_1\rangle_{g_1,[s_2,s_3)},\\
 s_4:
 &\qquad y_1(s_4^+)=\beta_1\,\langle y_2\rangle_{g_2,[s_3,s_4)},\\
 &\qquad y_2(s_4^+)=0,
 \qquad y_3(s_4^+)=0.
\end{align*}
In the FreeFEM++ benchmark the coefficients are
\[
 \beta_1=0.65,
 \qquad \beta_2=8.00,
 \qquad \beta_3=25.00,
 \qquad \beta_4=2.50,
 \qquad \beta_5=6.00.
\]
These values are deliberately moderate so that the benchmark remains numerically stable and useful for validation.

\subsection{Control configuration and cost functional}

The validation test uses two traps,
\[
 u=(\tau_1,\tau_2,x_1,x_2,y_1,y_2)\in\mathbb R^6.
\]
The reference control used in the gradient test is
\[
 u=(36,96,13,18,11,12),
\]
where times are measured in days and spatial coordinates in kilometres. The cost used for validation is
\begin{equation}\label{eq:academic-cost}
 J_h(C)=\frac12\sum_{n=1}^{N_t}\left(
 w_1\Delta g_1^n(c_1^n)^\top Mc_1^n
 +w_2\Delta g_2^n(c_2^n)^\top Mc_2^n
 +w_3\Delta g_3^n(c_3^n)^\top Mc_3^n
 \right).
\end{equation}
No explicit penalty on the control is included in the first validation test, so that the test isolates the differentiation of the controlled mortality and of the state equation.

\subsection{Gradient validation results}\label{subsec:gradient-validation-protocol}

The FreeFEM++ benchmark was evaluated with the biological calendar described above. The objective value at the reference control is
\[
 J_h(u)= \num{3.189955e+08}.
\]
The purpose of the test is twofold. First, we verify that the derivative obtained by solving the linearized state equation coincides with the derivative obtained from the discrete adjoint system. Second, we verify both values against forward finite differences.

For the reference direction
\[
 h=(1,-0.7,0.25,-0.40,-0.30,0.20),
\]
the two derivative evaluations give
\[
 J_h'(u)h\big|_{\rm lin}=\num{4.805658e+06},
 \qquad
 J_h'(u)h\big|_{\rm adj}=\num{4.805658e+06}.
\]
Their absolute difference is \num{2.793968e-09}, corresponding to a relative discrepancy of \num{5.813914e-16}. This agreement, at essentially machine precision, confirms that the discrete adjoint equation is the transpose of the implemented linearized state system.

The directional finite-difference quotient
\[
 FD(\varepsilon)=\frac{J_h(u+\varepsilon h)-J_h(u)}{\varepsilon}
\]
is reported in Table~\ref{tab:directional-gradient-validation}. The error with respect to the adjoint derivative decays linearly with \(\varepsilon\), whereas the last column remains essentially constant. Hence
\[
 J_h(u+\varepsilon h)
 =
 J_h(u)+\varepsilon J_h'(u)h+O(\varepsilon^2),
\]
as expected.

\begin{table}[htbp]
\centering
\scriptsize
\begin{tabular}{r r r r}
\toprule
$\varepsilon$ & $FD(\varepsilon)$ & $|FD(\varepsilon)-J_h'(u)h_{\rm adj}|$ & $|FD-J_h'h|/\varepsilon$\\
\midrule
\num{1.000000e-01} & \num{4.688695e+06} & \num{1.169622e+05} & \num{1.169622e+06}\\
\num{5.000000e-02} & \num{4.747033e+06} & \num{5.862471e+04} & \num{1.172494e+06}\\
\num{1.000000e-02} & \num{4.793932e+06} & \num{1.172569e+04} & \num{1.172569e+06}\\
\num{5.000000e-03} & \num{4.799795e+06} & \num{5.862184e+03} & \num{1.172437e+06}\\
\num{1.000000e-03} & \num{4.804485e+06} & \num{1.172309e+03} & \num{1.172309e+06}\\
\num{5.000000e-04} & \num{4.805071e+06} & \num{5.861457e+02} & \num{1.172291e+06}\\
\num{1.000000e-04} & \num{4.805540e+06} & \num{1.172276e+02} & \num{1.172276e+06}\\
\num{5.000000e-05} & \num{4.805599e+06} & \num{5.861775e+01} & \num{1.172355e+06}\\
\bottomrule
\end{tabular}
\caption{Directional finite-difference validation for the academic benchmark.}
\label{tab:directional-gradient-validation}
\end{table}

We also computed the six components of the reduced gradient. For each component \(u_\ell\), we compare the derivative obtained by the linearized equation, by the adjoint equation, and by the forward finite-difference quotient
\[
 FD_\ell(\varepsilon)=
 \frac{J_h(u+\varepsilon e_\ell)-J_h(u)}{\varepsilon}.
\]
Table~\ref{tab:componentwise-gradient-validation} reports the finite-difference value for the finest step,
\[
 \varepsilon_{\min}=\num{5.000000e-05}.
\]
The agreement between \(g_{\ell,\rm lin}\) and \(g_{\ell,\rm adj}\) is again at machine precision. The finite-difference values are not expected to match to machine precision, because they contain a first-order truncation error; nevertheless they converge to the adjoint gradient as \(\varepsilon\to0\).

\begin{table}[htbp]
\centering
\tiny
\begin{tabular}{c r r r r r r}
\toprule
Component & $u_\ell$ & $g_{\ell,\rm lin}$ & $g_{\ell,\rm adj}$ & $FD_\ell(\varepsilon_{\min})$ & $|g_{\ell,\rm lin}-g_{\ell,\rm adj}|$ & $|FD_\ell-g_{\ell,\rm adj}|$\\
\midrule
$\tau_{1}$ & \num{3.600000e+01} & \num{5.066595e+05} & \num{5.066595e+05} & \num{5.066127e+05} & \num{1.804437e-09} & \num{4.679883e+01}\\
$\tau_{2}$ & \num{9.600000e+01} & \num{5.325778e+02} & \num{5.325778e+02} & \num{5.324519e+02} & \num{7.844392e-12} & \num{1.259770e-01}\\
$x_{1}$ & \num{1.300000e+01} & \num{1.348043e+07} & \num{1.348043e+07} & \num{1.348042e+07} & \num{3.725290e-09} & \num{7.805422e+00}\\
$x_{2}$ & \num{1.800000e+01} & \num{6.654177e+04} & \num{6.654177e+04} & \num{6.654158e+04} & \num{1.600711e-10} & \num{1.914204e-01}\\
$y_{1}$ & \num{1.100000e+01} & \num{-3.187700e+06} & \num{-3.187700e+06} & \num{-3.187651e+06} & \num{4.190952e-09} & \num{4.960814e+01}\\
$y_{2}$ & \num{1.200000e+01} & \num{-2.145233e+03} & \num{-2.145233e+03} & \num{-2.144893e+03} & \num{2.273737e-12} & \num{3.393230e-01}\\
\bottomrule
\end{tabular}
\caption{Componentwise gradient validation for the academic benchmark. The control ordering is \(u=(\tau_1,\tau_2,x_1,x_2,y_1,y_2)\).}
\label{tab:componentwise-gradient-validation}
\end{table}

Finally, Figure~\ref{fig:componentwise-fd-errors} displays the finite-difference errors
\[
 e_\ell(\varepsilon)=|FD_\ell(\varepsilon)-g_{\ell,{\rm adj}}|
\]
in log-log scale. The nearly parallel first-order decay of all curves is the expected behaviour for a forward finite-difference quotient and provides an experimental validation of the adjoint gradient formula.

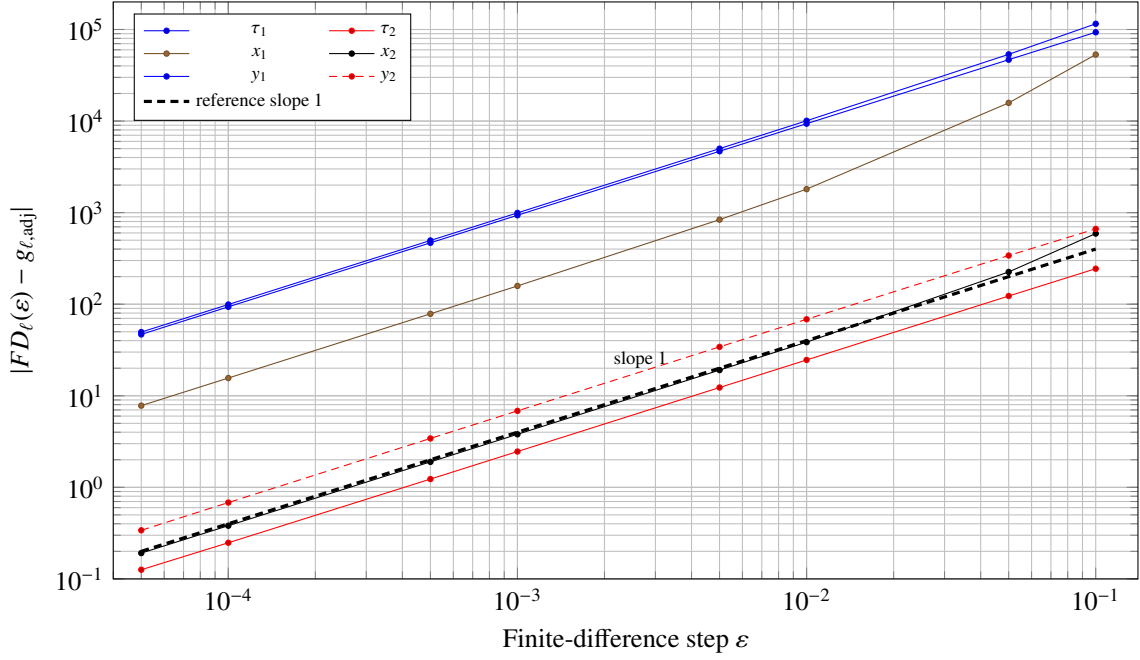
\begin{figure}[htbp]
\centering
\begin{tikzpicture}
\begin{loglogaxis}[
    width=0.92\textwidth,
    height=0.56\textwidth,
    xlabel={Finite-difference step $\varepsilon$},
    ylabel={$|FD_\ell(\varepsilon)-g_{\ell,\rm adj}|$},
    grid=both,
    legend style={at={(0.02,0.98)},anchor=north west,font=\scriptsize},
    legend columns=2,
    xmin=4e-5,
    xmax=1.4e-1,
    ymin=1e-1,
    ymax=2e5,
]
\addplot+[mark=*, mark size=1pt] coordinates {(5.000000e-05,4.679883e+01) (1.000000e-04,9.358967e+01) (5.000000e-04,4.679558e+02) (1.000000e-03,9.359253e+02) (5.000000e-03,4.680200e+03) (1.000000e-02,9.361531e+03) (5.000000e-02,4.679219e+04) (1.000000e-01,9.324674e+04)};
\addlegendentry{$\tau_{1}$}
\addplot+[mark=*, mark size=1pt] coordinates {(5.000000e-05,1.259770e-01) (1.000000e-04,2.481665e-01) (5.000000e-04,1.230928e+00) (1.000000e-03,2.460989e+00) (5.000000e-03,1.230493e+01) (1.000000e-02,2.460845e+01) (5.000000e-02,1.228134e+02) (1.000000e-01,2.441990e+02)};
\addlegendentry{$\tau_{2}$}
\addplot+[mark=*, mark size=1pt] coordinates {(5.000000e-05,7.805422e+00) (1.000000e-04,1.561959e+01) (5.000000e-04,7.854421e+01) (1.000000e-03,1.582182e+02) (5.000000e-03,8.384593e+02) (1.000000e-02,1.805566e+03) (5.000000e-02,1.579450e+04) (1.000000e-01,5.321769e+04)};
\addlegendentry{$x_{1}$}
\addplot+[mark=*, mark size=1pt] coordinates {(5.000000e-05,1.914204e-01) (1.000000e-04,3.809632e-01) (5.000000e-04,1.896828e+00) (1.000000e-03,3.794700e+00) (5.000000e-03,1.907293e+01) (1.000000e-02,3.850271e+01) (5.000000e-02,2.250773e+02) (1.000000e-01,5.906643e+02)};
\addlegendentry{$x_{2}$}
\addplot+[mark=*, mark size=1pt] coordinates {(5.000000e-05,4.960814e+01) (1.000000e-04,9.922484e+01) (5.000000e-04,4.964545e+02) (1.000000e-03,9.936615e+02) (5.000000e-03,4.998543e+03) (1.000000e-02,1.007334e+04) (5.000000e-02,5.353559e+04) (1.000000e-01,1.154117e+05)};
\addlegendentry{$y_{1}$}
\addplot+[mark=*, mark size=1pt] coordinates {(5.000000e-05,3.393230e-01) (1.000000e-04,6.820497e-01) (5.000000e-04,3.419691e+00) (1.000000e-03,6.839984e+00) (5.000000e-03,3.420727e+01) (1.000000e-02,6.842098e+01) (5.000000e-02,3.401046e+02) (1.000000e-01,6.643773e+02)};
\addlegendentry{$y_{2}$}
\addplot[black, very thick, densely dashed, mark=none] coordinates {(5e-5,2e-1) (1e-1,4e2)};
\addlegendentry{reference slope $1$}
\node[font=\scriptsize, anchor=south west] at (axis cs:2e-3,1.6e1) {slope $1$};
\end{loglogaxis}
\end{tikzpicture}
\caption{Componentwise finite-difference validation of the discrete adjoint gradient. The dashed black line has slope one in log--log scale. The observed decay is consistent with \( |FD_\ell(\varepsilon)-g_{\ell,\rm adj}|=O(\varepsilon) \), and therefore with a second-order Taylor remainder.}
\label{fig:componentwise-fd-errors}
\end{figure}

\subsection{Timing comparison for gradient evaluations}\label{subsec:gradient-timing}

The previous subsection validates the \emph{correctness} of the gradient. A separate timing benchmark assesses the \emph{computational usefulness} of the adjoint formulation in the actual FreeFEM++ implementation. This test was carried out independently of a full IPOPT run, because a complete optimization mixes gradient evaluations with objective evaluations, trial points generated by the line search and stopping criteria. Such a run is useful to solve the control problem, but it is not the cleanest way to measure the cost of a single gradient evaluation.

For a fixed mesh, time grid, number of traps and control vector \(u\), the diagnostic benchmark computes
\[
 \nabla j_{\rm adj}(u)
 \qquad\text{and}\qquad
 \nabla j_{\rm dir}(u),
\]
where \(\nabla j_{\rm adj}\) is obtained from the discrete adjoint state and \(\nabla j_{\rm dir}\) is obtained by solving one linearized state equation per scalar control component. Both routines use the same state solver and the same trapping field. The recorded quantities are
\[
 T_{\rm adj},\qquad T_{\rm dir},\qquad
 \frac{T_{\rm dir}}{T_{\rm adj}},
\]
together with the discrepancy between the two gradient vectors. The last quantity is essential: a speed-up is meaningful only if both routines return the same gradient up to the tolerances already observed in the validation test.

The adjoint implementation used in this benchmark solves the state and adjoint equations once and then evaluates all \(3N_T\) scalar gradient components by direct contractions with the temporal and spatial derivatives of the trapping kernels. Thus, the timings below correspond to the adjoint gradient routine used in the optimization experiments.

Table~\ref{tab:gradient-timing-battery} reports the timing battery for two representative control vectors at each value of $N_T$: a reference vector and the initial vector used in the corresponding optimization test. Because their dimensions change with $N_T$, objective values in different $N_T$ rows are not evaluations of one common control and should not be compared as such. The gradient discrepancy is at the level of round-off error in all cases. In fact, the largest relative difference between the direct and adjoint gradients is
\[
 2.51\cdot 10^{-15},
\]
and the largest componentwise absolute difference is below \(6.4\cdot 10^{-8}\). Thus, the timing battery provides an additional, independent verification that the adjoint routine and the gradients obtained from the linearized state equations compute the same discrete reduced gradient.

\begin{table}[htbp]
\centering
\scriptsize
\begin{tabular}{llrrrrr}
\toprule
Control vector & \(N_T\) & \(\dim u\) & \(J(u)\) & \(T_{\rm adj}\) (s) & \(T_{\rm dir}\) (s) & \(T_{\rm dir}/T_{\rm adj}\) \\
\midrule
Reference & 1 & 3  & \(2.62349\cdot 10^8\) & 18.63 & 139.14 & 7.47 \\
Initial   & 1 & 3  & \(3.06399\cdot 10^8\) & 18.38 & 147.06 & 8.00 \\
Reference & 2 & 6  & \(2.62344\cdot 10^8\) & 34.63 & 622.42 & 17.98 \\
Initial   & 2 & 6  & \(3.06382\cdot 10^8\) & 39.86 & 618.01 & 15.51 \\
Reference & 3 & 9  & \(2.60484\cdot 10^8\) & 59.07 & 1695.71 & 28.71 \\
Initial   & 3 & 9  & \(3.05200\cdot 10^8\) & 71.19 & 1388.41 & 19.50 \\
Reference & 4 & 12 & \(2.53150\cdot 10^8\) & 81.29 & 2487.19 & 30.59 \\
Initial   & 4 & 12 & \(3.01951\cdot 10^8\) & 79.99 & 2467.54 & 30.85 \\
\bottomrule
\end{tabular}
\caption{Measured timing comparison between the adjoint gradient evaluation and the direct linearized-state evaluation. The reported speed-up is \(T_{\rm dir}/T_{\rm adj}\); values above one favour the adjoint implementation. For all rows, the relative norm difference between the two gradient vectors is below \(2.51\cdot10^{-15}\).}
\label{tab:gradient-timing-battery}
\end{table}

Figure~\ref{fig:gradient-timing-battery} plots the same speed-up values against the number of traps. The measured behaviour agrees with the algebraic operation count of Subsection~\ref{subsec:adjoint-computational-implication}: the direct approach becomes increasingly expensive because it requires one linearized Stieltjes--parabolic solve for each scalar control component, whereas the adjoint approach reuses a single adjoint state to evaluate all components. The observed speed-up ranges from about \(7.5\) to \(8.0\) for one trap and reaches about \(30.6\) to \(30.9\) for four traps.

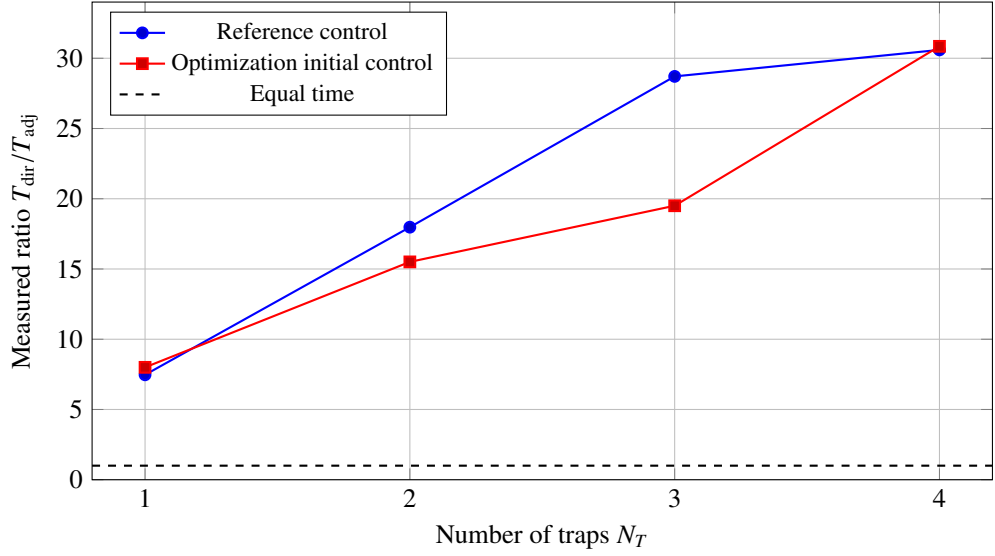
\begin{figure}[htbp]
\centering
\begin{tikzpicture}
\begin{axis}[
    width=0.82\textwidth,
    height=0.48\textwidth,
    xlabel={Number of traps $N_T$},
    ylabel={Measured ratio $T_{\rm dir}/T_{\rm adj}$},
    grid=both,
    xmin=0.8, xmax=4.2,
    ymin=0, ymax=34,
    xtick={1,2,3,4},
    legend style={at={(0.02,0.98)},anchor=north west,font=\small},
]
\addplot+[mark=*, thick] coordinates {(1,7.4676) (2,17.9760) (3,28.7086) (4,30.5948)};
\addlegendentry{Reference control}
\addplot+[mark=square*, thick] coordinates {(1,7.9993) (2,15.5054) (3,19.5022) (4,30.8494)};
\addlegendentry{Optimization initial control}
\addplot[black, dashed, thick, mark=none] coordinates {(0.8,1) (4.2,1)};
\addlegendentry{Equal time}
\end{axis}
\end{tikzpicture}
\caption{Measured speed-up of the direct linearized-state gradient evaluation relative to the adjoint gradient evaluation. The adjoint gradient remains numerically indistinguishable from the gradient obtained from the linearized state equations, while its computational advantage increases with the number of traps.}
\label{fig:gradient-timing-battery}
\end{figure}

The conclusion of the timing benchmark is therefore twofold. First, the adjoint equation is validated by an independent timing test because it reproduces the gradients obtained from the linearized state equations with relative discrepancies of order \(10^{-15}\). Second, the measured cost confirms the practical advantage of the adjoint formulation for multi-trap optimization. Averaging the two controls tested for each dimension gives speed-ups of approximately \(7.73\), \(16.74\), \(24.11\), and \(30.72\) for \(N_T=1,2,3,4\), respectively. Consequently, all optimization experiments use the discrete adjoint state to evaluate the reduced gradient supplied to IPOPT.

The optimization driver records IPOPT's convergence diagnostics. It also evaluates the projected residual \eqref{eq:projected-gradient-residual} and a scale-normalized version. These quantities are used jointly because a substantial reduction of \(J\) alone is not an optimality certificate.

\FloatBarrier
\section{Numerical optimization experiments}\label{sec:numerical-optimization}

The validated adjoint gradient was then supplied to IPOPT~\cite{WachterBiegler2006IPOPT}, following the standard reduced-space logic of PDE-constrained optimization~\cite{HinzeEtAl2009,NocedalWright2006}. We performed a multi-start battery of box-constrained runs with
\[
 N_T=1,2,3,4.
\]
All cases use the mesh, time grid, biological calendar and parameters of Section~\ref{sec:academic-validation}. The control ordering is
\[
 u=(\tau_1,\ldots,\tau_{N_T},x_1,\ldots,x_{N_T},y_1,\ldots,y_{N_T}),
\]
and the reduced gradient supplied to IPOPT is computed through the discrete adjoint state. The box constraints are
\[
 10\leq \tau_k\leq135,\qquad 4\leq x_k\leq36,\qquad 4\leq y_k\leq20.
\]
The uncontrolled reference values are
\begin{equation}\label{eq:optimization-uncontrolled-reference}
 J_{\rm no\ trap}=3.366833739480674\times 10^8,
 \qquad
 M_{\rm final}^{\rm no\ trap}=1933.704130426528,
\end{equation}
where
\[
 M_{\rm final}=\int_\Omega
 \bigl(y_1(s_5,x)+y_2(s_5,x)+y_3(s_5,x)\bigr)\,\dd x.
\]
No installation cost or explicit penalty on the number of traps is included. Consequently, comparisons across $N_T$ quantify the biological effect of allowing more fixed-intensity devices and do not constitute an economic cost--benefit analysis.

\subsection{Multi-start design and convergence}

For each $N_T$, eight prescribed initial controls were considered. Their identifiers are
\begin{center}
\small
\texttt{default}, \texttt{biological}, \texttt{early\_focus}, \texttt{late\_focus},\\
\texttt{uniform\_near\_focus}, \texttt{uniform\_wide},
\texttt{random\_focus\_06}, and \texttt{random\_wide\_07}.
\end{center}
They cover early and late activation patterns, clustered and spatially spread centres, and two reproducible random configurations. The purpose is not to certify global optimality, but to assess sensitivity to initialization and identify repeatedly attained good local solutions.

Of the 32 initial executions, 22 satisfied IPOPT's termination criterion within the original iteration allowance. The ten unfinished cases were restarted from their last saved controls with a maximum of 200 iterations. All ten restarts then terminated with \texttt{Optimal Solution Found}, so every initial control has a final representative with successful IPOPT termination. The canonical analysis was reconstructed from the final \texttt{multistart\_result.dat} file of each run. This avoids dependence on intermediate aggregate summaries.

Besides IPOPT's termination status, the driver evaluates $R_1^{\rm phys}$ and $\widehat R_1^{\rm phys}$ from \eqref{eq:projected-gradient-residual}. The displayed values are therefore unit-dependent physical-coordinate diagnostics, not IPOPT optimality measures.

\subsection{Best values and dependence on initialization}

Table~\ref{tab:ipopt-optimization-results} gives the best objective value found among the eight starts for each $N_T$. The final-mass reduction is a complementary diagnostic and must not be confused with the objective reduction: $J_h$ is accumulated over the biological generation, whereas $M_{\rm final}$ is a single terminal quantity.

\begin{table}[htbp]
\centering
\small
\caption{Best values found among eight prescribed initial controls for each number of traps. The last column reports the external unit-dependent diagnostic $\widehat R_1^{\rm phys}$ from \eqref{eq:projected-gradient-residual}; it is not IPOPT's dual infeasibility.}
\label{tab:ipopt-optimization-results}
\begin{tabular}{c l r r r r}
\toprule
$N_T$ & Best start & $J_h(u^*)$ & $\Delta J$ & $M_{\rm final}$ & $\widehat R_1^{\rm phys}$ \\
\midrule
1 & \texttt{late\_focus}  & $1.797123\times10^8$ & 46.62\% & 1445.223 & $1.28\times10^{-4}$ \\
2 & \texttt{early\_focus} & $1.165302\times10^8$ & 65.39\% & 1161.290 & $3.14\times10^{-3}$ \\
3 & \texttt{early\_focus} & $7.747849\times10^7$ & 76.99\% & 1000.648 & $1.27\times10^{-3}$ \\
4 & \texttt{early\_focus} & $5.852500\times10^7$ & 82.62\% &  883.369 & $2.28\times10^{-3}$ \\
\bottomrule
\end{tabular}
\end{table}

The best multi-start values improve substantially on the solutions obtained from the single default initialization. Relative to the preceding best case, adding the second, third and fourth devices reduces $J_h$ by approximately $35.16\%$, $33.51\%$, and $24.46\%$, respectively. The corresponding marginal reductions of terminal mass are $19.65\%$, $13.83\%$, and $11.72\%$. Thus, the fourth device remains beneficial, but the terminal diagnostic exhibits decreasing marginal returns. Each additional trap also adds the same prescribed mortality intensity, so the monotone improvement across $N_T$ is expected and cannot identify an optimal device count without an explicit deployment cost or shared resource constraint.

The objective values vary markedly across starts. Table~\ref{tab:multistart-dispersion} summarizes this dispersion. Several starts reach values within $5\%$ of the best result, but other starts converge to substantially worse stationary points.

\begin{table}[htbp]
\centering
\small
\caption{Dispersion of the eight successfully terminated multi-start objective values for each $N_T$.}
\label{tab:multistart-dispersion}
\begin{tabular}{c r r r c}
\toprule
$N_T$ & Minimum $J_h$ & Median $J_h$ & Maximum $J_h$ & Starts within $5\%$ \\
\midrule
1 & $1.797123\times10^8$ & $1.932955\times10^8$ & $2.192521\times10^8$ & 3 \\
2 & $1.165302\times10^8$ & $1.550658\times10^8$ & $1.906389\times10^8$ & 3 \\
3 & $7.747849\times10^7$ & $1.070188\times10^8$ & $2.100519\times10^8$ & 3 \\
4 & $5.852500\times10^7$ & $7.633063\times10^7$ & $9.183887\times10^7$ & 3 \\
\bottomrule
\end{tabular}
\end{table}

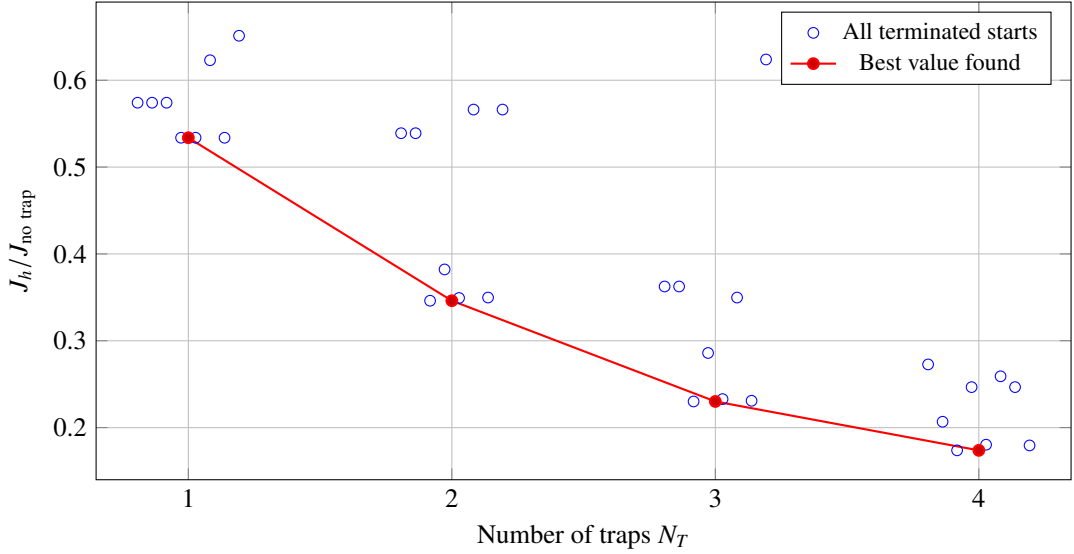
\begin{figure}[htbp]
\centering
\begin{tikzpicture}
\begin{axis}[
    width=0.88\textwidth,
    height=0.48\textwidth,
    xlabel={Number of traps $N_T$},
    ylabel={$J_h/J_{\rm no\ trap}$},
    xmin=0.65,xmax=4.35,
    ymin=0.14,ymax=0.69,
    xtick={1,2,3,4},
    grid=major,
    legend style={at={(0.98,0.98)},anchor=north east,font=\small},
]
\addplot+[only marks,mark=o,mark size=2pt] table[x=x_jitter,y=J_over_J0] {multistart_objectives_jitter.dat};
\addlegendentry{All terminated starts}
\addplot+[thick,mark=*,mark size=2pt] coordinates {(1,0.53377245) (2,0.34611206) (3,0.23012271) (4,0.17382801)};
\addlegendentry{Best value found}
\end{axis}
\end{tikzpicture}
\caption{Dependence of the final objective on the initial control. Horizontal jitter is used only to separate the eight starts for each $N_T$. The wide vertical dispersion provides numerical evidence of several distinct locally converged objective levels and of nonconvex behaviour; the lower envelope is not a certificate of global optimality.}
\label{fig:multistart-objective-dispersion}
\end{figure}

The repeated attainment of closely related low values is encouraging: three starts are within $5\%$ of the best result for every $N_T$. Nevertheless, the range between the best and worst final values is too large to describe the reduced problem as effectively unimodal. The results must therefore be presented as the best local solutions found by the prescribed battery rather than as certified global optima. In addition, permutation of trap labels leaves the control field unchanged, so controls are ordered by activation time before being compared.

\subsection{Evolution of the total population mass}\label{subsec:total-mass-evolution}

For a control $u$, define the discrete total mass
\begin{equation}\label{eq:discrete-total-mass}
 M^n(u)=\int_\Omega\bigl(c_{1,h}^n(u)+c_{2,h}^n(u)+c_{3,h}^n(u)\bigr)\,\dd x,
 \qquad n=0,\ldots,N.
\end{equation}
The state equation was solved again at the four best controls in Table~\ref{tab:ipopt-optimization-results}. The recomputed terminal masses agree with the canonical optimization outputs to numerical roundoff at the displayed scale.

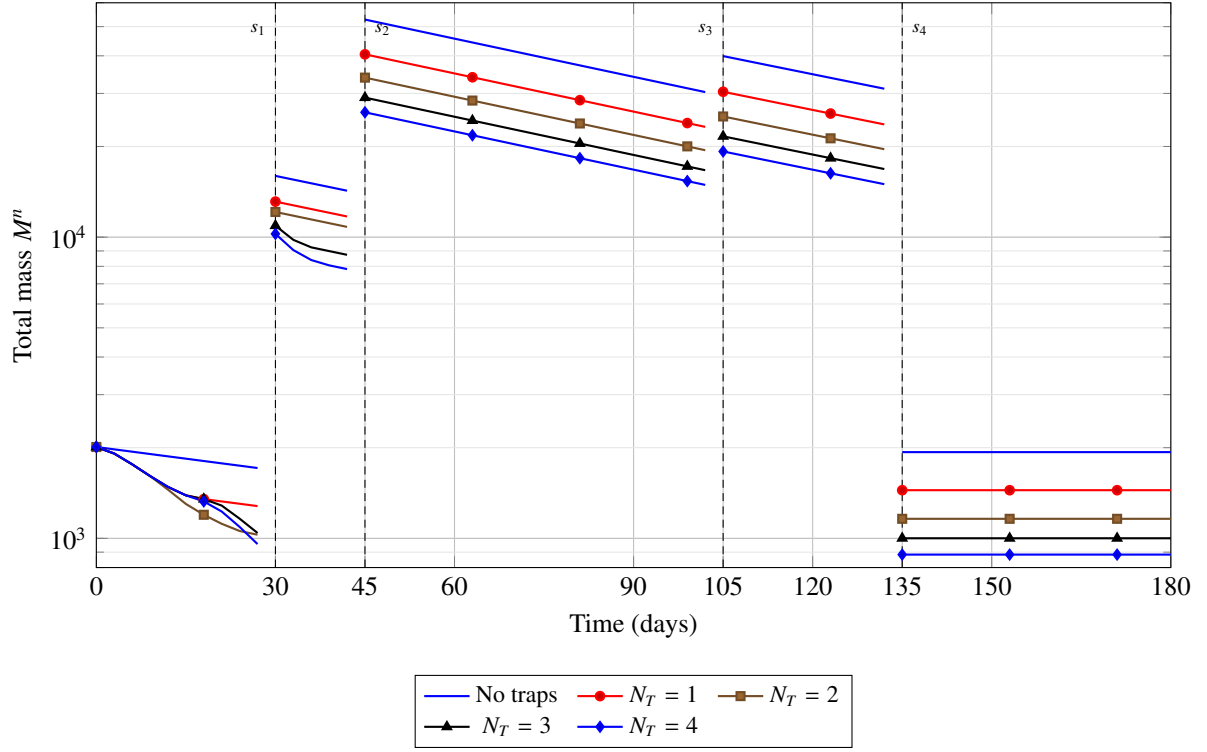
\begin{figure}[htbp]
\centering
\begin{tikzpicture}
\begin{axis}[
    width=0.96\textwidth,
    height=0.55\textwidth,
    xlabel={Time (days)},
    ylabel={Total mass $M^n$},
    xmin=0,xmax=180,
    ymin=8e2,ymax=6e4,
    ymode=log,
    log basis y=10,
    xtick={0,30,45,60,90,105,120,135,150,180},
    grid=both,
    minor grid style={draw=gray!20},
    unbounded coords=jump,
    legend columns=3,
    legend style={at={(0.5,-0.19)},anchor=north,font=\small,/tikz/every even column/.append style={column sep=0.5em}},
]
\addplot+[thick,mark=none] table[x=time_days,y=no_traps] {mass_evolution_comparison_breaks.dat};
\addlegendentry{No traps}
\addplot+[thick,mark=*,mark repeat=6,mark size=1.6pt] table[x=time_days,y=NT1] {mass_evolution_comparison_breaks.dat};
\addlegendentry{$N_T=1$}
\addplot+[thick,mark=square*,mark repeat=6,mark size=1.5pt] table[x=time_days,y=NT2] {mass_evolution_comparison_breaks.dat};
\addlegendentry{$N_T=2$}
\addplot+[thick,mark=triangle*,mark repeat=6,mark size=1.8pt] table[x=time_days,y=NT3] {mass_evolution_comparison_breaks.dat};
\addlegendentry{$N_T=3$}
\addplot+[thick,mark=diamond*,mark repeat=6,mark size=1.8pt] table[x=time_days,y=NT4] {mass_evolution_comparison_breaks.dat};
\addlegendentry{$N_T=4$}
\addplot[black,densely dashed,thin,forget plot] coordinates {(30,8e2) (30,6e4)};
\addplot[black,densely dashed,thin,forget plot] coordinates {(45,8e2) (45,6e4)};
\addplot[black,densely dashed,thin,forget plot] coordinates {(105,8e2) (105,6e4)};
\addplot[black,densely dashed,thin,forget plot] coordinates {(135,8e2) (135,6e4)};
\node[font=\scriptsize,anchor=north east] at (axis cs:30,5.5e4) {$s_1$};
\node[font=\scriptsize,anchor=north west] at (axis cs:45,5.5e4) {$s_2$};
\node[font=\scriptsize,anchor=north east] at (axis cs:105,5.5e4) {$s_3$};
\node[font=\scriptsize,anchor=north west] at (axis cs:135,5.5e4) {$s_4$};
\end{axis}
\end{tikzpicture}
\caption{Evolution of the total finite-element mass for the uncontrolled state and the best multi-start controls with one to four traps. Curve segments are not joined across biological events, so the pre-event and post-event values remain visually separated at each prescribed jump.}
\label{fig:total-mass-evolution}
\end{figure}

The new trajectories differ qualitatively from those obtained from the default starts. Every best control contains a trap centred at the lower admissible time $\tau=10$ days, so the controlled curves separate from the uncontrolled state almost immediately. For $N_T=2$, the second activation at day $19.21$ strengthens removal before $s_1$. For $N_T=3$, two co-located devices centred at day $28.57$ have temporal support overlapping with $s_1$, while the four-trap solution distributes three additional centres over days $25.45$, $28.61$, and $31.52$. The reductions accumulated during the foundress phase are inherited by the worker-production event at $s_1$ and then propagated through the later impulsive transitions. Consequently, the curves remain ordered throughout the active season even though the best controls contain no separate late intervention near $s_3$.

At $s_4$, the benchmark resets the future-foundress and worker/male compartments and replaces the foundress compartment by a fraction of the averaged future-foundress population. The foundress Stieltjes clock is then constant during hibernation. Hence every trajectory is constant on $[s_4,s_5]$, and the terminal values in Table~\ref{tab:ipopt-optimization-results} are the post-reset foundress masses transferred to the next generation.

\subsection{Best activation times and centres}\label{subsec:optimal-times-biological-cycle}

The representative controls associated with the best objective values are listed in Table~\ref{tab:optimal-controls}. They were sorted by activation time to remove the irrelevant permutation symmetry of identical devices.

\begin{table}[htbp]
\centering
\small
\caption{Best controls found by the multi-start battery. The last column is the distance from the initial Gaussian centre $z_0=(8,12)$, in kilometres.}
\label{tab:optimal-controls}
\begin{tabular}{c c r r r r}
\toprule
$N_T$ & Trap & $\tau_k$ (days) & $x_k$ (km) & $y_k$ (km) & $|z_k-z_0|$ (km) \\
\midrule
1 & 1 & 10.0000 & 7.97960 & 11.99951 & 0.02041 \\
\midrule
2 & 1 & 10.0000 & 7.97812 & 11.99954 & 0.02189 \\
2 & 2 & 19.2058 & 7.97611 & 11.99957 & 0.02389 \\
\midrule
3 & 1 & 10.0000 & 7.97805 & 11.99947 & 0.02196 \\
3 & 2 & 28.5730 & 7.97395 & 11.99967 & 0.02605 \\
3 & 3 & 28.5730 & 7.97395 & 11.99967 & 0.02605 \\
\midrule
4 & 1 & 10.0000 & 7.97857 & 11.99925 & 0.02145 \\
4 & 2 & 25.4483 & 7.97043 & 12.00000 & 0.02957 \\
4 & 3 & 28.6122 & 7.97154 & 11.99996 & 0.02846 \\
4 & 4 & 31.5215 & 7.97413 & 11.99950 & 0.02587 \\
\bottomrule
\end{tabular}
\end{table}

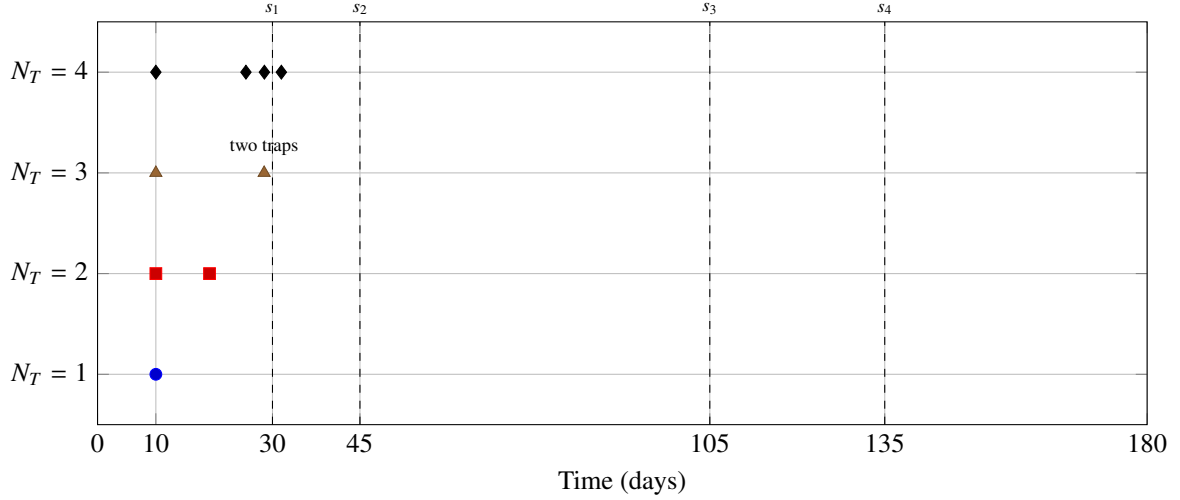
\begin{figure}[htbp]
\centering
\begin{tikzpicture}
\begin{axis}[
    width=0.94\textwidth,
    height=0.42\textwidth,
    xmin=0,xmax=180,
    ymin=0.5,ymax=4.5,
    xlabel={Time (days)},
    ytick={1,2,3,4},
    yticklabels={$N_T=1$,$N_T=2$,$N_T=3$,$N_T=4$},
    xtick={0,10,30,45,105,135,180},
    grid=major,
    clip=false,
]
\addplot[dashed,thin,forget plot] coordinates {(30,0.5) (30,4.5)};
\addplot[dashed,thin,forget plot] coordinates {(45,0.5) (45,4.5)};
\addplot[dashed,thin,forget plot] coordinates {(105,0.5) (105,4.5)};
\addplot[dashed,thin,forget plot] coordinates {(135,0.5) (135,4.5)};
\addplot+[only marks,mark=*,mark size=2.2pt] coordinates {(10,1)};
\addplot+[only marks,mark=square*,mark size=2.2pt] coordinates {(10,2) (19.2058,2)};
\addplot+[only marks,mark=triangle*,mark size=2.6pt] coordinates {(10,3) (28.5730,3)};
\node[anchor=south,font=\scriptsize] at (axis cs:28.5730,3.08) {two traps};
\addplot+[only marks,mark=diamond*,mark size=2.6pt] coordinates {(10,4) (25.4483,4) (28.6122,4) (31.5215,4)};
\node[anchor=south,font=\scriptsize] at (axis cs:30,4.48) {$s_1$};
\node[anchor=south,font=\scriptsize] at (axis cs:45,4.48) {$s_2$};
\node[anchor=south,font=\scriptsize] at (axis cs:105,4.48) {$s_3$};
\node[anchor=south,font=\scriptsize] at (axis cs:135,4.48) {$s_4$};
\end{axis}
\end{tikzpicture}
\caption{Activation times of the best multi-start controls relative to the biological events. The temporal kernel radius is $T_M=10$ days. Every representative uses the lower admissible bound, and the additional devices concentrate before or around $s_1$.}
\label{fig:optimal-times-biological-calendar}
\end{figure}

The multi-start timing pattern is therefore earlier than in the single-start experiment. The lower bound $\tau=10$ is active for every $N_T$, indicating that earlier admissible deployment could potentially be beneficial and that the temporal bound itself should be included in future sensitivity studies. The additional devices act before or across $s_1=30$, when foundress removal has the longest remaining time to affect both continuous dynamics and subsequent impulsive production. Some worse local minima retain interventions near $s_2$ or $s_3$, but these late windows are not present in the best values found by the reported battery.

All centres remain within $0.03$ km of $z_0=(8,12)$, far below the kernel radius $R_M=4$ km. This is consistent with the single radially concentrated academic focus and confirms that the benchmark is primarily informative about timing. In the best three-trap control, two traps coincide to numerical precision at day $28.57$ and at the same centre. The additive mortality law contains no separation, saturation, or overlap penalty, so co-location represents increased local intensity rather than a numerical inconsistency. Distinct physical deployment sites would require explicit separation constraints or a different resource model.

\subsection{Scope of the multi-start evidence}

The multi-start battery materially strengthens the numerical study, but it does not prove global optimality. Eight initial controls sample only a finite subset of the admissible set, and the wide distribution in Figure~\ref{fig:multistart-objective-dispersion} provides numerical evidence of nonconvex behaviour. The best values should therefore be described as ``best found'' or ``best multi-start solutions''. Moreover, because no deployment cost is included, the decreasing objective as $N_T$ increases does not determine an operationally optimal number of traps. Kernel-parameter sensitivity, explicit costs or resource constraints, and minimum-separation or saturation effects remain necessary before management interpretation.

\FloatBarrier
\section{Calibration and held-out temporal evaluation of spatial fields on the Galicia domain}\label{sec:galicia-validation}

The academic benchmark verifies the differentiation and optimization machinery, but its parameters and initial condition were not inferred from field data. We therefore carried out a separate Galicia experiment whose purpose is narrower: to test whether the uncontrolled Stieltjes model transfers broad spatial information from one annual nest distribution to the next. Following the distinction between calibration and evaluation in ecological modelling, the purpose, performance criteria, and intended use are stated before the results are interpreted~\cite{Rykiel1996}. The experiment is not used to certify the model as a complete representation of the invasion process. It contains one calibration transition and one held-out transition, so the reported metrics are descriptive rather than inferential.

\subsection{Galicia derivators over an arbitrary number of generations}\label{subsec:galicia-derivators}

For the Galicia experiment we use the same compartment-dependent Stieltjes calendar as in the uncontrolled \textit{Vespa velutina} model of~\cite{AreaVelutina2025}, now repeated over several generations. Let $N_G\in\N$ be the number of generations, let $\lambda_1,\ldots,\lambda_5>0$ denote the phase lengths of foundress dispersal, primary-nest activity, secondary-nest activity, future-foundress activity/dispersal, and hibernation, and set
\[
 L:=\lambda_1+\lambda_2+\lambda_3+\lambda_4+\lambda_5,
 \qquad
 T_{k,0}:=(k-1)L,
 \qquad
 T_{k,m}:=T_{k,0}+\sum_{j=1}^{m}\lambda_j,
\]
for $k=1,\ldots,N_G$ and $m=1,\ldots,5$. We also write
\[
 R_{a,b}(t):=\max\{0,\min\{t-a,\,b-a\}\},\qquad a<b,
\]
for the truncated ramp that accumulates unit Stieltjes time on $(a,b]$ and remains constant outside that interval. Then the three derivators used on the Galicia domain are
\begin{align}
 g_1^{(N_G)}(t)
 &=\sum_{k=1}^{N_G}\Bigl(R_{T_{k,0},T_{k,4}}(t)+\chi_{(T_{k,4},\infty)}(t)\Bigr),\label{eq:galicia-g1}\\
 g_2^{(N_G)}(t)
 &=\sum_{k=1}^{N_G}\Bigl(R_{T_{k,0},T_{k,5}}(t)+\chi_{(T_{k,3},\infty)}(t)+\chi_{(T_{k,4},\infty)}(t)\Bigr),\label{eq:galicia-g2}\\
 g_3^{(N_G)}(t)
 &=\sum_{k=1}^{N_G}\Bigl(R_{T_{k,0},T_{k,5}}(t)+\chi_{(T_{k,1},\infty)}(t)+\chi_{(T_{k,2},\infty)}(t)
 +\chi_{(T_{k,3},\infty)}(t)+\chi_{(T_{k,4},\infty)}(t)\Bigr).\label{eq:galicia-g3}
\end{align}
These formulas are the multi-generation extension of the one-generation hornet calendar used in~\cite{AreaVelutina2025}; compare also the discussion of systems with several derivators in~\cite{LopezPousoMarquez2019}. The foundress clock $g_1^{(N_G)}$ is frozen during each hibernation interval $(T_{k,4},T_{k,5}]$, whereas $g_2^{(N_G)}$ and $g_3^{(N_G)}$ continue to evolve continuously throughout the whole annual cycle. The atomic parts represent abrupt biological transfers: renewal of foundresses at $T_{k,4}$ for $g_1^{(N_G)}$, production and reset events at $T_{k,3}$ and $T_{k,4}$ for $g_2^{(N_G)}$, and all four event times for $g_3^{(N_G)}$.

Figure~\ref{fig:galicia-derivators-two-generations} shows a schematic realization of~\eqref{eq:galicia-g1}--\eqref{eq:galicia-g3} for two generations, using the same phase lengths as the academic benchmark. As in the general Stieltjes framework~\cite{PousoRodriguez2015,FernandezTojo2020Bochner}, the flat pieces encode inactivity and the atoms encode instantaneous transfers without changing the underlying spatial domain.

\begin{figure}[htbp]
\centering
\includegraphics[width=0.96\linewidth]{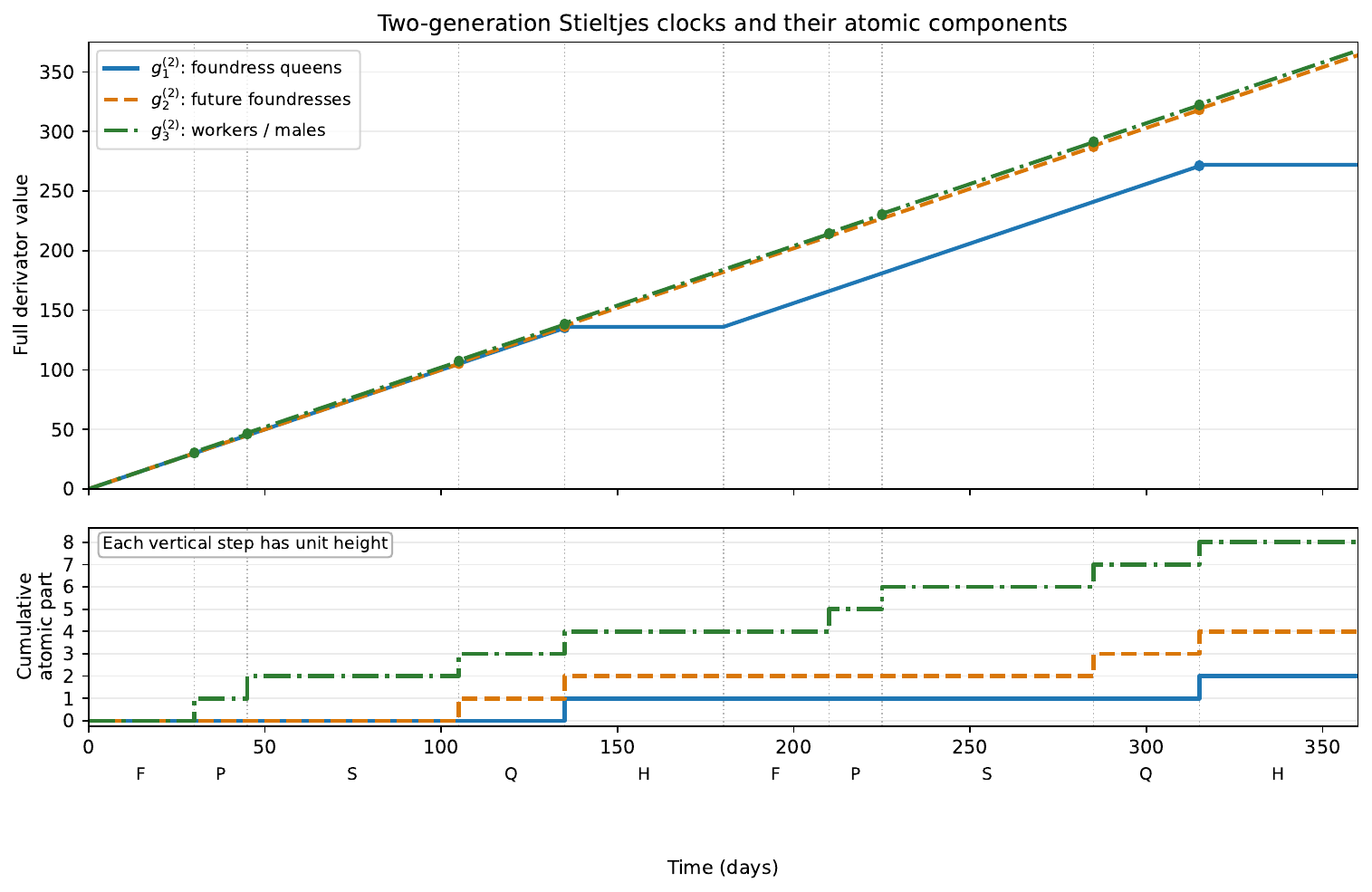}
\caption{Two-generation realization of the Galicia derivators~\eqref{eq:galicia-g1}--\eqref{eq:galicia-g3}. The upper panel shows the full clocks; the lower panel isolates their cumulative atomic components, so every jump is displayed at its true unit height. The letters F, P, S, Q and H denote foundress dispersal, primary-nest activity, secondary-nest activity, future-foundress activity/dispersal and hibernation, respectively. In the upper panel, open and filled markers indicate the left and right values at each jump.}
\label{fig:galicia-derivators-two-generations}
\end{figure}

\subsection{Observational layers and kernel fields}\label{subsec:galicia-kernel-fields}

Two administrative layers are available. The \texttt{neutralized} layer is almost complete in projected coordinates for 2023--2025 and is used for parameter and operator selection. The \texttt{eliminated} layer has a more variable coordinate-completeness history and is retained as a separate robustness layer. The two layers are not merged because they need not represent the same observation and intervention process.

For year $y$, let $\{x_{j,y}\}_{j=1}^{N_y}\subset\Omega$ denote the valid recorded locations. A point cloud cannot be compared directly with a finite-element state. We therefore define the smoothed observation field
\begin{equation}\label{eq:galicia-kde}
 Z_y^{\varepsilon_{\rm obs}}(x)
 =\sum_{j=1}^{N_y}\varphi_{\varepsilon_{\rm obs}}(x-x_{j,y}),
 \qquad \varepsilon_{\rm obs}=2000\ \mathrm{m},
\end{equation}
where $\varphi_{\varepsilon}$ is the compact spatial kernel used in the numerical model. Kernel smoothing is a standard device for converting point observations into a spatial intensity proxy~\cite{Silverman1986}. Here it is used as an observation representation rather than as a claim that the administrative records form a fully observed point process.

For spatial-shape comparisons we normalize
\begin{equation}\label{eq:galicia-normalized-field}
 \widehat Z_y=\frac{Z_y^{\varepsilon_{\rm obs}}}
 {\int_\Omega Z_y^{\varepsilon_{\rm obs}}(x)\,\dd x}.
\end{equation}
The observed bandwidth is fixed for all candidates. This is important: changing the observed bandwidth during parameter selection would change the target field and confound model calibration with observation preprocessing. The value $2$ km is a prescribed representation scale rather than a bandwidth estimate, and no boundary correction is applied; sensitivity to both choices remains to be assessed.

The model initial condition is constructed from the input year's point cloud with a separate radius
\[
 r_{\rm init}\in\{1000,2000,3000,4000,5000\}\ \mathrm{m},
\]
and is globally normalized to a mass proportional to $N_y$. The diffusion grid is
\[
 \nu\in\{0.5,1,2,3,4\}\times10^6\ \mathrm{m^2/month}.
\]
The observation bandwidth and the initial-condition radius thus have distinct roles.

\subsection{Observation operators}\label{subsec:galicia-observation-operators}

Observations and model states live in different spaces. We make the comparison explicit through linear observation operators, in the standard sense of a map from model state space to observation space~\cite{LahozSchneider2014}. For a given annual cycle $k$, set
\[
 t_j:=T_{k,j-1},\qquad j=1,\ldots,6,
\]
so that $t_1=T_{k,0}$ and $t_6=T_{k,5}=T_{k+1,0}$. The cycle index is suppressed below because each operator is applied to a single annual transition. For the uncontrolled state $Y=(y_1,y_2,y_3)$, six candidates are tested:
\begin{align}
 H_F(Y)&=y_1(t_2^-),\label{eq:HF}\\
 H_P(Y)&=\frac1{t_3-t_2}\int_{t_2}^{t_3}y_3(t)\,\dd t,\label{eq:HP}\\
 H_S(Y)&=\frac1{t_4-t_3}\int_{t_3}^{t_4}y_3(t)\,\dd t,\label{eq:HS}\\
 H_{PS}(Y)&=\frac1{t_4-t_2}\int_{t_2}^{t_4}y_3(t)\,\dd t,\label{eq:HPS}\\
 H_Q(Y)&=\frac1{t_5-t_4}\int_{t_4}^{t_5}y_2(t)\,\dd t,\label{eq:HQ}\\
 H_R(Y)&=y_1(t_6).\label{eq:HR}
\end{align}
They represent, respectively, foundresses before the first transfer, average primary-nest activity, average secondary-nest activity, combined nest activity, future-foundress activity, and the end-of-cycle renewal field. The discrete implementation uses the corresponding arithmetic averages on the uniform time grid.

The distinction between $H_F$ and the later operators is structurally important. Before the first biological jump, $y_1$ evolves by diffusion and spatially homogeneous mortality. After normalization, the homogeneous scalar factor disappears. Thus $H_F$ is expected to be nearly indistinguishable from a pure-diffusion baseline. Operators \eqref{eq:HP}--\eqref{eq:HR} test whether calendar-dependent compartmental filtering improves on a scalar-diffusion forecast. In the Galicia implementation, all compartments share the same homogeneous diffusion and mortality coefficients and the transfers are spatially uniform scalars. Hence these operators form temporal mixtures of heat-semigroup states. Their comparison can assess the usefulness of that temporal--compartmental filter, but cannot identify the biological transfer coefficients or validate a spatially heterogeneous biological mechanism.

\subsection{Calibration--validation split and baselines}\label{subsec:galicia-protocol}

Within the calibration--evaluation protocol, all model choices are made from the 2023--2024 transition. The selected operator and parameters are then frozen and evaluated on 2024--2025. We call this a held-out temporal evaluation: 2025 is excluded from the selection rule, but the exercise is not a preregistered blind forecast because earlier exploratory analyses had inspected the same year. The split follows the forecasting principle of separating fitting and evaluation~\cite{DietzeEtAl2018}. It is temporal rather than spatial holdout; therefore it does not remove all effects of spatial autocorrelation, and it should not be interpreted as spatial cross-validation~\cite{RobertsEtAl2017,PlotonEtAl2020}.

Two baselines are used.
\begin{enumerate}[label=(B\arabic*)]
 \item \emph{Persistence}: the normalized field from the input year is used directly as the prediction for the following year, i.e.
 \[
 P_{\rm pers}^{23\to24}=\widehat Z_{2023},
 \qquad
 P_{\rm pers}^{24\to25}=\widehat Z_{2024}.
 \]
 \item \emph{Matched diffusion}: the scalar heat equation
 \begin{equation}\label{eq:galicia-diffusion-baseline}
 \partial_t w-\nu\Delta w=0\quad\hbox{in }\Omega,
 \qquad \partial_{\bm n}w=0\quad\hbox{on }\partial\Omega,
 \end{equation}
 starts from the same initial field as the biological model. Its solution is sampled or averaged over the same temporal window as the corresponding biological observation operator.
\end{enumerate}
Persistence asks whether the model improves on the strongest immediate empirical predictor, whereas matched diffusion isolates the value added by the Stieltjes transfers and compartment structure. Baseline-relative skill is preferable to reporting goodness of fit without a reference forecast~\cite{DietzeEtAl2018}.

For normalized predictions $P$ and observations $Z$, the reported metrics are
\begin{align}
 E_{L^2}(P,Z)
 &=\frac{\|P-Z\|_{L^2(\Omega)}}{\|Z\|_{L^2(\Omega)}},
 \label{eq:galicia-L2}\\
 r(P,Z)
 &=\frac{\int_\Omega(P-\overline P)(Z-\overline Z)\,\dd x}
 {\left(\int_\Omega(P-\overline P)^2\,\dd x\right)^{1/2}
  \left(\int_\Omega(Z-\overline Z)^2\,\dd x\right)^{1/2}},
 \qquad
 \overline P=\frac1{|\Omega|}\int_\Omega P\,\dd x,
 \label{eq:galicia-correlation}
\end{align}
and the centroid distance
\begin{align}
 d_c(P,Z)
 &=\|c(P)-c(Z)\|_{\R^2},
 \qquad
 c(P)=\frac{\int_\Omega xP(x)\,\dd x}{\int_\Omega P(x)\,\dd x}.
 \label{eq:galicia-centroid}
\end{align}
The metrics are complementary: after unit-mass normalization, $E_{L^2}$ penalizes local field-amplitude and shape discrepancies but contains no information on annual abundance; correlation measures co-variation independently of scale; and $d_c$ detects systematic displacement of the mass distribution. No confidence intervals or sampling distributions are attached to these descriptive metrics.

For calibration, the skills relative to persistence are
\begin{equation}\label{eq:galicia-skills}
 S_{L^2}=1-\frac{E_{L^2}^{\rm bio}}{E_{L^2}^{\rm pers}},
 \qquad
 S_{\rm corr}=\frac{r^{\rm bio}-r^{\rm pers}}{1-r^{\rm pers}},
 \qquad
 S_{\rm comp}=\frac12(S_{L^2}+S_{\rm corr}).
\end{equation}
The primary choice maximizes $S_{\rm comp}$ on the \texttt{neutralized} calibration layer. Equal weighting is a pragmatic multi-metric decision rule, not a proper probabilistic score or an inferential test; sensitivity to alternative weights was not examined. Combining the two skills nevertheless prevents selection from being driven exclusively by additional smoothing, which can reduce an $L^2$ error while degrading spatial correlation. The unchanged primary choice is then applied to \texttt{eliminated}. That layer is a robustness check derived from the same administrative system, not an independent external validation dataset.

\subsection{Selected operator and parameter region}\label{subsec:galicia-selection}

The battery contains $2\times5\times5\times6=300$ layer--radius--diffusion--operator combinations. The primary selection is
\begin{equation}\label{eq:galicia-primary-selection}
 \boxed{H_P,\qquad r_{\rm init}=2000\ \mathrm{m},\qquad
 \nu=4\times10^6\ \mathrm{m^2/month}.}
\end{equation}
Thus the average primary-nest activity is selected instead of the pre-transfer foundress field. Figure~\ref{fig:galicia-skill-surface} shows that the calibration score is flat near its maximum. In particular, $(3000,4\times10^6)$ and $(4000,3\times10^6)$ give scores very close to the selected value. Moreover, the selected diffusion is on the upper grid boundary. The pair in \eqref{eq:galicia-primary-selection} must therefore be interpreted as a representative effective smoothing scale, not as a precisely identified biological parameter. This weak-identifiability interpretation is standard in inverse problems when substantially different parameter values produce nearly equivalent model outputs~\cite{Tarantola2005}.

\begin{figure}[htbp]
\centering
\includegraphics[width=0.72\linewidth]{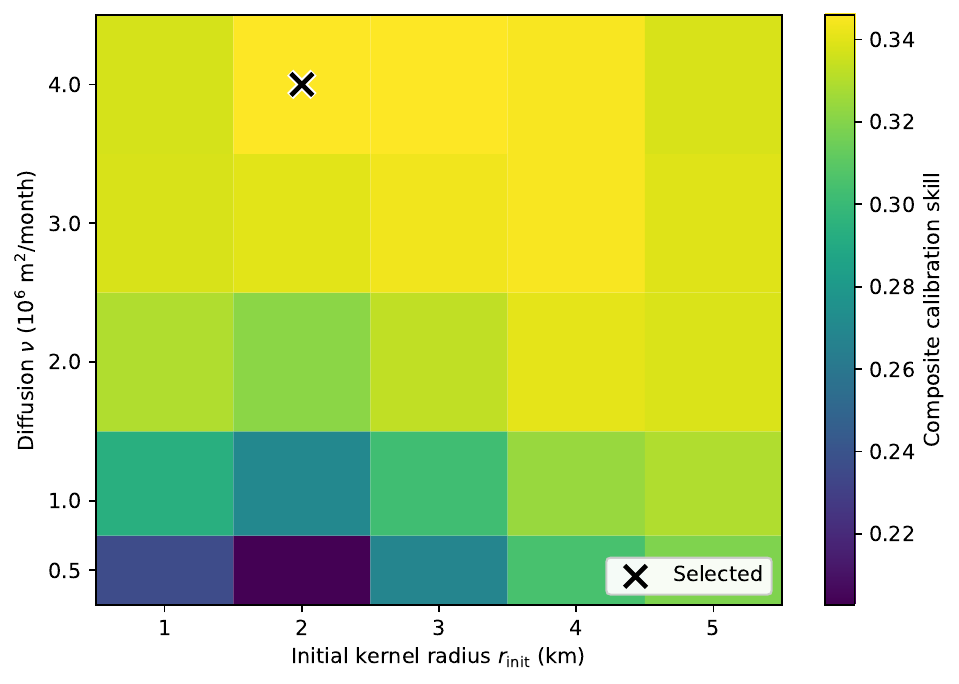}
\caption{Composite calibration skill for the primary-nest observation operator on the \texttt{neutralized} layer. The cross marks the selected pair. The shallow maximum and the upper-bound value of $\nu$ indicate weak parameter identification.}
\label{fig:galicia-skill-surface}
\end{figure}

\subsection{Held-out evaluation and cross-layer robustness}\label{subsec:galicia-results}

Table~\ref{tab:galicia-primary-results} reports the frozen primary selection. On \texttt{neutralized}, the biological model improves the held-out $L^2$ error by
\[
 1-\frac{0.6319}{0.6879}=8.14\%
\]
relative to persistence and by
\[
 1-\frac{0.6319}{0.6527}=3.18\%
\]
relative to matched diffusion. Applying the same choice unchanged to \texttt{eliminated} improves $L^2$ by $10.69\%$ relative to persistence and by $1.01\%$ relative to diffusion. The ordering
\[
 E_{L^2}^{\rm bio}<E_{L^2}^{\rm diff}<E_{L^2}^{\rm pers}
\]
is therefore reproduced across both layers.

\begin{table}[htbp]
\centering
\caption{Calibration and held-out metrics for the primary selection \eqref{eq:galicia-primary-selection}. Smaller $L^2$ and centroid values and larger correlations are preferable. The primary choice is selected only on the \texttt{neutralized} calibration data and transferred unchanged to \texttt{eliminated}.}
\label{tab:galicia-primary-results}
\small
\begin{tabular}{llrrr}
\toprule
Layer and period & Metric & Biological & Diffusion & Persistence\\
\midrule
Neutralized, 2023--2024 & relative $L^2$ & \textbf{0.6279} & 0.6422 & 0.9779\\
 & correlation & \textbf{0.7314} & 0.7152 & 0.5967\\
 & centroid distance (km) & 6.553 & 6.613 & \textbf{6.506}\\
\addlinespace
Neutralized, 2024--2025 & relative $L^2$ & \textbf{0.6319} & 0.6527 & 0.6879\\
 & correlation & 0.7507 & 0.7087 & \textbf{0.8007}\\
 & centroid distance (km) & 9.369 & \textbf{9.254} & 9.458\\
\addlinespace
Eliminated, 2024--2025 & relative $L^2$ & \textbf{0.6784} & 0.6853 & 0.7596\\
 & correlation & 0.6675 & 0.6321 & \textbf{0.7031}\\
 & centroid distance (km) & 8.843 & \textbf{8.787} & 8.883\\
\bottomrule
\end{tabular}
\end{table}

Correlation gives a more cautious conclusion. Persistence retains the largest held-out correlation on both layers. The biological prediction nevertheless improves correlation relative to matched diffusion, by approximately $0.042$ on \texttt{neutralized} and $0.035$ on \texttt{eliminated}. Centroid errors remain near $9$ km for every method, so the model does not resolve the systematic displacement of the annual distribution.

Figure~\ref{fig:galicia-spatial-fields} illustrates the principal effect. The biological model preserves the dominant northwestern and southwestern concentrations but produces a substantially smoother field than the 2025 observations. This broad smoothing explains the improved $L^2$ discrepancy, while the loss of fine-scale peaks helps explain why persistence keeps the larger correlation.

\begin{figure}[htbp]
\centering
\includegraphics[width=0.94\linewidth]{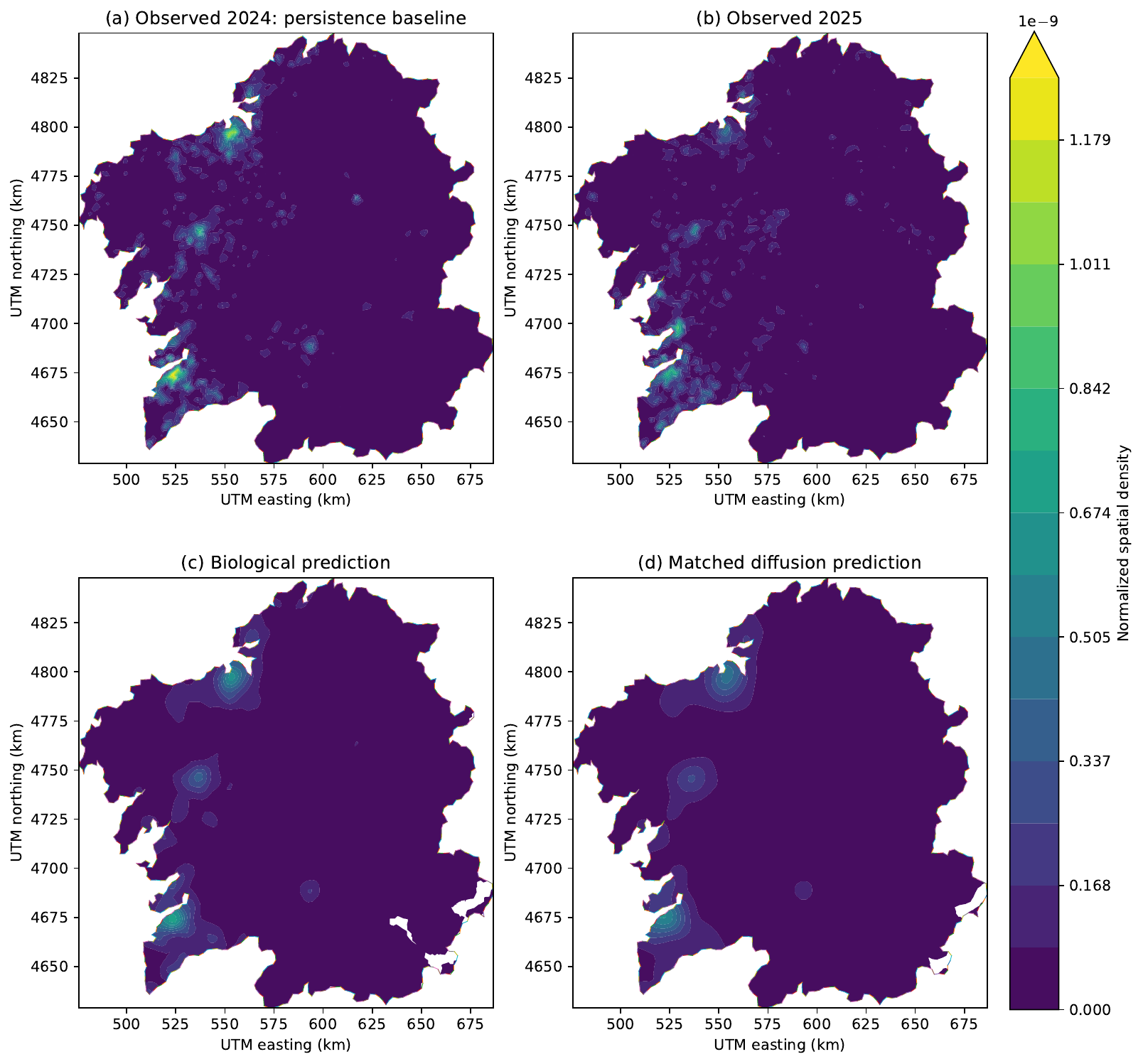}
\caption{Normalized spatial fields for the held-out 2024--2025 evaluation on the \texttt{neutralized} layer. The fields are rendered directly on the 5128-vertex, 9763-triangle FreeFEM++ mesh used by the state solver. Consequently, the coastline is the exact boundary of \texttt{Th.msh}, rather than a convex hull or a polygon reconstructed from the nodal cloud. A common color scale and equal UTM axis scaling are used in all four panels.}
\label{fig:galicia-spatial-fields}
\end{figure}

\subsection{Comparison of biological observation operators}\label{subsec:galicia-operator-comparison}

Table~\ref{tab:galicia-operators} compares the best calibration-selected parameter pair for each operator on \texttt{neutralized}. The primary-nest operator gives the smallest held-out $L^2$ error and the largest correlation among the six biological candidates. The advantage is modest but coherent across the two metrics.

\begin{table}[htbp]
\centering
\caption{Best calibration-selected candidate for each observation operator on the \texttt{neutralized} layer, followed by held-out metrics.}
\label{tab:galicia-operators}
\small
\begin{tabular}{lrrrrr}
\toprule
Operator & $r_{\rm init}$ (km) & $\nu/10^6$ & Validation $L^2$ & Validation corr. & Bio--diff. $L^2$\\
\midrule
$H_F$: pre-transfer foundresses & 4 & 2 & 0.6411 & 0.7378 & 0.0041\\
$H_P$: primary-nest activity & 2 & 4 & \textbf{0.6319} & \textbf{0.7507} & 0.1741\\
$H_S$: secondary-nest activity & 4 & 1 & 0.6386 & 0.7431 & 0.0610\\
$H_{PS}$: combined nest activity & 4 & 1 & 0.6381 & 0.7440 & 0.0437\\
$H_Q$: future-foundress activity & 2 & 1 & 0.6414 & 0.7407 & 0.1057\\
$H_R$: end-of-cycle renewal & 2 & 1 & 0.6414 & 0.7407 & 0.1902\\
\bottomrule
\end{tabular}
\end{table}

For $H_F$, the normalized biological and diffusion fields differ by only $0.0041$ in relative $L^2$, confirming the structural equivalence before the first transfer. For $H_P$, the biological--diffusion discrepancy rises to $0.1741$, and the biological model outperforms matched diffusion in held-out evaluation. This is the clearest evidence in the experiment that the calendar-dependent temporal--compartmental filter can outperform the matched scalar-diffusion baseline at broad scale. Because the coefficients are spatially homogeneous and common across compartments, the result does not isolate or validate individual biological rates.

\begin{figure}[htbp]
\centering
\includegraphics[width=0.82\linewidth]{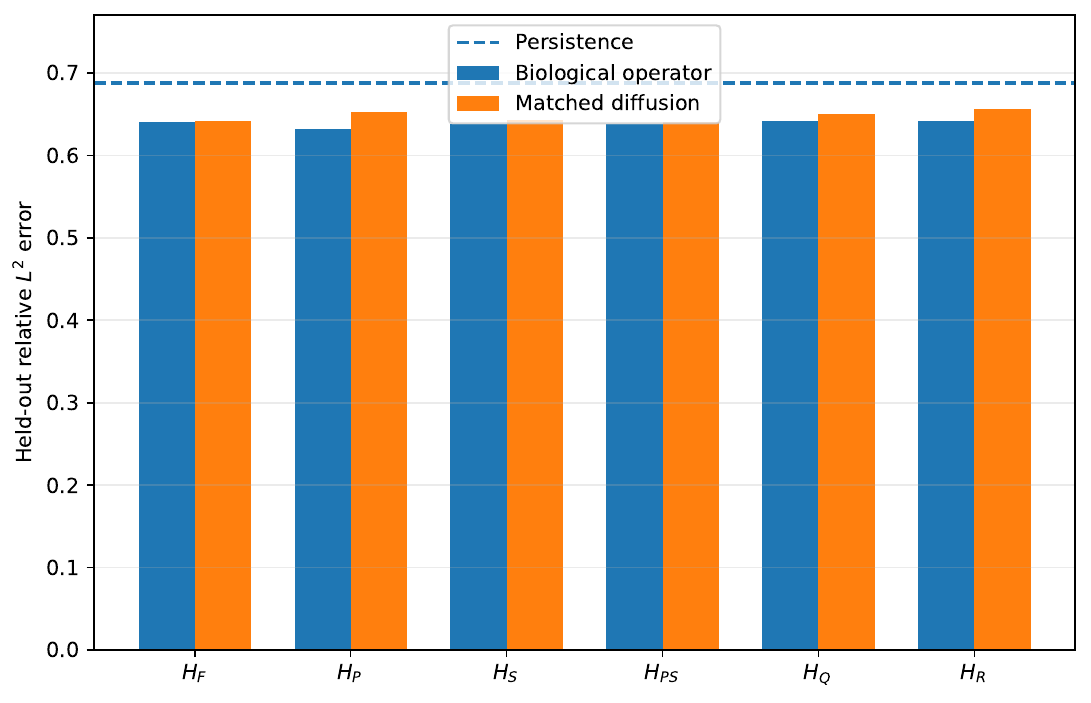}
\caption{Held-out relative $L^2$ error on the \texttt{neutralized} layer for the six observation operators, their matched-diffusion baselines, and persistence. Each operator is evaluated at the parameter pair selected from 2023--2024 calibration only.}
\label{fig:galicia-operator-L2}
\end{figure}

\subsection{Abundance non-identifiability and scope of the evidence}\label{subsec:galicia-abundance}

A scalar factor $\rho$ can be calibrated by matching the total mass in 2024 and then used to predict a 2025 count. The resulting values are $5141.0$ versus $14061$ observed records on \texttt{neutralized}, and $4471.2$ versus $7029$ on \texttt{eliminated}. These discrepancies must not be interpreted as an independent validation of abundance. Because the state equations and the tested observation operators are linear, the initial mass is proportional to the annual count, and $\rho$ is calibrated by total mass, the calculation contains no independent abundance information. Indeed, if the unscaled one-year propagation multiplies total mass by $q$, calibration on 2023--2024 gives $\rho q=N_{2024}/N_{2023}$. Applying the same factor once more yields
\begin{equation}\label{eq:galicia-count-identity}
 \widehat N_{2025}=(\rho q)N_{2024}=\frac{N_{2024}^2}{N_{2023}}.
\end{equation}
It is therefore independent of the spatial parameter pair except for numerical roundoff. A meaningful abundance model requires at least an explicit observation-effort or detection component, and possibly nonlinear demographic mechanisms.

The Galicia experiment consequently supports only the following restricted statement: the selected full-cycle operator has partial predictive skill for the \emph{normalized broad spatial pattern}. It improves quadratic spatial error over both baselines and transfers across the two administrative layers, but persistence retains higher correlation, the centroid displacement is unresolved, fine-scale peaks are oversmoothed, the diffusion scale lies on the parameter-grid boundary, and annual counts are not identified. These limitations are essential because spatial dependence and observation-process changes can otherwise create false confidence in ecological maps~\cite{PlotonEtAl2020}. The experiment contains only one fitted transition and one held-out transition, the two administrative layers are not independent, the kernel representation has unquantified bandwidth and boundary sensitivity, and no uncertainty intervals are available. The reported skill should therefore be interpreted as a case-study diagnostic rather than as a population-level estimate of predictive performance.

\subsection{Galicia control formulation: trapping clusters and effective coverage radii}\label{subsec:galicia-control-clusters}

The spatial support radius used in the academic benchmark must not be interpreted as the attraction range of one physical bait trap. Available management guidance explicitly identifies the feeding area, the distance travelled around the nest, trap specificity, and trapping uncertainty as unresolved research topics~\cite{MITECO2015,MonceauEtAl2014}. The Spanish national strategy also discourages generalized non-selective trapping, recommends traps at selected locations and, in established areas, mainly at apiaries, and expresses deployment effort as several traps per apiary rather than as a large attraction radius for a single device~\cite{MITECO2015}. More generally, continental experience indicates that trapping and nest-destruction measures must be spatially coordinated and should not be expected, in isolation, to stop regional spread~\cite{RobinetEtAl2017,LaurinoEtAl2022}.

For the Galicia-domain control stage, we therefore interpret each optimized centre as a \emph{trapping cluster} or \emph{territorial management unit}. Cluster $k$ is described by
\begin{equation}\label{eq:galicia-cluster-control}
 \mathcal C_k=(\tau_k,z_k,n_k,R_k),
\end{equation}
where $\tau_k$ is its activation time, $z_k\in\Omega_{\rm adm}$ is the cluster centre, $n_k\in\N$ is the number of physical traps deployed within the unit, and $R_k$ is an effective coverage radius. The latter combines the spatial distribution of the devices, repeated movement of insects through the treated area, the duration of exposure, and the spatial resolution of the population model. It is not an empirically established lure-attraction radius.

To preserve the smooth compactly supported structure of the academic kernel, let
\begin{equation}\label{eq:galicia-cluster-kernel}
 K_R(x)=
 \exp\!\left(1-\frac{1}{1-|x|^2/R^2}\right)
 \one_{\{|x|<R\}}.
\end{equation}
Then $K_R(0)=1$ and the mathematical support has diameter $2R$. The kernel is strongly tapered: its half-maximum radius is
\begin{equation}\label{eq:galicia-halfmax-radius}
 r_{1/2}=R\sqrt{\frac{\log 2}{1+\log 2}}\simeq0.64R,
\end{equation}
so the high-intensity core is considerably smaller than the full support. This distinction is important when interpreting kilometre-scale radii.

A saturating aggregate mortality law is preferable to linear superposition of individual traps. The implemented order of operations is important: the contributions of all physical traps are first aggregated into a local exposure, and two scalar saturation maps are then applied to that aggregate. The precise normalized kernel, activation profile and nested exposure are defined below. This construction keeps the mortality bounded and, unlike a clusterwise saturation performed before summation, preserves an exact partition invariance for coincident allocations.

\subsubsection*{Normalized moving kernels and exact derivatives}

The implemented Galicia kernel is normalized on the computational domain. The zero extension of $K_R$ belongs to $C_c^\infty(\mathbb R^2)$; hence translations are continuously differentiable in $L^1(\Omega)$, and differentiation under the integral defining $Z_R$ is justified by dominated convergence. For a centre $z$ and radius $R$, set
\begin{equation}\label{eq:galicia-kernel-normalization}
 Z_R(z)=\int_\Omega K_R(x-z)\,\dd x,
 \qquad
 \kappa_{R,z}(x)=I_{\rm ref}\frac{K_R(x-z)}{Z_R(z)},
\end{equation}
where $I_{\rm ref}>0$ is fixed. Thus $\int_\Omega\kappa_{R,z}=I_{\rm ref}$ even when the compact support is truncated by the territorial boundary. The admissible centre set used below is the finite union of two closed and bounded $5$ km boxes. These boxes are chosen so that $Z_R(z)>0$ for every admissible centre; since $Z_R$ is continuous, compactness yields the uniform bound
\[
 Z_R(z)\ge Z_{R,\min}:=\min_{z\in\Omega_{\rm adm}}Z_R(z)>0.
\]
Consequently, the normalized kernel and its first centre derivatives are well defined and uniformly bounded on the entire admissible set. If $\xi$ is either centre coordinate, the quotient rule gives
\begin{equation}\label{eq:galicia-normalized-kernel-derivative}
 \partial_\xi\kappa_{R,z}(x)
 =I_{\rm ref}\frac{\partial_\xi K_R(x-z)Z_R(z)
 -K_R(x-z)\partial_\xi Z_R(z)}{Z_R(z)^2},
 \qquad
 \partial_\xi Z_R(z)=\int_\Omega\partial_\xi K_R(x-z)\,\dd x.
\end{equation}
In particular,
\begin{equation}\label{eq:galicia-normalized-derivative-zero-mass}
 \int_\Omega\partial_\xi\kappa_{R,z}(x)\,\dd x=0,
\end{equation}
which is an exact audit of the normalization derivative. The coverage diagnostic is
\begin{equation}\label{eq:galicia-coverage-definition}
 c_R(z)=\frac{Z_R(z)}{Z_R^{\mathbb R^2}},
 \qquad
 \partial_\xi c_R(z)=\frac{\partial_\xi Z_R(z)}{Z_R^{\mathbb R^2}},
\end{equation}
where $Z_R^{\mathbb R^2}$ is the full-space kernel integral. It affects only the admissibility safeguard, not the normalized mortality field.

For a campaign of fixed duration $d_k$ and smoothing width $\varepsilon>0$, the code uses
\begin{equation}\label{eq:galicia-smooth-window}
 \chi_k(t;\tau_k)
 =\frac14\left(1+\tanh\frac{t-\tau_k}{\varepsilon}\right)
 \left(1+\tanh\frac{\tau_k+d_k-t}{\varepsilon}\right).
\end{equation}
Set
\[
 A_k(t)=\frac{t-\tau_k}{\varepsilon},
 \qquad
 B_k(t)=\frac{\tau_k+d_k-t}{\varepsilon}.
\]
Differentiating both factors in~\eqref{eq:galicia-smooth-window} gives the explicit formula
\begin{equation}\label{eq:galicia-smooth-window-derivative}
 \partial_{\tau_k}\chi_k(t;\tau_k)
 =-\frac{1}{4\varepsilon}\operatorname{sech}^2\!A_k(t)\bigl(1+\tanh B_k(t)\bigr)
 +\frac{1}{4\varepsilon}\bigl(1+\tanh A_k(t)\bigr)\operatorname{sech}^2\!B_k(t).
\end{equation}
The first term differentiates the switch-on factor and the second differentiates the switch-off factor; their signs are opposite, as expected for a translation of a fixed-duration campaign. Writing
\begin{equation}\label{eq:galicia-nested-exposure}
 L_u(t,x)=\sum_{k=1}^{N_C}n_k\chi_k(t;\tau_k)\kappa_{R_k,z_k}(x),
 \quad
 E_u=\eta_{\max}(1-e^{-\gamma L_u}),
 \quad
 a_u=a_{\max}(1-e^{-E_u}),
\end{equation}
the complete chain rule is
\begin{equation}\label{eq:galicia-mortality-chain-rule}
 \partial_\theta a_u
 =a_{\max}e^{-E_u}\eta_{\max}\gamma e^{-\gamma L_u}
 \partial_\theta L_u,
\end{equation}
with
\begin{align}
 \partial_{\tau_k}L_u
 &=n_k\,\partial_{\tau_k}\chi_k\,\kappa_{R_k,z_k},\label{eq:galicia-L-tau-derivative}\\
 \partial_{x_k}L_u
 &=n_k\chi_k\,\partial_{x_k}\kappa_{R_k,z_k},
 &
 \partial_{y_k}L_u
 &=n_k\chi_k\,\partial_{y_k}\kappa_{R_k,z_k}.
 \label{eq:galicia-L-space-derivatives}
\end{align}
Equations~\eqref{eq:galicia-normalized-kernel-derivative}--\eqref{eq:galicia-L-space-derivatives} are the exact formulas implemented and validated by finite differences. In the assembled code, the raw kernels and their centre derivatives are represented in the same $P_1$ space and every normalization integral is evaluated with the same quadrature rule. Consequently, the discrete analogue of $\int_\Omega\partial_\xi\kappa_{R,z}\dd x=0$ is an algebraic identity up to roundoff rather than a comparison between incompatible quadratures.

Table~\ref{tab:galicia-control-production-constants} records the numerical constants used in every production run from the L2 continuation through the final temporal-refinement study. The normalization target is the full-space integral of the reference kernel with radius $R_{\rm ref}=2$ km,
\[
 I_{\rm ref}=1.2681121611275896\,(2000\ {\rm m})^2
 =5.072448644510359\times10^6\ {\rm m}^2.
\]

\begin{table}[htbp]
\centering
\small
\caption{Production constants for the Galicia trapping-cluster control. These values define the numerical scenario; the intensity constants have not been calibrated from field intervention data.}
\label{tab:galicia-control-production-constants}
\begin{tabular}{lll}
\toprule
Quantity & Value & Role \\
\midrule
$N_C$ & $2$ & number of optimized trapping clusters \\
$n_1=n_2$ & $10$ traps & physical devices assigned to each cluster \\
$R_1=R_2$ & $4$ km & effective compact-support radius \\
$d_1=d_2$ & $1$ month & fixed campaign duration \\
$\varepsilon$ & $0.08$ months & smoothing width of the activation window \\
$\eta_{\max}$ & $1$ & upper level of the exposure saturation \\
$\gamma$ & $0.25$ & inverse exposure scale in $E_u$ \\
$a_{\max}$ & $0.6$ month$^{-1}$ & maximum additional mortality coefficient \\
$I_{\rm ref}$ & $5.072448644510359\times10^6$ m$^2$ & prescribed integral of each normalized kernel \\
$R_{\rm ref}$ & $2$ km & reference radius defining $I_{\rm ref}$ \\
$t_{\rm scale}$ & $1$ month & temporal optimization scale \\
$z_{\rm scale}$ & $10^5$ m & scale for optimized centre offsets \\
\bottomrule
\end{tabular}
\end{table}

The reported value of $1-J_{\rm bio}$ is therefore conditional on the fixed constants in Table~\ref{tab:galicia-control-production-constants}. In particular, it is a reproducible numerical diagnostic for this scenario, not an estimate of real trap capture efficacy.

\begin{proposition}[Partition invariance of coincident allocations]\label{prop:galicia-partition-invariance}
Suppose a collection of clusters has the same centre, radius, activation time and duration. Replacing their trap counts $n_1,\ldots,n_r$ by one count $n=\sum_{j=1}^rn_j$, or conversely splitting $n$ into any such partition, leaves $L_u$, $E_u$ and $a_u$ unchanged.
\end{proposition}

\begin{proof}
All coincident clusters share the same factor $\chi(t;\tau)\kappa_{R,z}(x)$. Their contribution to~\eqref{eq:galicia-nested-exposure} is
\[
 \sum_{j=1}^rn_j\chi(t;\tau)\kappa_{R,z}(x)
 =\left(\sum_{j=1}^rn_j\right)\chi(t;\tau)\kappa_{R,z}(x).
\]
Hence $L_u$ is unchanged. Since $E_u$ and $a_u$ are pointwise scalar functions of $L_u$, they are unchanged as well.
\end{proof}

Proposition~\ref{prop:galicia-partition-invariance} removes an artificial numerical gain from subdividing an identical co-located budget. It does not imply uniqueness or stability of optimized placements at distinct centres.

Because no field dataset presently identifies $R_k$ as a biological parameter, the initial design considered radius scenarios
\begin{equation}\label{eq:galicia-radius-scenarios}
 R_k\in\mathcal R=\{1,2,4\}\ {\rm km}.
\end{equation}
These values correspond to support diameters of $2$, $4$, and $8$ km, while their half-maximum core diameters are approximately $1.28$, $2.56$, and $5.12$ km. They must not be described as recommended attraction diameters for single traps. The numerical screening in Section~\ref{sec:galicia-control-screening} evaluates $R=2$ and $4$ km: $R=2$ km remains mesh-sensitive on L1, whereas $R=4$ km is stable from L1 to L2 and is therefore adopted for the first optimization stage. The $1$ km scenario is deferred until a finer spatial discretization is used. The 10 km and 30 km radii appearing in the Spanish strategy delimit areas in which apiaries may be prioritized relative to previous detections; they are surveillance and deployment-selection distances, not physical capture ranges~\cite{MITECO2015}.

A resource-aware robust formulation for a later allocation stage can take the form
\begin{equation}\label{eq:galicia-cluster-objective}
 \min_{u\in\mathcal U_{\rm cl}}
 \left\{
 \sum_{s=1}^{N_S}p_s J_{\rm bio}^{(s)}(Y_u^{(s)})
 +c_C N_C+c_T\sum_{k=1}^{N_C}n_k
 \right\},
\end{equation}
subject to a total budget, accessible centres $z_k\in\Omega_{\rm adm}$, and, when appropriate, separation constraints between cluster centres. The scenarios $s$ should include the weakly identified dispersal and initial-field scales found in Section~\ref{sec:galicia-validation}. The local discrete-adjoint study reported below fixes $N_C=2$, $n_1=n_2=10$, and $R_1=R_2=4$ km and optimizes activation times and centres. Integer cluster counts and device allocations in~\eqref{eq:galicia-cluster-objective} are reserved for a later outer enumeration or mixed-integer stage. Neither the present local study nor that possible extension is a management recommendation.

\FloatBarrier
\section{Numerical screening of Galicia trapping-cluster controls}\label{sec:galicia-control-screening}

The held-out experiment in Section~\ref{sec:galicia-validation} concerns the uncontrolled state and the observation operator. Before transferring the adjoint optimizer to the Galicia domain, we performed a separate numerical screening of the cluster formulation in Section~\ref{subsec:galicia-control-clusters}. The purpose is to determine whether the discretization is sufficiently stable and which fixed scenarios are informative as initializations and constraints. The biological transition and compartment equations are left unchanged. In particular, the screening does not add an uncalibrated dependence of future-foundress production on late worker abundance.

For each controlled case, let $M_P(u)$ denote the spatial mass of the primary-phase observation operator and let $M_F(u)$ denote the end-of-cycle foundress mass. We use the normalized diagnostic
\begin{equation}\label{eq:galicia-control-Jbio}
 J_{\rm bio}(u)
 =\frac12\frac{M_P(u)}{M_P(0)}
 +\frac12\frac{M_F(u)}{M_F(0)},
\end{equation}
where the zero-control baseline shares the same mesh, time step, diffusion, initial-field treatment, and mass formulation. Thus $J_{\rm bio}=1$ for the matched uncontrolled case, and $1-J_{\rm bio}$ measures the relative improvement of this two-component numerical diagnostic. It is not yet a management cost: neither term has been calibrated against intervention outcomes, and physical deployment costs are not included.

The control-screening study contains a 59-case core battery and 28 directed post-core tests. The core battery compares consistent and lumped reaction--mass discretizations, raw and clipped--renormalized initial fields, mesh levels L1 and L2, 15, 30, and 60 subdivisions per month, effective cluster radii of 2 and 4 km, diffusion coefficients $\nu\in\{3,4,5\}\times10^6$ m$^2$/month, four allocations of 20 physical traps, phase windows, and compartment masks. Directed tests audit the masks, compare one-month activation windows, confirm the fixed-budget ranking on L2 with 60 subdivisions per month, and extend the two best allocations to 120 subdivisions per month. All 87 cases completed. The zero-control consistency and the co-located partition identities $8=4+4=2+2+2+2$ passed to machine precision.

\subsection{Positivity and discretization choices}\label{subsec:galicia-control-positivity}

The consistent-mass formulation develops small but systematic negative undershoots. On L1 with 30 subdivisions per month, the minimum state is $-1.264\times10^{-6}$ and the relative negative mass is $3.504\times10^{-4}$. On L2 the corresponding values decrease to $-5.197\times10^{-8}$ and $5.046\times10^{-7}$. Refining only in time does not remove the L1 undershoot: with 60 subdivisions per month the relative negative mass is $5.063\times10^{-4}$. By contrast, the lumped formulation gives zero negative mass in the L1 and L2 reference cases; the only nonzero high-resolution diagnostic is $1.17\times10^{-10}$, at the level of accumulated floating-point noise. The raw interpolated initial field is already nonnegative, so clipping is not responsible for the improvement. We therefore use mass lumping in the Galicia control computations and retain clipping and mass renormalization only as a defensive preprocessing option.

Table~\ref{tab:galicia-radius-mesh} compares the two effective cluster radii under the lumped formulation. For $R=2$ km, changing from L1 to L2 alters the predicted benefit by approximately $5.32\%$; the support contains only 8 L1 vertices. For $R=4$ km, the L1 and L2 values agree to $2.66\times10^{-7}$ in $J_{\rm bio}$ and the benefit change is below $0.01\%$. The $4$ km radius is therefore accepted for L1 exploration, whereas a $2$ km cluster requires at least L2 confirmation. As emphasized in Section~\ref{subsec:galicia-control-clusters}, $R=4$ km is an effective territorial support radius, not the attraction radius of one physical trap.

\begin{table}[htbp]
\centering
\caption{Mesh dependence of the normalized biological diagnostic for one reference cluster under the lumped formulation and 30 subdivisions per month.}
\label{tab:galicia-radius-mesh}
\begin{tabular}{cccc}
\toprule
$R$ (km) & $J_{\rm bio}$, L1 & $J_{\rm bio}$, L2 & Relative change of benefit \\
\midrule
2 & 0.995660612 & 0.995879626 & $5.32\%$ \\
4 & 0.993034389 & 0.993034123 & $<0.01\%$ \\
\bottomrule
\end{tabular}
\end{table}

For the accepted $R=4$ km configuration on L1, the temporal sequence is
\begin{equation}\label{eq:galicia-time-sequence-R4}
 J_{15}=0.993145559,
 \qquad
 J_{30}=0.993034389,
 \qquad
 J_{60}=0.992978322.
\end{equation}
The ratio $|J_{60}-J_{30}|/|J_{30}-J_{15}|=0.504$ is consistent with an observed first-order temporal refinement pattern. The two leading fixed-budget configurations were subsequently evaluated on L2 up to 120 subdivisions per month. Their difference ratios are $0.4986$ and $0.5001$, corresponding to estimated orders $1.004$ and $0.9997$. These tests justify using L1 with 30 subdivisions per month during optimization, L2 with 60 subdivisions per month for systematic verification, and L2 with 120 subdivisions per month only for selected final controls.

\subsection{Fixed physical budget and spatial allocation}\label{subsec:galicia-fixed-budget}

Four prescribed allocations satisfying $\sum_k n_k=20$ were compared: one cluster of 20 traps, two hotspot clusters of 10 traps, five hotspot clusters of 4 traps, and five approximately uniform clusters of 4 traps. Table~\ref{tab:galicia-fixed-budget} gives the L2 results with 60 subdivisions per month. The ordering is unchanged across L1--L2 refinement and across the three diffusion scenarios $\nu\in\{3,4,5\}\times10^6$ m$^2$/month.

\begin{table}[htbp]
\centering
\caption{Fixed-budget comparison on L2 with 60 subdivisions per month. The percentage is $100(1-J_{\rm bio})$. High-resolution values are reported for the two leading configurations.}
\label{tab:galicia-fixed-budget}
\begin{tabular}{lccc}
\toprule
Allocation & $J_{\rm bio}$, 60 & Reduction & $J_{\rm bio}$, 120 \\
\midrule
Two hotspot clusters, $2\times10$ & 0.985018263 & $1.49817\%$ & 0.984960039 \\
Five hotspot clusters, $5\times4$ & 0.987934371 & $1.20656\%$ & 0.987888584 \\
One hotspot cluster, $1\times20$ & 0.988746757 & $1.12532\%$ & --- \\
Five uniform clusters, $5\times4$ & 0.998255685 & $0.17443\%$ & --- \\
\bottomrule
\end{tabular}
\end{table}

The two-hotspot allocation is therefore the best of the four tested configurations. At 120 subdivisions per month it produces a $1.50400\%$ reduction of $J_{\rm bio}$, compared with $1.21114\%$ for five hotspot clusters. This ranking is not a proof of global optimality, because only four prescribed layouts have been compared. Moreover, equal numbers of physical traps do not imply equal integrated control after spatial smoothing, overlap, and saturation. The subsequent optimization therefore retains the physical-device budget while also reporting cluster count, integrated exposure, and deployment cost.

\subsection{Equal-duration activation windows}\label{subsec:galicia-equal-windows}

The core phase windows have different durations. To isolate timing, directed tests use one-month windows with the same locations, device count, amplitude, and nominal duration. Table~\ref{tab:galicia-equal-windows} shows a monotone loss of effectiveness as intervention is delayed. The early foundress window gives the largest reduction, followed by the primary-nest phase and the late foundress window. Late secondary-nest and future-foundress windows have little effect on the current diagnostic.

\begin{table}[htbp]
\centering
\caption{One-month timing comparison under a common spatial deployment.}
\label{tab:galicia-equal-windows}
\begin{tabular}{lcc@{\qquad}lcc}
\toprule
Window & $J_{\rm bio}$ & Reduction & Window & $J_{\rm bio}$ & Reduction \\
\midrule
Early F & 0.997476848 & $0.25232\%$ & S, month 2 & 0.999715470 & $0.02845\%$ \\
P & 0.998181031 & $0.18190\%$ & Late Q & 0.999865702 & $0.01343\%$ \\
Late F & 0.998423454 & $0.15765\%$ & S, month 3 & 0.999902447 & $0.00976\%$ \\
S, month 1 & 0.999259181 & $0.07408\%$ & S, month 4 & 0.999953175 & $0.00468\%$ \\
Early Q & 0.999411105 & $0.05889\%$ & & & \\
\bottomrule
\end{tabular}
\end{table}

The early-F reduction is approximately $5.6$ times the early-Q reduction, despite a slightly smaller integrated temporal exposure caused by truncation at the initial time. This supports prioritizing the beginning of F and the P phase in the adjoint optimization. It does not establish a universal field-season recommendation, because the calendar coefficients and the control-to-mortality conversion are not yet calibrated from trapping outcomes.

\subsection{Compartment-mask audit and preservation of the original model}\label{subsec:galicia-mask-audit}

The core battery initially produced identical values for the masks \texttt{all\_active}, \texttt{flying\_phase}, and \texttt{legacy\_all}. Directed diagnostics record the weighted discrete removal
\begin{equation}\label{eq:galicia-mask-removal}
 D_{p,i}=\sum_{n\in p}\Delta g_i^n
 \int_\Omega m_{p,i}a^n(x)y_i^n(x)\,\dd x
\end{equation}
for every phase $p$ and compartment $i$. The audit shows that the equality is structural: the masks differ only on compartments that are absent, inactive, or downstream-irrelevant during the corresponding phases. The common full masks give $J_{\rm bio}=0.993034389$, foundress-only control gives $0.994227900$, future-foundress-only control gives $0.999280059$, and worker-only control gives $0.999340642$. Worker-only control during S or Q records nonzero local removal but leaves $J_{\rm bio}=1$ under the present transfer equations and objective; worker control during P gives the full worker-only effect.

We retain the original biological model. The audit is therefore interpreted as identifying ineffective channels of the current model, rather than as evidence for inserting an uncalibrated coupling from late workers to future-foundress production. In the optimization reported below, variables associated solely with worker removal in S or Q are therefore omitted when the objective is~\eqref{eq:galicia-control-Jbio}. If a later biological study establishes such a coupling, it should be introduced and calibrated as a separate model extension, followed by a new adjoint and validation cycle.

\subsection{Numerical configuration and remaining scope}\label{subsec:galicia-control-frozen}

The screening supports the following production configuration:
\begin{equation}\label{eq:galicia-production-configuration}
 \boxed{\begin{aligned}
 &\text{lumped mass},\qquad R=4\ {\rm km},\\
 &\text{L1 with 30 subdivisions/month for optimization},\\
 &\text{L2 with 60 subdivisions/month for verification}.
\end{aligned}}
\end{equation}
The computations reported below use L2 with 120 subdivisions per month for the final time continuation and L2 with 240 subdivisions per month for the terminal fixed-control refinement. A two-hotspot allocation of ten traps per cluster provides the reproducible starting control, and the temporal search emphasizes early F and P. These choices reduce numerical ambiguity, but they do not turn $J_{\rm bio}$ into a calibrated management objective. The discrete adjoint for the lumped, saturating cluster discretization is validated in the next subsection. With $N_C=2$ and $n_1=n_2=10$ fixed, IPOPT then optimizes one-month activation times and local centre offsets under coverage, separation, and box safeguards. A later outer allocation study may compare other integer partitions and deployment costs, and a management-oriented study must additionally propagate diffusion, observation, accessibility, and trap-effectiveness uncertainty.

\subsection{Discrete-adjoint validation for the specified Galicia residual}\label{subsec:galicia-adjoint-validation}

We next differentiate the exact L1 residual associated with~\eqref{eq:galicia-production-configuration}, rather than a continuous surrogate. The forward trajectory stores the five biological phases and the historical averages used by the transfers. Reverse accumulation transposes the lumped implicit steps, the F--P--S--Q--H event maps, and the discrete averaging operators. The mortality field retains the nested saturation law of the screening study. For a cluster centre $z_k=(x_k,y_k)$, the compact kernel is normalized after restriction to the Galicia mesh; consequently, the derivatives with respect to $x_k$ and $y_k$ include the derivative of the normalization integral. The physical gradient is ordered as
\[
 \nabla J_{\rm bio}
 =\bigl(\partial_{\tau_1}J,\partial_{\tau_2}J,
          \partial_{x_1}J,\partial_{x_2}J,
          \partial_{y_1}J,\partial_{y_2}J\bigr).
\]
Times are measured in months. For optimization and directional tests, spatial increments are nondimensionalized by $10^5$ m.

Three independent checks were used before connecting the gradient to IPOPT.\@ First, direct differentiation of the smoothed activation windows and moving normalized kernels was compared with centred finite differences at the L1 mesh nodes. The maximum weighted relative error over both the nine-month and one-month scenarios was $1.47\times10^{-8}$; the integral of each normalized-kernel derivative was below $9.8\times10^{-13}$ in the FreeFEM++ runs, and the co-located partition-invariance error was zero. Second, an algebraic calendar audit compared the transposed biological transfers with an independently propagated directional perturbation, giving a relative difference of $2.78\times10^{-17}$ and a best finite-difference error below $4\times10^{-12}$. Third, the complete PDE adjoint was compared with a separately implemented linearized state equation and with centred differences of the reduced functional.

\begin{table}[htbp]
\centering
\caption{Validation of the complete discrete adjoint on the Galicia L1 configuration. D9 reproduces the nine-month screening $2\times10$ case; D1 uses equal one-month campaigns and is the initialization adopted for optimization.}
\label{tab:galicia-adjoint-validation}
\scriptsize
\begin{tabular}{@{}lccccc@{}}
\toprule
Scenario & Duration & $(\tau_1,\tau_2)$ & $J_{\rm bio}$ & \shortstack{adjoint--\\linearized} & \shortstack{best\\adjoint--FD} \\
\midrule
D9 nine-month check & 9 months & $(1,1)$ & 0.985151056 & $1.89\times10^{-16}$ & $1.89\times10^{-6}$ \\
D1 equal effort & 1 month & $(1,3)$ & 0.995303364 & $1.07\times10^{-16}$ & $7.06\times10^{-7}$ \\
\bottomrule
\end{tabular}
\end{table}

The D9 trajectory also reproduces the screening zero-control and controlled masses with relative discrepancies below $2.6\times10^{-7}$. Its reduction, $1-J_{\rm bio}=1.48489\%$, is consistent with the L1 screening value; the small difference reflects the exact trajectory reconstruction and normalization used by the adjoint driver. For D1, the primary-phase and final-foundress masses are reduced by $0.35281\%$ and $0.58652\%$, respectively, giving a combined reduction of $0.469664\%$.

At the D1 starting control, the physical gradient is
\[
\begin{aligned}
\nabla J_{\rm bio}={}&(-8.12893\times10^{-4},\ 6.15536\times10^{-4},
 -1.38809\times10^{-7},\\
&\hspace{2.2cm}-6.49429\times10^{-8},\ 4.07342\times10^{-8},\ 4.07837\times10^{-8}).
\end{aligned}
\]
Because these components have different physical units, their magnitudes must not be compared directly. After scaling activation times by one month and spatial offsets by $10^5$ m, the starting control is clearly nonstationary. The agreement among the adjoint, the independent linearized equation, and finite differences closes the numerical prerequisite for the local two-cluster optimization studied next.

\FloatBarrier
\section{Galicia two-cluster optimization and final refinement}\label{sec:galicia-ipopt-optimization}

\subsection{Optimization variables, safeguards, and hierarchy of discretizations}

The optimization fixes two clusters, ten traps per cluster, radius $R=4$ km, and one-month campaign durations. The physical control is
\[
 u=(\tau_1,\tau_2,x_1,x_2,y_1,y_2),
\]
and the optimizer uses the dimensionless coordinates
\begin{equation}\label{eq:galicia-scaled-control}
 \zeta=\left(\frac{\tau_1}{1\ {\rm month}},\frac{\tau_2}{1\ {\rm month}},
 \frac{x_1-x_1^{\rm ref}}{10^5\ {\rm m}},
 \frac{x_2-x_2^{\rm ref}}{10^5\ {\rm m}},
 \frac{y_1-y_1^{\rm ref}}{10^5\ {\rm m}},
 \frac{y_2-y_2^{\rm ref}}{10^5\ {\rm m}}\right).
\end{equation}
The spatial boxes allow local offsets of at most $5$ km from the screened hotspots. Coverage and separation are implemented as differentiable exterior penalties, not as additional hard constraints. More precisely, with $(r)_+=\max\{r,0\}$,
\begin{align}
 P_{\rm cov}(u)&=\frac{\lambda_c}{2}\sum_{k=1}^{2}
 \bigl(c_{\min}-c_R(z_k)\bigr)_+^2,\label{eq:galicia-coverage-penalty}\\
 P_{\rm sep}(u)&=\frac{\lambda_s}{2}
 \bigl(d_{\min}-d_\epsilon(z_1,z_2)\bigr)_+^2,\label{eq:galicia-separation-penalty}
\end{align}
where $c_{\min}=0.85$, $d_{\min}=30$ km, $\lambda_c=1$, $\lambda_s=10^{-8}\ {\rm m}^{-2}$, and $d_\epsilon(z_1,z_2)=\sqrt{|z_1-z_2|^2+\epsilon_d^2}$ with $\epsilon_d^2=10^{-12}\ {\rm m}^2$. Since $\frac{\dd}{\dd r}\frac12(r_+)^2=r_+$, their nonzero centre derivatives are
\begin{align}
 \nabla_{z_k}P_{\rm cov}(u)
 &=-\lambda_c\bigl(c_{\min}-c_R(z_k)\bigr)_+\nabla c_R(z_k),
 \label{eq:galicia-coverage-penalty-gradient}\\
 \nabla_{z_1}P_{\rm sep}(u)
 &=-\lambda_s\bigl(d_{\min}-d_\epsilon(z_1,z_2)\bigr)_+
 \frac{z_1-z_2}{d_\epsilon(z_1,z_2)},
 \label{eq:galicia-separation-penalty-gradient-z1}\\
 \nabla_{z_2}P_{\rm sep}(u)
 &=-\nabla_{z_1}P_{\rm sep}(u).
 \label{eq:galicia-separation-penalty-gradient-z2}
\end{align}
The time derivatives of both penalties are zero. The squared positive-part function is $C^1$ but not $C^2$ at the activation threshold. This regularity is sufficient for the first-order identities and projected-stationarity diagnostics used here; no second-order convergence theorem for IPOPT is invoked. The negligible distance regularization only avoids division by zero in the derivative. The optimizer minimizes
\begin{equation}\label{eq:galicia-penalized-objective}
 \widehat J(u)=J_{\rm bio}(u)+P_{\rm cov}(u)+P_{\rm sep}(u)
\end{equation}
subject only to the stated box bounds. Therefore the projection residual of Section~\ref{subsec:projected-gradient-residual} is the correct stationarity diagnostic for the actual optimized problem. At the reported accepted controls the inequalities are strict, both penalties and their gradients vanish, and consequently $\nabla\widehat J=\nabla J_{\rm bio}$. IPOPT is applied first on L1 with $30$ subdivisions per month, then continued on L2 with $60$ subdivisions per month. Selected final evaluations use L2 with $120$ and $240$ subdivisions per month.

\subsection{L1 time-only and timing--placement optimization}

Starting from the D1 equal-effort control, IPOPT terminates after eight iterations for the time-only problem at
\[
 (\tau_1,\tau_2)=(1.032911558,\ 2.748750898)\ {\rm months},
\]
with $J_{\rm bio}=0.995248260191$ and an active-variable projected residual $2.76\times10^{-9}$. Allowing the two centres to move produces two independent timing--placement trajectories. One starts from the reference continuation and the other from a displaced initialization followed by a saved-point continuation. Their best controls have an objective gap $2.84\times10^{-9}$; the first centres differ by $0.13$ m and the second by $0.46$ m. Thus the two trajectories terminate at numerically indistinguishable points in the same apparent L1 basin; this does not prove uniqueness of a local minimizer.

\begin{table}[htbp]
\centering
\small
\caption{L1 Galicia optimization results. Residuals are dimensionless projected residuals in the active optimizer coordinates.}
\label{tab:galicia-l1-ipopt}
\begin{tabular}{lccc}
\toprule
Control & $J_{\rm bio}$ & Reduction $100(1-J_{\rm bio})$ & Projected residual \\
\midrule
D1 reference & 0.995303363809 & $0.469664\%$ & nonstationary \\
Time-only solution & 0.995248260191 & $0.475174\%$ & $2.76\times10^{-9}$ \\
Timing--placement A & 0.995130729047 & $0.486927\%$ & $3.47\times10^{-6}$ \\
Timing--placement B & 0.995130726209 & $0.486927\%$ & $5.65\times10^{-6}$ \\
\bottomrule
\end{tabular}
\end{table}

\subsection{L2 directional check and projected-residual continuation}

When recomputed on L2 with $60$ subdivisions per month, the two L1 controls remain nearly indistinguishable:
\[
 J_A^{60}=0.995126431876,\qquad
 J_B^{60}=0.995126476735,
\]
with objective gap $4.49\times10^{-8}$. Both improve on the L2 time-only value $0.995210913407$, but their L2 projected residuals are approximately $2.17\times10^{-3}$. This residual is not a failure of the adjoint: an independent directional test gives
\[
 \nabla J_{\rm bio}(u_A)\cdot h=-3.375195459631\times10^{-3},
\]
and centred finite differences have relative errors $1.01\times10^{-2}$, $3.41\times10^{-3}$, and $9.27\times10^{-4}$ for perturbations $2\times10^{-3}$, $10^{-3}$, and $5\times10^{-4}$, respectively.

An L2 IPOPT continuation from $u_A$ terminates at the acceptable level after 19 iterations, 85 objective evaluations, and 21 gradient evaluations. The best observed point has
\begin{equation}\label{eq:galicia-l2-best-control}
\begin{aligned}
 \tau_1&=1.020594764, & x_1&=524176.365, & y_1&=4673868.807,\\
 \tau_2&=2.732526256, & x_2&=553750.553, & y_2&=4797018.875,
\end{aligned}
\end{equation}
where times are in months and spatial coordinates are in projected metres. It gives
\[
 J_{\rm bio}^{60}=0.995120149081,
 \qquad R^{\rm dim}=4.84\times10^{-6}.
\]
The two support coverages are $0.997205$ and $0.990873$, the centre separation is $126.65$ km, and all safeguards remain inactive.

\subsection{Final stationarity check at 120 subdivisions per month}

Evaluation of~\eqref{eq:galicia-l2-best-control} with $120$ subdivisions per month gives $J_{\rm bio}=0.995093218147$ and reveals a small residual concentrated in the first activation time. The preceding six-variable L2 continuation at 60 subdivisions required approximately $1.4689\times10^5$ s ($40.8$ h) of recorded computation. We therefore perform a short two-variable continuation in $(\tau_1,\tau_2)$ while keeping the L2 centres fixed, and then recompute all six gradient components at the resulting point. IPOPT terminates after three iterations at
\begin{equation}\label{eq:galicia-final-control}
\boxed{\begin{aligned}
 \tau_1^*&=1.016099365, & x_1^*&=524176.365, & y_1^*&=4673868.807,\\
 \tau_2^*&=2.719941224, & x_2^*&=553750.553, & y_2^*&=4797018.875.
\end{aligned}}
\end{equation}
The resulting objective and stationarity diagnostics are
\begin{equation}\label{eq:galicia-final-ndiv120}
 J_{\rm bio}^{120}(u^*)=0.995092856242,
 \qquad R_{\rm time}^{\rm dim}=5.91\times10^{-6},
 \qquad \|\nabla_{\rm space}J_{\rm bio}\|_{\infty}=2.97\times10^{-5}.
\end{equation}
The full scaled gradient infinity norm is $2.97\times10^{-5}$. Thus the fixed centres remain stationary to the selected $10^{-4}$ spatial criterion after temporal refinement, while the optimized time variables satisfy the stricter $10^{-5}$ criterion.

\subsection{Final temporal-refinement verification}

The same fixed control~\eqref{eq:galicia-final-control} is evaluated, without further optimization, at $60$, $120$, and $240$ subdivisions per month. Table~\ref{tab:galicia-final-temporal-refinement} shows a highly regular first-order refinement pattern for this scalar diagnostic.

\begin{table}[htbp]
\centering
\small
\caption{Temporal refinement of the exact same final control on the L2 mesh.}
\label{tab:galicia-final-temporal-refinement}
\begin{tabular}{ccc}
\toprule
Subdivisions/month & $J_{\rm bio}(u^*)$ & Reduction $100(1-J_{\rm bio})$ \\
\midrule
60  & 0.995120457808 & $0.487954\%$ \\
120 & 0.995092856242 & $0.490714\%$ \\
240 & 0.995079099724 & $0.492090\%$ \\
\bottomrule
\end{tabular}
\end{table}

The increments are
\[
 J^{120}-J^{60}=-2.760157\times10^{-5},\qquad
 J^{240}-J^{120}=-1.375652\times10^{-5},
\]
so that
\begin{equation}\label{eq:galicia-observed-time-order}
 \frac{|J^{240}-J^{120}|}{|J^{120}-J^{60}|}=0.498396,
 \qquad
 p_{\rm obs}=\log_2\frac{|J^{120}-J^{60}|}{|J^{240}-J^{120}|}=1.0046.
\end{equation}
For this fixed control and scalar diagnostic, the observed refinement is consistent with the intended first-order temporal scheme. Three levels do not constitute a convergence theorem, but the near-halving is a useful discretization check.

At $240$ subdivisions per month, the comparison among controls is
\begin{table}[htbp]
\centering
\small
\caption{Final L2 comparison at $240$ subdivisions per month.}
\label{tab:galicia-final-control-comparison}
\begin{tabular}{lcc}
\toprule
Control & $J_{\rm bio}^{240}$ & Reduction \\
\midrule
D1 reference & 0.995229781691 & $0.477022\%$ \\
Time-only control & 0.995172933673 & $0.482707\%$ \\
Final timing--placement control & 0.995079099724 & $0.492090\%$ \\
\bottomrule
\end{tabular}
\end{table}
The hierarchy
\[
 J_{\rm joint}^{240}<J_{\rm time}^{240}<J_{D1}^{240}
\]
is preserved. Relative to the time-only control, the joint control increases the diagnostic reduction by approximately $1.94\%$; relative to D1, the increase is approximately $3.16\%$. Both coverage values remain above $0.99$, the centre separation remains $126.65$ km, and the penalties remain zero.

\subsection{Spatial distribution at the uncontrolled population maximum}\label{subsec:galicia-controlled-maps}

To visualize the optimized scenario independently of the objective aggregation, define the total density of all state variables active in a biological phase by the sum of the corresponding finite-element fields. At the selected time below the trajectory lies in phase $S$, where the stored active variables are $y_1$ and $y_3$; hence
\begin{equation}\label{eq:galicia-total-active-density}
 y_{\rm tot}(t,x)=y_1(t,x)+y_3(t,x)
 \qquad\text{during phase }S.
\end{equation}
The snapshot time is selected without reference to the controlled solution:
\begin{equation}\label{eq:galicia-snapshot-selection}
 t^*\in\operatorname*{arg\,max}_{s_n}
 \int_\Omega y_{{\rm tot},0}(s_n,x)\,\dd x,
\end{equation}
where the subscript $0$ denotes the uncontrolled trajectory. On the L2 mesh with $240$ subdivisions per month, the first maximizing node is
\[
 t^*=4.004166667\ \text{months},
 \qquad n^*=721,
 \qquad\text{phase }S.
\]
At this same node,
\begin{align*}
 \int_\Omega y_{{\rm tot},0}(t^*,x)\,\dd x
 &=2.205122314230\times10^8,\\
 \int_\Omega y_{{\rm tot},u^*}(t^*,x)\,\dd x
 &=2.193031220288\times10^8.
\end{align*}
The instantaneous total-mass reduction is therefore
\begin{equation}\label{eq:galicia-snapshot-reduction}
 100\frac{\int_\Omega(y_{{\rm tot},0}-y_{{\rm tot},u^*})(t^*,x)\,\dd x}
 {\int_\Omega y_{{\rm tot},0}(t^*,x)\,\dd x}
 =0.548319\%.
\end{equation}
This percentage is not $100(1-J_{\rm bio})$: the former concerns one phase-dependent total-density snapshot, whereas $J_{\rm bio}$ combines the prescribed primary and final diagnostic masses.

\begin{figure}[htbp]
\centering
\includegraphics[width=\textwidth]{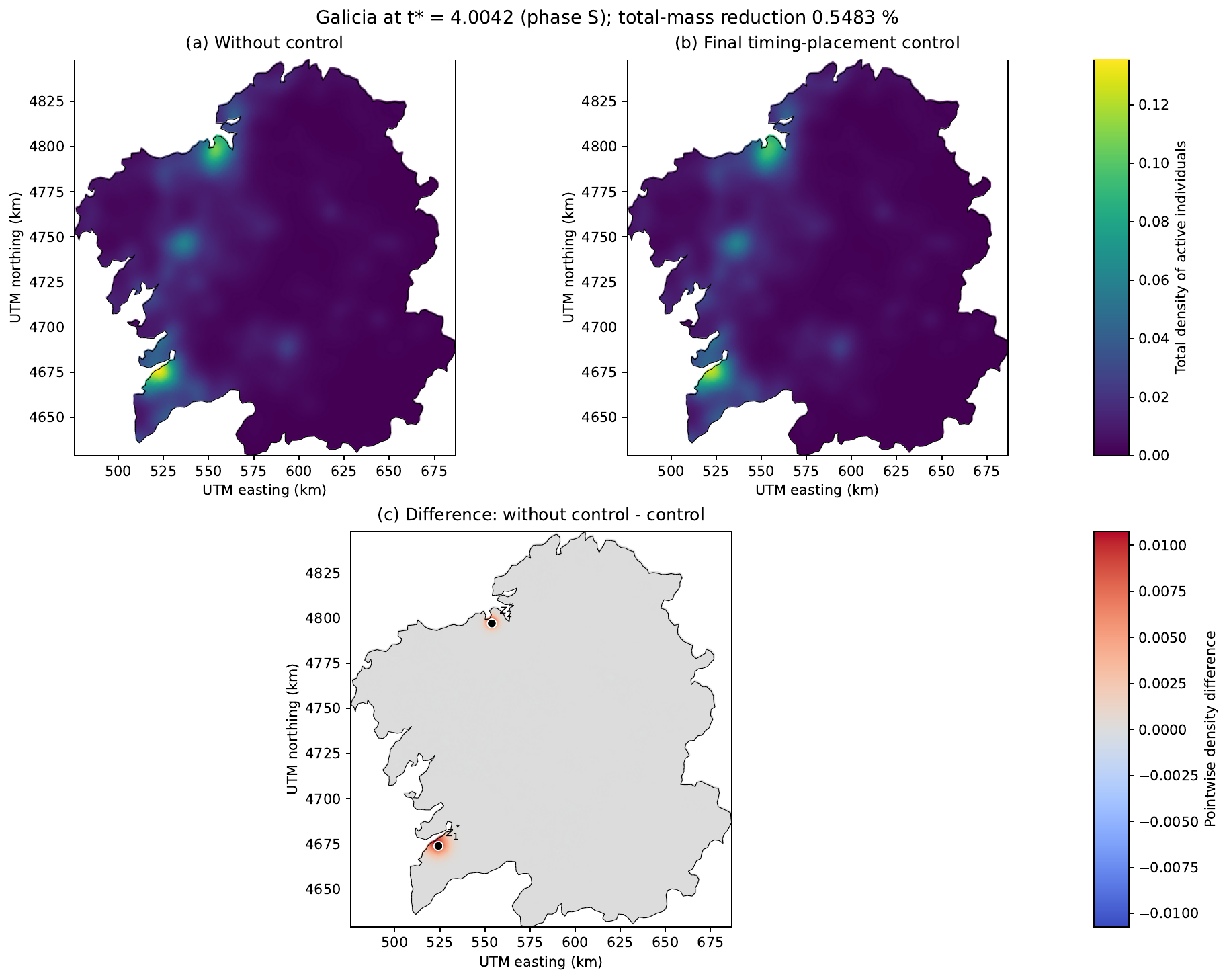}
\caption{Total density of all active individuals on the exact Galicia L2 mesh at the uncontrolled maximum $t^*=4.004166667$ months. Panels (a) and (b) use the same colour scale for the uncontrolled and controlled states. Panel (c) shows the pointwise difference without control minus controlled; positive values indicate a local reduction concentrated near the optimized trapping clusters. The black points shown in panel (c) mark the optimized cluster centres $z_1^*=(524176.365,4673868.807)$ m and $z_2^*=(553750.553,4797018.875)$ m in UTM coordinates.}
\label{fig:galicia-total-population-control-comparison}
\end{figure}

\begin{figure}[htbp]
\centering
\includegraphics[width=0.88\textwidth]{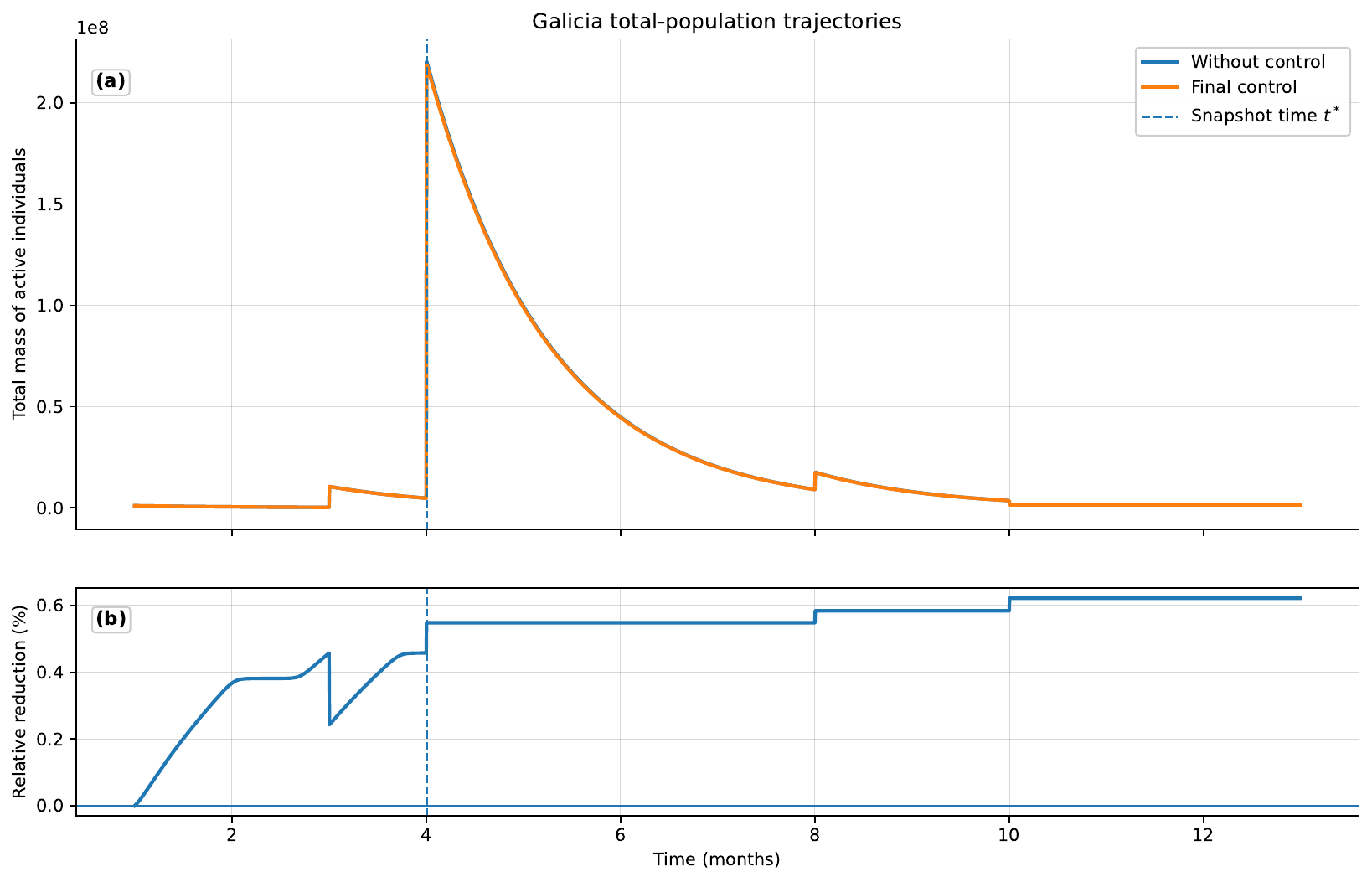}
\caption{Time trajectories of the total mass of the active compartments for the uncontrolled and final controlled states. Panel (a) shows the two mass curves, while panel (b) shows the relative reduction $100(M_0-M_u)/M_0$. The vertical dashed line identifies the snapshot time $t^*$ used in Figure~\ref{fig:galicia-total-population-control-comparison}. Discontinuous changes reflect the prescribed biological transfer events. The lower panel is included because the two mass curves are visually almost indistinguishable on their common scale.}
\label{fig:galicia-total-population-time-series}
\end{figure}

\subsection{Scope of the optimized control}

The calculations support a reproducible approximate box-stationary point of the stated penalized fully discrete scenario at 120 subdivisions per month and show that its objective value is stable under temporal refinement. They do not establish a global optimum, because convexity is not established and the field-domain search is deliberately local. Nor do they calibrate the conversion from trap count to mortality, validate $J_{\rm bio}$ as an economic or ecological management objective, or account for accessibility and deployment cost. Consequently, the absolute $0.49209\%$ reduction is conditional on the fixed intensity constants in Table~\ref{tab:galicia-control-production-constants} and must not be read as a field estimate of capture effectiveness. The principal numerical conclusion is therefore the reproducibility of the discrete-adjoint optimization and the stability of the reported diagnostic under the tested refinements, not a direct prescription of trap locations or dates.

\section{Discussion}

The academic optimization, the held-out state evaluation, and the Galicia cluster screening answer three different questions. The academic benchmark validates the discrete adjoint and the optimization workflow under controlled conditions. The held-out experiment asks whether the uncontrolled state transfers broad spatial information between annual record fields. The cluster screening asks whether a field-domain control discretization is numerically stable and which prescribed scenarios provide defensible initializations for the later optimizer. Keeping these roles separate avoids using spatial agreement to conceal numerical errors, using gradient verification as evidence of ecological validity, or interpreting a fixed-scenario comparison as an optimized management recommendation.

The multi-start experiments reveal a pronounced dependence on initialization. All 32 runs end with successful IPOPT termination, but their objective values span wide intervals for every $N_T$. This dispersion provides strong numerical evidence of a nonconvex reduced landscape, and a single terminated run is not representative. Several different starts nevertheless produce values within $5\%$ of the best result, giving some reproducibility to the lower envelope identified by the battery. The best controls favour earlier intervention than the original single-start solutions: every representative contains a trap at the lower admissible time and the remaining devices concentrate before or around $s_1$. Early removal reduces the foundress population before the first production event, and the total-mass trajectories show that this difference is propagated through the later impulsive transfers. The active lower bound should be varied in a dedicated sensitivity analysis before assigning biological significance to day 10.

The Galicia state experiment provides a second, more cautious result. Selecting the operator and parameters on 2023--2024 and freezing them for 2024--2025 yields a reproducible improvement of normalized $L^2$ error over persistence on both observational layers. The selected primary-nest operator also outperforms a matched pure-diffusion model. Since the pre-transfer foundress operator is nearly identical to diffusion after normalization, the improvement of the primary-nest operator is evidence that the calendar-dependent temporal--compartmental filter adds predictive value relative to the matched scalar-diffusion baseline. The evidence is modest rather than decisive: persistence still gives the largest correlation, centroid errors remain near $9$ km, and the model smooths away many local peaks.

The control screening provides numerical support for the field-domain numerical study. In the tested cases, mass lumping removes the negative undershoots observed with the consistent formulation, the $4$ km territorial radius is stable from L1 to L2, and the temporal refinement ratios are consistent with the intended first-order scheme. The ranking of the four fixed-budget layouts persists under mesh refinement and diffusion perturbation, with two hotspot clusters of ten physical traps giving the largest tested reduction. The high-resolution evaluations make it unlikely that this ranking is solely a time-step artefact within the tested range. The equal-duration tests also reproduce the academic preference for early intervention: early F and P are substantially more effective than late S or Q under the current coefficients.
The Galicia adjoint experiment closes the differentiation step without changing the specified numerical configuration. The complete reverse accumulation agrees with an independently implemented linearized equation at machine precision in both the nine-month and equal-effort scenarios, while centred finite differences show the expected convergence window. The moving-kernel derivatives include the normalization factor and therefore remain valid near the territorial boundary. The subsequent L1 and L2 optimization confirms that this validated gradient can be used in a realistic-domain continuation without loss of admissibility.

The mask audit clarifies an apparent anomaly without modifying the original biological model. Exact outcome coincidences arise because some masked compartments are absent or do not feed the two components of $J_{\rm bio}$ during the phases in which the masks differ. In particular, late worker removal can be nonzero while having no effect on primary-phase activity or end-of-cycle foundress mass. We regard this as a transparent structural implication of the current model. Adding a worker-to-reproduction coupling merely to force a control effect would introduce an uncalibrated mechanism. The appropriate response in the present optimizer is to omit ineffective control channels and to treat any enriched demographic coupling as a separate, data-supported model extension.

The contrast between $L^2$ error and correlation in the state experiment remains informative. A smoother field can reduce the integrated quadratic discrepancy by suppressing highly localized peaks, yet lose the relative ranking of local intensities. Reporting only one metric would therefore give an incomplete assessment. The persistence and matched-diffusion baselines are equally important: the first prevents a mechanistic model from receiving credit for information already contained in the previous year's map, while the second isolates the contribution of the compartment transfers from ordinary dispersal.

The selected state parameter pair is not a precise biological estimate. The calibration surface is shallow and the selected diffusion lies on the upper boundary of the tested grid. Radius and diffusion both control effective smoothing and are partially confounded. The reported pair should be regarded as a representative of a region of comparable spatial scales. A later study should quantify profile likelihoods or Bayesian posterior uncertainty, vary the fixed observation bandwidth, and include environmental heterogeneity or anisotropic transport if the purpose is parameter inference rather than predictive comparison.

The annual record totals cannot currently validate abundance. The linear normalization and scalar count calibration lead to identity~\eqref{eq:galicia-count-identity}, while the administrative layers display very different year-to-year count changes. An explicit observation model is required to separate population abundance from search effort, reporting, classification, and intervention practice. Until then, the normalized-field analysis is more defensible than the total-count comparison.

The spatial postprocessing provides a complementary interpretation of the optimized trajectory. At the uncontrolled maximum of the phase-dependent total active population, the controlled field has $0.5483\%$ less integrated mass. The map difference is spatially heterogeneous and is concentrated around, but not restricted to, the two optimized intervention regions because diffusion and the subsequent biological transfers propagate earlier local removals. This snapshot is descriptive and must not be confused with the aggregate diagnostic $J_{\rm bio}$ used by the optimizer.

Important control limitations remain. The conversion from a physical trap count to mortality intensity is not field calibrated; the diagnostic~\eqref{eq:galicia-control-Jbio} has no economic weight; equal physical budgets do not imply equal integrated exposure; and the local optimizer does not certify global optimality. Nevertheless, the numerical chain is now closed: two L1 trajectories terminate at nearly identical controls, the L2 directional derivative is independently verified, the L2 continuation reaches a scaled residual below $5\times10^{-6}$, the $120$-subdivision control has a full scaled gradient below $3\times10^{-5}$, and the same fixed control displays a first-order temporal refinement pattern through $240$ subdivisions per month. Accessibility, integer allocation, deployment costs, and uncertainty in the trap-to-mortality conversion remain subjects for a future management-oriented model rather than prerequisites for the numerical conclusions of the present paper.

\section{Conclusions}

We have developed a finite-dimensional timing--placement control problem for a Stieltjes-time reaction--diffusion model of \textit{Vespa velutina}. The continuous formulation identifies how the intrinsic clocks, flat intervals, and biological atoms enter the weak model, while the computational analysis is deliberately rigorous at the fully discrete level. The finite-element/Stieltjes--Euler state equation is expressed as an algebraic residual, which yields existence of a minimizer for the box-constrained discrete problem, necessary first-order conditions, linearized-state equations, and an exact discrete adjoint.

In the academic benchmark, adjoint and direct sensitivities agree to machine precision. The mean adjoint speed-up increases from approximately $7.73$ for one trap to $30.72$ for four traps. Thirty-two successfully terminated IPOPT runs produce a broad dispersion of locally converged values consistent with nonconvex behaviour: the best values found reduce the academic objective by approximately $46.6\%$, $65.4\%$, $77.0\%$, and $82.6\%$ for one to four fixed-intensity devices, but the dispersion across starts precludes a claim of global optimality.

The Galicia state experiment introduces explicit observation operators and a frozen temporal evaluation. The selected primary-nest operator improves normalized $L^2$ discrepancy over persistence by $8.1\%$ on the primary layer and $10.7\%$ on the secondary layer, and modestly improves on matched diffusion. Persistence nevertheless retains the highest correlation, centroid discrepancies remain close to $9$ km, and the available annual records do not independently identify abundance.

The Galicia control-screening study preserves the original biological transfers. Its 59-case core battery and 28 directed tests support mass lumping, a $4$ km effective cluster radius, two hotspot clusters of ten traps, L1 with $30$ subdivisions per month for exploration, and L2 for verification. The exact adjoint of the specified lumped and saturating residual is validated through local derivative tests, an independent calendar transpose, complete linearized equations, and centred differences.

The two-cluster optimization is completed on the stated discretization hierarchy. Two independent L1 timing--placement trajectories terminate at controls with objective gap $2.84\times10^{-9}$ and sub-metre differences in their centres. Re-evaluation on L2 exposes a genuine discretization-level residual, independently confirmed by a directional finite-difference test. L2 continuation reduces the dimensionless projected residual to $4.84\times10^{-6}$. A short $120$-subdivision time continuation gives the final control~\eqref{eq:galicia-final-control}, with temporal residual $5.91\times10^{-6}$ and full scaled gradient infinity norm $2.97\times10^{-5}$.

For exactly this fixed control, the objective values at $60$, $120$, and $240$ subdivisions per month are $0.995120458$, $0.995092856$, and $0.995079100$. The refinement-increment ratio is $0.4984$ and the observed order is $1.0046$. At the finest level the final joint control preserves zero penalties and gives a $0.49209\%$ diagnostic reduction, compared with $0.48271\%$ for the time-only control and $0.47702\%$ for D1. These results close the planned numerical validation campaign.

At the uncontrolled maximum of the total active population, $t^*=4.004166667$ months, the final controlled trajectory reduces the integrated active mass from $2.205122314230\times10^8$ to $2.193031220288\times10^8$, an instantaneous reduction of $0.548319\%$. The corresponding exact-mesh maps display the uncontrolled field, the controlled field, and their pointwise difference, while explicitly marking the two optimized centres. This spatial visualization is consistent with the objective-level ordering but is not itself an additional optimization criterion.

The final conclusion is intentionally limited. The computations support an approximate box-stationary point at 120 subdivisions per month and an observed first-order refinement pattern for the objective of the same fixed control through 240 subdivisions per month. They do not certify a global optimum of the continuous problem, calibrate real trap effectiveness, or provide a direct management prescription. Those questions require field-calibrated trap mortality, accessibility and cost data, and a formal uncertainty analysis.

\section*{Data and code availability}
The complete reproducibility archive associated with this study is publicly available in Zenodo as version~1.1~\cite{FernandezAreaZenodo2026}. It contains the FreeFEM++ and Python source code, processed simulation inputs, frozen reference outputs, environment and execution instructions, SHA-256 integrity manifests, and a detailed correspondence between each numerical experiment and the sections, figures, and tables of the article. The raw administrative nest records are not redistributed because they remain subject to the access conditions of the original data provider.

\section*{CRediT authorship contribution statement}
Francisco J. Fern\'andez: Conceptualization, Formal analysis, Software, Investigation, Validation, Visualization, Writing -- original draft, Writing -- review and editing. Iv\'an Area: Conceptualization, Methodology, Mathematical modelling, Validation, Writing -- review and editing.

\section*{Declaration of competing interest}
The authors declare that they have no known competing financial interests or personal relationships that could have appeared to influence the work reported in this paper.

\section*{Declaration of generative AI and AI-assisted technologies in the writing process}
During the preparation of this work, the authors used OpenAI ChatGPT to support language editing, document restructuring, LaTeX preparation and consistency checking. After using this tool, the authors reviewed and edited the content and take full responsibility for the content of the publication.

\appendix

\section{Derivatives of the trap kernels}\label{app:kernel-derivatives}

For the academic temporal bump, write $r=(t-\tau_k)/T_M$ and
\[
 \delta_{T_M}^{1D}(t-\tau_k)
 =\frac{1}{C_1T_M}e^{1/(r^2-1)}\one_{\{|r|<1\}}.
\]
Inside the support, differentiation with respect to the centre gives
\begin{equation}\label{eq:academic-temporal-bump-explicit-derivative}
 \frac{\partial}{\partial\tau_k}\delta_{T_M}^{1D}(t-\tau_k)
 =\delta_{T_M}^{1D}(t-\tau_k)
 \frac{2r}{T_M(r^2-1)^2}.
\end{equation}
The compactly supported exponential extends with all derivatives equal to zero at $|r|=1$, so~\eqref{eq:academic-temporal-bump-explicit-derivative} defines the global derivative by zero outside the support.

For the radial spatial bump, let
\[
 \rho_k^2(x)=\frac{|x-z_k|^2}{R_M^2},
 \qquad
 \delta_{R_M}^{2D}(x-z_k)
 =\frac{1}{C_2R_M^2}e^{1/(\rho_k^2-1)}\one_{\{\rho_k<1\}}.
\]
Inside the support,
\begin{align}
 \frac{\partial}{\partial x_k}\delta_{R_M}^{2D}(x-z_k)
 &=\delta_{R_M}^{2D}(x-z_k)
 \frac{2(x-x_k)}{R_M^2(1-\rho_k^2)^2},\label{eq:academic-spatial-bump-x-derivative}\\
 \frac{\partial}{\partial y_k}\delta_{R_M}^{2D}(x-z_k)
 &=\delta_{R_M}^{2D}(x-z_k)
 \frac{2(y-y_k)}{R_M^2(1-\rho_k^2)^2}.
 \label{eq:academic-spatial-bump-y-derivative}
\end{align}
The positive sign is correct because the derivative is taken with respect to the centre rather than the spatial evaluation point. Combining~\eqref{eq:academic-temporal-bump-explicit-derivative}--\eqref{eq:academic-spatial-bump-y-derivative} with~\eqref{eq:trap-control-continuous} yields the academic control derivatives in~\eqref{eq:kernel-derivative-tau}--\eqref{eq:kernel-derivative-y}.

For the Galicia computation, the raw kernel is $K_R$ from~\eqref{eq:galicia-cluster-kernel}. Inside its support, with $\rho^2=|x-z|^2/R^2$,
\begin{align}
 \partial_{z_1}K_R(x-z)
 &=K_R(x-z)\frac{2(x_1-z_1)}{R^2(1-\rho^2)^2},\\
 \partial_{z_2}K_R(x-z)
 &=K_R(x-z)\frac{2(x_2-z_2)}{R^2(1-\rho^2)^2}.
\end{align}
The normalized derivatives are not obtained by merely rescaling these expressions: the derivative of the domain integral must also be included. The exact quotient formula is~\eqref{eq:galicia-normalized-kernel-derivative}, and the nested saturation derivative is~\eqref{eq:galicia-mortality-chain-rule}. These are the formulas validated in the Galicia local and complete directional tests.

\section{Implementation checklist}

The following checklist summarizes the status of the numerical implementation after the academic validation tests:
\begin{enumerate}[label=(I\arabic*)]
    \item the implemented adjoint has been checked against the implemented linearized state operator through directional and componentwise gradient tests;
    \item the averaged biological transitions and their transpose factors are included in the academic benchmark and are part of the adjoint-gradient validation;
    \item the discrete quadratic cost is used consistently as the finite-dimensional objective optimized by the code;
    \item the finite-difference test and the timing battery are documented in Section~\ref{sec:academic-validation};
    \item the eight-start IPOPT battery for each \(N_T=1,2,3,4\), including the ten completed restarts, is reported in Section~\ref{sec:numerical-optimization};
    \item the driver exports the physical-coordinate projected-box diagnostic and its normalized form for every multi-start run;
    \item the academic numerical section includes the full multi-start dispersion, the best controls, and independently recomputed total-mass trajectories;
    \item the Galicia experiment fixes the observed kernel bandwidth, selects operator and parameters on 2023--2024 only, and evaluates the frozen choice on 2024--2025 against persistence and matched diffusion;
    \item annual count predictions are reported only as a structural diagnostic because the linear normalization makes them non-identifying for the spatial PDE;
    \item the 59-case Galicia control core battery passes zero-control consistency and co-located partition-invariance checks;
    \item mass lumping is used for the Galicia control stage because the consistent formulation produces measurable negative undershoots, while the lumped formulation is nonnegative up to floating-point noise;
    \item the $4$ km territorial cluster radius is accepted for L1 exploration, whereas the $2$ km radius requires L2 resolution;
    \item first-order temporal refinement behaviour is observed in selected L2 runs up to 120 subdivisions per month;
    \item the mask audit explains exact outcome coincidences without changing the original biological transfer model;
    \item the Galicia discrete adjoint, moving normalized-kernel derivatives, calendar transpose, and independent linearized directional equation have been validated before IPOPT;
    \item two independent L1 timing--placement trajectories terminate at numerically indistinguishable controls in the same apparent basin, and the L2 continuation reduces the dimensionless projected residual below $5\times10^{-6}$;
    \item the final control is approximately box-stationary at 120 subdivisions per month and exhibits an observed first-order refinement pattern, for the scalar objective, through 240 subdivisions per month;
    \item the final postprocessing compares uncontrolled and controlled total active-density fields at an objectively selected uncontrolled maximum, using a common state scale, a separate difference scale, and explicit black markers for both optimized centres;
    \item the Galicia optimizer uses L1 for exploration and L2 for continuation, with two ten-trap clusters, one-month durations, scaled local centre offsets, and separately reported coverage and separation safeguards;
    \item L2 with 60 subdivisions per month is reserved for verification of the best L1 multi-start controls, with selected 120-subdivision confirmation runs.
\end{enumerate}

\end{document}